\documentclass[a4paper, 12pt, twoside]{article}

\usepackage[top=3cm, bottom=3cm, left=3cm, right=3cm]{geometry}
\usepackage{tikz}
\usetikzlibrary{shapes, arrows, arrows.meta}

\usepackage{amsmath}
\usepackage{amsthm}
\usepackage{amssymb}
\usepackage{mathtools}
\usepackage[integrals]{wasysym}
\usepackage{stmaryrd}
\usepackage{bbm}
\usepackage[mathcal]{euscript}

\usepackage[utf8]{inputenc}
\usepackage[T1]{fontenc}

\usepackage[english]{babel}

\usepackage{cite}

\usepackage{lmodern}

\usepackage{calc}
\usepackage[inline]{enumitem}
\usepackage[normalem]{ulem}
\usepackage{url}

\usepackage[hidelinks, bookmarks, bookmarksnumbered, pdfstartview={XYZ null null 1.00}]{hyperref}

\theoremstyle{plain}
\newtheorem{theorem}{Theorem}[section]
\newtheorem{lemma}[theorem]{Lemma}
\newtheorem{corollary}[theorem]{Corollary}
\newtheorem{proposition}[theorem]{Proposition}
\theoremstyle{definition}
\newtheorem{definition}[theorem]{Definition}
\newtheorem{hypothesis}[theorem]{Hypotheses}
\newtheorem{remark}[theorem]{Remark}
\newtheorem{example}[theorem]{Example}

\DeclareMathOperator{\sign}{sign}
\DeclareMathOperator{\Real}{Re}
\DeclareMathOperator{\Imag}{Im}
\DeclareMathOperator{\diff}{d\!}

\DeclarePairedDelimiter{\norm}{\lVert}{\rVert}
\DeclarePairedDelimiter{\abs}{\lvert}{\rvert}
\DeclarePairedDelimiter{\floor}{\lfloor}{\rfloor}
\DeclarePairedDelimiter{\ceil}{\lceil}{\rceil}

\newcommand{\Id}{\mathrm{Id}}
\newcommand{\id}{\mathrm{id}}

\newcommand{\suchthat}{\ifnum\currentgrouptype=16 \mathrel{}\middle|\mathrel{}\else\mid\fi}

\renewcommand{\leq}{\leqslant}

\renewcommand{\geq}{\geqslant}

\numberwithin{table}{section}
\numberwithin{figure}{section}
\numberwithin{equation}{section}

\begin{document}

\setlength{\parskip}{1pt plus 1pt minus 1pt}

\setlist[enumerate, 1]{label={\textnormal{(\alph*)}}, ref={(\alph*)}, leftmargin=0pt, itemindent=*}
\newlist{hypolist}{enumerate}{1}
\setlist[hypolist]{label={\textnormal{(H\arabic*)}}, leftmargin=0pt, itemindent=*, resume=hypolist, widest*=10}

\title{Well-posedness and asymptotic behavior of difference equations with a time-dependent delay and applications\thanks{This work was partially supported by France 2030 funding ANR-11-IDEX-0003 H-CODE.}}
\author{Guilherme Mazanti\thanks{Université Paris-Saclay, CNRS, CentraleSupélec, Inria, Laboratoire des signaux et systèmes, 91190 Gif-sur-Yvette, France.} \thanks{Fédération de Mathématiques de CentraleSupélec, 91190, Gif-sur-Yvette, France.} \and Jaqueline G. Mesquita\thanks{Universidade de Brasília, Departamento de Matemática, Campus Universitário Darcy Ribeiro, Asa Norte, Brasília-DF 70910-900, Brazil.}}
\date{\today}

\maketitle

\begin{abstract}
In this paper, we investigate the well-posedness and asymptotic behavior of difference equations of the form
\begin{equation*}
x(t) = A x(t - \tau(t)), \qquad t \geq 0,
\end{equation*}
where the unknown function $x$ takes values in $\mathbb R^d$ for some positive integer $d$, $A$ is a $d \times d$ matrix with real coefficients, and $\tau\colon [0, +\infty) \to (0, +\infty)$ is a time-dependent delay. 

We provide our investigations for three spaces of functions: continuous, regulated, and $L^p$. We compare our results for these three cases, showing how the hypotheses change according to the space that we are treating. Finally, we provide applications of our results to difference equations with state-dependent delays for the cases of continuous and regulated function spaces, as well as to transport equations in one space dimension with time-dependent velocity.

\bigskip


\smallskip

\end{abstract}

\tableofcontents

\section{Introduction}

Difference equations, sometimes also known as continuous-time difference equations, difference delay equations, or delay difference equations, are equations in which the value of an $\mathbb R^d$-valued unknown function $x$ at time $t$ is expressed in terms of the values of $x$ at previous times. In the linear time-invariant case with finitely many constant delays, a difference equation can be written under the form
\begin{equation}
\label{eq:intro-general-linear-difference}
x(t) = \sum_{j=1}^N A_j x(t - \tau_j),
\end{equation}
where $A_1, \dotsc, A_N$ are $d \times d$ matrices with real coefficients and $\tau_1, \dotsc, \tau_N$ are positive delays. The study of \eqref{eq:intro-general-linear-difference} has been motivated in the literature both from the fact that its properties, such as stability, allow one to obtain information on corresponding properties for neutral time-delay systems of the form
\[
\frac{\diff}{\diff t} \left[x(t) - \sum_{j=1}^N A_j x(t - \tau_j)\right] = L x_t
\]
for some linear operator $L$ (see, e.g., \cite{Henry1974Linear, Hale2002Strong} and \cite[Chapter~9]{Hale1993Introduction}), and also from the fact that some hyperbolic partial differential equations on space dimension one can be put into the form \eqref{eq:intro-general-linear-difference} by using the method of characteristics, a fact widely used in the literature (see, e.g., \cite{Bastin2016Stability, Cooke1968Differential, Slemrod1971Nonexistence, Brayton1967Nonlinear, Coron2015Dissipative, Chitour2016Stability}).

Their links with neutral time-delay systems and hyperbolic PDEs have motivated many works on the difference equation \eqref{eq:intro-general-linear-difference} along the years, such as \cite{Melvin1974Stability, Avellar1980Zeros, Cruz1970Stability, Henry1974Linear, Hale2002Strong, Michiels2009Strong, Chitour2016Stability, Avellar1990Difference, Mazanti2017Relative, Chitour2020Approximate, Chitour2023Hautus, Hale2003Stability, Silkowski1976Star}. In particular, it has been noticed that stability properties of \eqref{eq:intro-general-linear-difference} are not continuous with respect to the delays $\tau_1, \dotsc, \tau_N$, in the sense that \eqref{eq:intro-general-linear-difference} can be exponentially stable for a certain choice of delays $(\tau_1, \dotsc, \tau_N)$, but unstable for delays $(\tau_1^\prime, \dotsc, \tau_N^\prime)$ arbitrarily close to $(\tau_1, \dotsc, \tau_N)$ (see, e.g., \cite[Section~9.6]{Hale1993Introduction}, \cite[Example~1.37]{Mazanti2016Stability}, \cite[Example~1.36]{Michiels2014Stability}). For this reason, some works have considered more robust stability properties, and the main result in this sense is the following one (see, e.g., \cite[Theorem~5.2]{Avellar1980Zeros}, \cite[Chapter~9, Theorem~6.1]{Hale1993Introduction}, and \cite{Silkowski1976Star}), known as the Hale--Silkowski criterion.

\begin{proposition}
\label{prop:Hale-Silkowski}
Consider the difference equation \eqref{eq:intro-general-linear-difference} and define
\begin{equation}
\label{eq:rho-HS}
\rho_{\mathrm{HS}} = \max_{(\theta_1, \dotsc, \theta_N) \in [0, 2\pi]^N} \rho\left(\sum_{j=1}^N A_j e^{i \theta_j}\right),
\end{equation}
where $\rho(M)$ denotes the spectral radius of the matrix $M$. The following assertions are equivalent.
\begin{enumerate}
\item We have $\rho_{\mathrm{HS}} < 1$.
\item There exists a rationally independent family of positive delays $\tau_1, \dotsc, \tau_N$ such that \eqref{eq:intro-general-linear-difference} is exponentially stable.
\item\label{item:HS-local} There exist positive delays $\tau_1, \dotsc, \tau_N$ and $\varepsilon > 0$ such that \eqref{eq:intro-general-linear-difference} is exponentially stable for any choice of positive delays $\tau_1^\prime, \dotsc, \tau_N^\prime$ with $\abs{\tau_j^\prime - \tau_j} < \varepsilon$ for every $j \in \{1, \dotsc, N\}$.
\item \eqref{eq:intro-general-linear-difference} is exponentially stable for every choice of positive delays $\tau_1, \dotsc, \tau_N$.
\end{enumerate}
\end{proposition}

The most striking fact about the above result is that exponential stability of \eqref{eq:intro-general-linear-difference} for \emph{some} rationally independent family of positive delays $\tau_1, \dotsc, \tau_N$ is equivalent to exponential stability of \eqref{eq:intro-general-linear-difference} for \emph{any} choice of positive delays $\tau_1, \dotsc, \tau_N$ (with no constraints on their rational independence), and that this can be characterized by $\rho_{\mathrm{HS}} < 1$, where the quantity $\rho_{\mathrm{HS}}$ is independent of the delays. This result can also be extended to cases in which the delays satisfy some rational dependence structure, as detailed in \cite{Michiels2009Strong, Chitour2016Stability}, and to time-varying matrices $A_1, \dotsc, A_N$, as detailed in \cite{Chitour2016Stability}.

Since the condition $\rho_{\mathrm{HS}} < 1$ ensures exponential stability of \eqref{eq:intro-general-linear-difference} for every choice of (constant) positive delays $\tau_1, \dotsc, \tau_N$, a natural question is whether the same condition also ensures stability of the difference equation with time-varying delays
\begin{equation}
\label{eq:intro-linear-difference-varying}
x(t) = \sum_{j=1}^N A_j x(t - \tau_j(t)).
\end{equation}
The answer to this question is negative, as illustrated by a counterexample in \cite[Section~4]{Avellar1990Difference}. The article \cite{Avellar1990Difference} also provides sufficient conditions for the stability of \eqref{eq:intro-linear-difference-varying}.

Time-delay systems with time-varying delays have been extensively studied in the literature \cite{Fridman2008Robust, Fridman2007Stability, Shustin2007Delay, Louisell2001Delay, Michiels2005Stabilization, Verriest2012State, Bonnet2020L2, Mazanti2014Stabilization}, in particular through the use of Lyapunov and Lyapunov--Krasovski\u{\i} functions, which commonly provide sufficient conditions for stability, often under the form of linear matrix inequalities. In addition, many works also require conditions on the rate of variation of the delay, such as assuming that a given time-dependent delay $\tau(t)$ satisfies $\dot\tau(t) < 1$ for every $t$, or $\dot\tau(t) \leq c$ for some constant $c < 1$ and every $t$, the main interest of these assumptions being to ensure that $t \mapsto t - \tau(t)$ is increasing. Other strategies for dealing with time-dependent delays include perturbation approaches, used, for instance, in \cite{Bonnet2020L2, Mazanti2014Stabilization}, which consist in establishing first a stability or stabilization result for constant nominal values of the delays, and then proving that such stability properties are preserved if the delays are time-varying and take values in small intervals containing the nominal values.

The goal of this work is to provide more precise results on the stability of \eqref{eq:intro-linear-difference-varying} in the case $N = 1$, that is,
\begin{equation}
\label{eq:diff-eqn-N1-time}
x(t) = A x(t - \tau(t)), \qquad t \geq 0.
\end{equation}
We are interested in studying the well-posedness of \eqref{eq:diff-eqn-N1-time} and obtaining sufficient and, as much as possible, also necessary conditions for its convergence to zero or exponential stability, in three different spaces of functions: continuous functions, regulated functions, and $L^p$ functions, showing how the conditions for well-posedness or stability change in each case. Note that, with respect to \cite{Avellar1990Difference}, even though we consider here the particular case $N = 1$ of \eqref{eq:intro-linear-difference-varying}, we consider continuous, regulated, and $L^p$ solutions, while \cite{Avellar1990Difference} considers only continuous solutions. In addition, \cite{Avellar1990Difference} only provides sufficient conditions for asymptotic stability, while we also consider exponential stability, and we obtain, in some results (see, e.g., Proposition~\ref{prop:necessary-asympt}), that some of our conditions are also necessary \cite{Avellar1990Difference}.

To get the main results of this paper for the case involving regulated functions, we introduce a new class of functions called \emph{well-regulated functions}, that are regulated functions with an additional condition preventing some forms of fast oscillations (see Definition~\ref{def:regulated}\ref{item:well-regulated} below for details). The most important property concerning well-regulated functions that we prove can be found in Theorem~\ref{thm:well-regulated}, which ensures that this class of functions preserves regulatedness by right composition. This property is the key to prove our main results for the equation \eqref{eq:diff-eqn-N1-time}, and it opens several paths for future investigations of regulated functions. Indeed, one of the main challenges when one deals with equations with state-dependent delays in the space of regulated functions is to ensure that the composition which appears in the equation is well-defined, since, due to the state-dependent term, it appears a composition of two regulated functions, and it is a known fact that the regulated functions are not preserved, in general, by right-composition (see also Remark~\ref{remk:composition-regulated} below).

As for the results of the paper involving $L^p$ functions, Section~\ref{sec:Lp} investigates the right composition of a function in $L^p$, describing a class of functions preserving $L^p$ integrability by right composition (Theorem~\ref{thm:composition-Lp}). For that, it is necessary first to investigate the $(\mathfrak L, \mathfrak B)$-measurability of the composition of two $L^p$ functions, which is the subject of Proposition~\ref{prop:preserves-measurability}. 

With these preliminary results in hands, we introduce and explain all the hypotheses and definitions that will be necessary for providing the main results in Section~\ref{sec:Hypo-def}.

Section~\ref{sec:well-posed} is devoted to investigating the well-posedness of the equation \eqref{eq:diff-eqn-N1-time}. We start, in Definition~\ref{def:solutions}, by providing some different notions of solutions: in addition to continuous, regulated, and $L^p$ solutions, we also introduce other auxiliary notions of solutions that are useful for the sequel. Our main strategy for studying both well-posedness and the asymptotic behavior of solution relies on an explicit formula for solutions, which is developed in Section~\ref{sec:explicit-formula} (see Theorems~\ref{theorem-1} and \ref{thm:explicit-ae}), and this explicit formula also yields uniqueness of solutions (see Corollary~\ref{coro:uniqueness}). Existence of solutions is the topic of Section~\ref{sec:exist}. Note that \eqref{eq:diff-eqn-N1-time} can be rewritten as $x(t) = A x(\sigma_1(t))$, where $\sigma_1(t) = t - \tau(t)$, and hence one of the main questions for the existence of solutions $x$ with a given regularity is whether right composition by $\sigma_1$ preserves such regularity. Our assumptions are made in such a way as to ensure that this is indeed the case, and we illustrate the importance of each assumption through comprehensive examples, showing also that some of these assumptions turn out to be necessary for existence of solutions (see Proposition~\ref{prop:h2-necessary}).

Section~\ref{sec:stability} is divided in two parts. The first one is devoted to investigating convergence to the origin of the solutions of \eqref{eq:diff-eqn-N1-time}. Theorem~\ref{thm:cv-to-0} provides sufficient conditions ensuring that solutions with bounded initial condition converge pointwise to $0$ in the topology of $\mathbb R^d$. As an immediate consequence, considering continuous solutions for the problem \eqref{eq:diff-eqn-N1-time}, it is possible to conclude that this convergence also holds in the norm of $\mathsf{C}^b_t$ (Corollary~\ref{Corol-1-C}). Analogously, taking regulated solutions for the problem \eqref{eq:diff-eqn-N1-time}, the convergence can be extended in the norm of $\mathsf{G}^b_t$ (Corollary~\ref{Prop-1-C}). The same investigation can be made for solutions in $L^\infty$, taking into account that the solutions in $L^\infty$ satisfies \eqref{eq:diff-eqn-N1-time} only for a.e.\ $t \in \mathbb R_+$ (Corollary~\ref{coro:cv-to-0-Linfty}). We point out that Corollary~\ref{coro:cv-to-0-Linfty} on asymptotic behavior of the solutions is no longer valid if $L^\infty$ is replaced with $L^p$ for some $p \in [1, +\infty)$ (Example~\ref{example-3.1}). On the other hand, it is possible to provide a result with changes in the assumptions, in order to ensure convergence to the origin of the solutions of \eqref{eq:diff-eqn-N1-time} in $L^p$ with $p \in [1, +\infty)$ (Theorem~\ref{Thm-1-L-P}). The second part of this section is devoted to exponential stability of \eqref{eq:diff-eqn-N1-time}. In the same way as in the preceding subsection, we start by investigating sufficient conditions to ensure exponential convergence to $0$ of $\abs{x(t)}$ for solutions with bounded initial conditions (Theorem~\ref{thm:exp-to-0}). We point out that the conditions concerning the delay function play an important role to guarantee the proof of Theorem~\ref{thm:exp-to-0}. We also discuss about these conditions, showing that they are sufficient, but not necessary. As a consequence of Theorem~\ref{thm:exp-to-0}, we provide a criterion for the exponential convergence to $0$ of continuous solutions of \eqref{eq:diff-eqn-N1-time} (Corollary~\ref{coro:exp-cont}). For this case, we also discuss about the necessity of the assumptions which appears in Corollary~\ref{coro:exp-cont}. Also, Corollaries~\ref{coro:exp-G} and \ref{coro:exp-Linfty} follow as a consequence of Theorem~\ref{thm:exp-to-0} for regulated solutions and solutions in $L^\infty$, respectively. In the same way as before, we show by a counterexample that Corollary~\ref{coro:exp-Linfty} does not remain true in $L^p$ for $p\in [1, +\infty)$. Finally, Theorem~\ref{thm:exp-Lp} provides sufficient conditions for ensuring exponential stability in the case of $L^p$ for $p\in [1, +\infty)$.

Section~\ref{sec:6} is devoted to applications of the results presented in the previous sections. In the first part of Section~\ref{sec:6}, we apply the results to an equation with state-dependent delays. Theorem~\ref{thm:state-dep} provides sufficient conditions for the exponential convergence to zero of solutions, when the solution is continuous as well as regulated.
The second part (Section~\ref{sec:6.2}) applies the results to a transport equation in a bounded interval one space dimension with varying transport speed. Using the method of characteristics, this transport equation is transformed into a difference equation of the form \eqref{eq:diff-eqn-N1-time} (Proposition~\ref{prop:transport}) and hence, as a consequence, one obtains sufficient conditions for the exponential stability of solutions of the transport equation in $L^p$ for $p \in [1, +\infty]$ (Corollary~\ref{coro:appli-transport}).

\medskip

\noindent\textbf{Notation.} In the paper, $\mathbb N$ and $\mathbb N^\ast$ denote the set of nonnegative and positive integers, respectively, while $\mathbb R_+, \mathbb R_-, \mathbb R_+^\ast, \mathbb R_-^\ast$ denote the sets of nonnegative, nonpositive, positive, and negative real numbers, respectively. For convenience, for two real numbers $a, b$ with $a \leq b$, we denote by $\llbracket a, b\rrbracket$ the set of integers between $a$ and $b$, i.e., $\llbracket a, b\rrbracket = [a, b] \cap \mathbb Z$. For $a$ and $b$ in $\mathbb R \cup \{-\infty, +\infty\}$ with $a \leq b$, we set $\lvert a, b]$ as the interval $[a, b]$ if $a > -\infty$ and $(a, b]$ otherwise, and the interval $\lvert a, b)$ is defined in a similar manner. The real and imaginary parts of a complex number $z$ are denoted by $\Real z$ and $\Imag z$, respectively.

Given a square matrix $A$, its spectrum and its spectral radius are denoted by $\sigma(A)$ and $\rho(A)$, respectively. The set of $d \times d$ matrices with real coefficients is denoted by $\mathcal M_d(\mathbb R)$ and, given a norm $\abs{\cdot}$ in $\mathbb R^d$, we use the same notation $\abs{\cdot}$ for the corresponding induced matrix norm in $\mathcal M_d(\mathbb R)$.

Given two topological spaces $A, B$, we use both $C(A, B)$ and $C^0(A, B)$ to denote the set of continuous functions defined in $A$ and taking values in $B$. When $A$ is an open subset of $\mathbb R^{d_1}$ and $B \subset \mathbb R^{d_2}$ for some positive integers $d_1, d_2$, for $k \in \mathbb N^\ast \cup \{+\infty\}$, we use $C^k(A, B)$ to denote the set of $k$ times continuously differentiable functions defined in $A$ and taking values in $B$. For a function $f$ defined in a nonempty interval $I \subset \mathbb R$, we denote its lateral limits by $f(t_0^-)$ and $f(t_0^+)$, defined, respectively, by $f(t_0^-) = \lim_{t \to t_0^-} f(t)$ and $f(t_0^+) = \lim_{t \to t_0^+} f(t)$, when these limits make sense.

Given measurable spaces $(X, \mathfrak S_X)$ and $(Y, \mathfrak S_Y)$, we recall that a function $f\colon X \to Y$ is said to be $(\mathfrak S_X, \mathfrak S_Y)$-measurable if $f^{-1}(A) \in \mathfrak S_X$ for every $A \in \mathfrak S_Y$. If $f$ is $(\mathfrak S_X, \mathfrak S_Y)$-measurable and $\mu$ is a measure in $(X, \mathfrak S_X)$, we denote by $f_{\#} \mu$ the push-forward of the measure $\mu$ by $f$, which is a measure in $(Y, \mathfrak S_Y)$. We use $\mathfrak B$ (respectively, $\mathfrak L$) to denote the Borel (respectively, Lebesgue) $\sigma$-algebra in $\mathbb R^d$ or in any of its (measurable) subsets, and $\mathcal L$ to denote the Lebesgue measure in $(\mathbb R^d, \mathfrak L)$.

\section{Preliminary results}

Note that \eqref{eq:diff-eqn-N1-time} can be rewritten as $x(t) = A x(\sigma_1(t))$, where $\sigma_1(t) = t - \tau(t)$. If one wishes to consider solutions $t \mapsto x(t)$ of \eqref{eq:diff-eqn-N1-time} with a certain regularity in $t$, it is important to ensure that such a regularity is preserved by the composition $x \circ \sigma_1$. For instance, a necessary and sufficient condition for $x \circ \sigma_1$ to be continuous for every continuous $x$ is the continuity of $\sigma_1$. This section studies such a question for other types of regularity. More precisely, in Section~\ref{sec:regulated}, we provide the definition of regulated functions and identify the class of functions preserving regulatedness by right composition, while, in Section~\ref{sec:Lp}, we address a similar question for functions in $L^p$.

\subsection{Regulated functions}
\label{sec:regulated}

\begin{definition}
\label{def:regulated}
Let $a, b \in \mathbb R$ with $a < b$ and $f \colon [a, b] \to \mathbb R$.
\begin{enumerate}
\item We say that $f$ is \emph{regulated} if, for every $t_0 \in [a, b)$, the lateral limit $f(t_0^+)$ exists and, for every $t_0 \in (a, b]$, the lateral limit $f(t_0^-)$ exists. The set of regulated functions defined in $[a, b]$ is denoted by $G([a, b], \mathbb R)$.

\item\label{item:well-regulated} We say that $f$ is \emph{well-regulated} if $f$ is regulated and, in addition, the following two assertions hold true:
\begin{enumerate}[label={(\roman*)}, ref={\roman*}]
\item\label{item:well-regulated-+} For every $t_0 \in [a, b)$, there exists $\varepsilon \in (0, b - t_0)$ such that one of the following three assertions holds true:
\begin{enumerate}[label={(\theenumii-\arabic*)}]
    \item\label{item:well-regulated-++} $f(t) > f(t_0^+)$ for every $t \in (t_0, t_0 + \varepsilon)$;
    \item\label{item:well-regulated-+0} $f(t) = f(t_0^+)$ for every $t \in (t_0, t_0 + \varepsilon)$;
    \item\label{item:well-regulated-+-} $f(t) < f(t_0^+)$ for every $t \in (t_0, t_0 + \varepsilon)$.
\end{enumerate}
\item\label{item:well-regulated--} For every $t_0 \in (a, b]$, there exists $\varepsilon \in (0, t_0 - a)$ such that one of the following three assertions holds true:
\begin{enumerate}[label={(\theenumii-\arabic*)}]
    \item\label{item:well-regulated--+} $f(t) > f(t_0^-)$ for every $t \in (t_0 - \varepsilon, t_0)$;
    \item\label{item:well-regulated--0} $f(t) = f(t_0^-)$ for every $t \in (t_0 - \varepsilon, t_0)$;
    \item\label{item:well-regulated---} $f(t) < f(t_0^-)$ for every $t \in (t_0 - \varepsilon, t_0)$.
\end{enumerate}
\end{enumerate}
The set of well-regulated functions defined in $[a, b]$ is denoted by $WG([a, b], \mathbb R)$.
\end{enumerate}
\end{definition}

We extend the above definitions to functions defined in an interval of the form $[a, +\infty)$ by saying that $f$ is regulated (respectively, well-regulated) in $[a, +\infty)$ if it is regulated (respectively, well-regulated) in $[a, b]$ for every $b > 0$, and we denote the corresponding functional space by\footnote{In the literature, $G([a, +\infty), \mathbb R)$ is usually defined as the set of \emph{bounded} regulated functions in order for it to be a Banach space, but we do not make this requirement here.} $G([a, +\infty), \mathbb R)$ (respectively, $WG([a, +\infty), \mathbb R)$). We also extend the definitions similarly to functions defined in an interval of the form $(-\infty, a]$ or in $\mathbb R$. A vector-valued function is said to be regulated if all of its components are regulated, and the space of regulated functions defined in a nonempty closed interval $I \subset \mathbb R$ and taking values in $\mathbb R^d$ is denoted by $G(I, \mathbb R^d)$. Finally, we denote by $G(I, A)$ (respectively, $WG(I, A)$) the set of all regulated (respectively, well-regulated) functions defined in the nonempty closed interval $I$ and taking values in $A \subset \mathbb R^d$ (respectively, $A \subset \mathbb R$).

\begin{remark}
If $I \subset \mathbb R$ is a nonempty compact interval, then $G(I, \mathbb R^d)$ is a Banach space when endowed with the norm $\norm{f}_{G(I, \mathbb R^d)} = \sup_{t \in I} \abs{f(t)}$. If $I$ is unbounded, the above function is no longer a norm in $G(I, \mathbb R^d)$ since the latter set may contain unbounded functions, but $G(I, \mathbb R^d)$ is a Fréchet space when endowed with the family of seminorms $\norm{f}_{G(I, \mathbb R^d), n} = \sup_{t \in I \cap [-n, n]} \abs{f(t)}$ for $n \in \mathbb N^\ast$.
\end{remark}

\begin{remark}
\label{remk:cinfty-not-well-regulated}
Well-regulated functions include several classes of functions, such as monotone functions, piecewise monotone functions (with a finite number of changes in monotonicity in every bounded interval), or analytic functions. However, continuous functions, and even $C^{\infty}$ functions, may fail to be well-regulated. Indeed, the function $f\colon [-1, 1] \to \mathbb R$ defined by $f(t) = e^{-1/t^2} \sin(1/t)$ for $t \neq 0$ and $f(0) = 0$ belongs to $C^{\infty}([-1, 1], \mathbb R)$ but is not well-regulated.
\end{remark}

\begin{remark}
\label{remk:sum-not-well-regulated}
The sum of well-regulated functions may fail to be well-regulated. Indeed, the function $f\colon \mathbb R \to \mathbb R$ defined by $f(0) = 0$ and $f(t) = -t \left(1 - \frac{1}{2} \sin\left(\frac{1}{t}\right)\right)$ for $t \neq 0$ is well-regulated, as well as the identity function $\id$, but $f + \id$ is not well-regulated.
\end{remark}

The main interest of the class of well-regulated functions is the next result, which identifies well-regulated functions as the class of functions preserving regulatedness by right composition.

\begin{theorem}
\label{thm:well-regulated}
Let $I, J$ be two nonempty closed intervals in $\mathbb R$ and $d$ be a positive integer. Then
\begin{equation}
\label{eq:charact-well-regulated}
WG(J, I) = \{\varphi\colon J \to I \suchthat f \circ \varphi \in G(J, \mathbb R^d) \text{ for every } f \in G(I, \mathbb R^d)\}.
\end{equation}
\end{theorem}

\begin{proof}
With no loss of generality, we assume in this proof that $d = 1$.

Let $\varphi \in WG(J, I)$ and $f \in G(I, \mathbb R)$. Let $t_0$ be either an interior point of $J$ or its left extremity, if the latter is finite. Since $\varphi$ is well-regulated, $\varphi(t_0^+)$ exists and there exists $\eta > 0$ such that $\varphi$ satisfies \ref{item:well-regulated-++}, \ref{item:well-regulated-+0}, or \ref{item:well-regulated-+-} from Definition~\ref{def:regulated}\ref{item:well-regulated} in the interval $(t_0, t_0 + \eta)$.

Let us consider first the case \ref{item:well-regulated-++}, and assume thus that $\varphi(t) > \varphi(t_0^+)$ for every $t \in (t_0, t_0 + \eta)$. Let $L = \lim_{s \to \varphi(t_0^+)^+} f(s)$, which exists since $f$ is regulated. Take $\varepsilon > 0$. Then there exists $\delta > 0$ such that $\abs{f(s) - L} < \varepsilon$ for every $s \in (\varphi(t_0^+), \varphi(t_0^+) + \delta)$. On the other hand, by definition of $\varphi(t_0^+)$, there exists $\rho > 0$ such that $\abs{\varphi(t) - \varphi(t_0^+)} < \delta$ for every $t \in (t_0, t_0 + \rho)$. Letting $\widetilde\rho = \min\{\rho, \eta\}$, we deduce that, for every $t \in (t_0, t_0 + \widetilde\rho)$, we have $\varphi(t) \in (\varphi(t_0^+), \varphi({t_0^+}) + \delta)$, and hence $\abs{f(\varphi(t)) - L} < \varepsilon$, proving that $\lim_{t \to t_0^+} f(\varphi(t)) = L$. Cases~\ref{item:well-regulated-+0} and \ref{item:well-regulated-+-} are treated by similar arguments, setting $L = f(\varphi(t_0^+))$ and $L = f(\varphi(t_0^+)^-)$, respectively. Hence, $\lim_{t \to t_0^+} f(\varphi(t))$ exists.

One shows by a completely similar argument that $\lim_{t \to t_0^-} f(\varphi(t))$ exists for every $t_0$ in the interior of $J$ or equal to its right extremity if it is finite. Hence, $f \circ \varphi$ is regulated, as required.

Assume now that $\varphi\colon J \to I$ is such that $f \circ \varphi \in G(J, \mathbb R)$ for every $f \in G(I, \mathbb R)$. Notice that, since the identity function $\id$ is regulated, we have $\varphi = \id \circ \varphi \in G(J, \mathbb R)$, i.e., $\varphi$ is regulated.

Assume, to obtain a contradiction, that \eqref{item:well-regulated-+} or \eqref{item:well-regulated--} is false. We only treat the first case, the other case being treated similarly. Then there exists $t_0$ in the interior of $J$ or equal to its left extremity if the latter is finite and a sequence $(t_n)_{n \in \mathbb N^\ast}$ in $J$ with $t_n \to t_0^+$ as $n \to +\infty$ such that at least two of the sets
\begin{align*}
I_+ & = \{n \in \mathbb N^\ast \suchthat \varphi(t_n) > \varphi(t_0^+)\}, \\
I_0 & = \{n \in \mathbb N^\ast \suchthat \varphi(t_n) = \varphi(t_0^+)\}, \\
I_- & = \{n \in \mathbb N^\ast \suchthat \varphi(t_n) < \varphi(t_0^+)\},
\end{align*}
are infinite. Let $f\colon I \to \mathbb R$ be defined for $s \in I$ by $f(s) = \sign(s - \varphi(t_0^+))$, and notice that $f \in G(I, \mathbb R)$. We have
\begin{align*}
I_+ & = \{n \in \mathbb N^\ast \suchthat f(\varphi(t_n)) = 1\}, \\
I_0 & = \{n \in \mathbb N^\ast \suchthat f(\varphi(t_n)) = 0\}, \\
I_- & = \{n \in \mathbb N^\ast \suchthat f(\varphi(t_n)) = -1\},
\end{align*}
and since at least two of the above sets are infinite, the sequence $(f(\varphi(t_n)))_{n \in \mathbb N}$ does not converge. Since $t_n \to t_0^+$, this implies that $\lim_{t \to t_0^+} f(\varphi(t))$ does not exist, and hence $f \circ \varphi$ is not regulated, which is a contradiction since $\varphi$ belongs to the set in the right-hand side of \eqref{eq:charact-well-regulated}.
\end{proof}

\begin{remark}
The above proof can be easily adapted to regulated functions taking values in a Banach space $\mathsf X$. Indeed, the proof that any well-regulated function preserves regulatedness by right composition is identical, while, for the proof of the converse statement, one replaces in the argument provided above the identity function and the sign function by the functions $t \mapsto t v$ and $t \mapsto \sign(t) v$, respectively, where $v$ is an arbitrary nonzero element of $\mathsf X$.
\end{remark}

\begin{remark}
\label{remk:composition-regulated}
This property of well-regulatedness plays important role, since the composition of regulated functions in general is not a regulated function. For instance, consider two regulated functions $f,g \colon [0,1] \to \mathbb R$ given by $f(x)= x\sin(1/x)$ and $g(x) = \sign (x)$, both are regulated, however the composition $g \circ f$ has no one-sided limits at $0$. Also, even a
composition of a regulated and continuous functions need not to
be regulated.
\end{remark}

\subsection{Right composition of a function in \texorpdfstring{$L^p$}{Lp}}
\label{sec:Lp}

We are now interested in identifying the class of functions $g$ such that $f \circ g$ belongs to $L^p$ for every $f \in L^p$, i.e., the operator of composition with $g$ on the right preserves the $L^p$ space. As a preliminary step, we provide the following result on the $(\mathfrak L, \mathfrak B)$-measurability of the composition $f \circ g$.

\begin{proposition}
\label{prop:preserves-measurability}
For $i \in \{1, 2\}$, let $X_i \subset \mathbb R^{d_i}$, $m \in \mathbb N^\ast$, and $g\colon X_1 \to X_2$. The following assertions are equivalent.
\begin{enumerate}
\item\label{item:preserves-measurability} For every $f\colon X_2 \to \mathbb R^m$ $(\mathfrak L, \mathfrak B)$-measurable, $f \circ g$ is $(\mathfrak L, \mathfrak B)$-measurable.
\item\label{item:g-is-LL-measurable} The function $g$ is $(\mathfrak L, \mathfrak L)$-measurable.
\end{enumerate}
\end{proposition}

\begin{proof}
The fact that \ref{item:g-is-LL-measurable} implies \ref{item:preserves-measurability} follows immediately from the definition of measurable functions. Conversely, assume that \ref{item:preserves-measurability} holds true and let $A \subset X_2$ be Lebesgue measurable. Let $v \in \mathbb R^m \setminus \{0\}$ and consider the function $f\colon X_2 \to \mathbb R^m$ be given for $x \in X_2$ by $f(x) = \mathbbm 1_A(x) v$. By assumption, $f \circ g$ is $(\mathfrak L, \mathfrak B)$-measurable and thus, in particular, $(f \circ g)^{-1}(\{v\}) = g^{-1}(A)$ is Lebesgue measurable, yielding \ref{item:g-is-LL-measurable}.
\end{proof}

We are now in position to characterize the set of functions preserving $L^p$ by right composition.

\begin{theorem}
\label{thm:composition-Lp}
For $i \in \{1, 2\}$, let $X_i \subset \mathbb R^{d_i}$ be an open set, $m \in \mathbb N^\ast$, and $g\colon X_1 \to X_2$.
\begin{enumerate}[ref=\alph*]
\item\label{item:composition-Linfty-equivs} The following assertions are equivalent.
\begin{enumerate}[label=\textup{(\theenumi-\roman*)}, leftmargin=*]
\item\label{item:composition-Linfty} For every $f \in L^\infty(X_2, \mathbb R^m)$, we have $f \circ g \in L^\infty(X_1, \mathbb R^m)$.
\item\label{item:g-absolutely-continuous} The function $g$ is $(\mathfrak L, \mathfrak L)$-measurable and the measure $g_\# \mathcal L$ is absolutely continuous with respect to $\mathcal L$.
\end{enumerate}
\item\label{item:composition-Lp-equivs} Assume further that $\mathcal L(X_1) < +\infty$ and take $p \in [1, +\infty)$. Then the following assertions are equivalent.
\begin{enumerate}[label=\textup{(\theenumi-\roman*)}, leftmargin=*]
\item\label{item:composition-Lp} For every $f \in L^p(X_2, \mathbb R^m)$, we have $f \circ g \in L^p(X_1, \mathbb R^m)$.
\item\label{item:g-radon-nikodym} The function $g$ is $(\mathfrak L, \mathfrak L)$-measurable, the measure $g_\# \mathcal L$ is absolutely continuous with respect to $\mathcal L$, and the Radon--Nikodym derivative of $g_\# \mathcal L$ with respect to $\mathcal L$ is bounded.
\end{enumerate}
\end{enumerate}
\end{theorem}

\begin{proof}
Let us first prove \eqref{item:composition-Linfty-equivs}. Assume that \ref{item:g-absolutely-continuous} holds true and let $f \in L^\infty(X_2, \mathbb R^m)$. Recall that, by Proposition~\ref{prop:preserves-measurability}, $f \circ g$ is $(\mathfrak L, \mathfrak B)$-measurable. In addition, since $f$ is bounded, so is $f \circ g$. One is only left to show that $f \circ g$ defines an element in $L^\infty(X_1, \mathbb R^m)$, i.e., that, if $\widetilde f = f$ a.e.\ in $X_2$, then $\widetilde f \circ g = f \circ g$ a.e.\ in $X_1$. Let $N = \{x \in X_2 \suchthat \widetilde f(x) \neq f(x)\}$ and notice that $\mathcal L(N) = 0$. In addition, $\widetilde f \circ g(x) \neq f \circ g(x)$ if and only if $x \in g^{-1}(N)$. Since $g_\# \mathcal L$ is absolutely continuous with respect to $\mathcal L$, we have $\mathcal L(g^{-1}(N)) = 0$, showing that $\widetilde f \circ g = f \circ g$ a.e.\ in $X_1$ and yielding the conclusion.

Assume now that \ref{item:composition-Linfty} holds true. In particular, by Proposition~\ref{prop:preserves-measurability}, $g$ is $(\mathfrak L, \mathfrak L)$-measurable. Let $N \subset X_2$ be a Lebesgue-measurable set satisfying $\mathcal L(N) = 0$, $v \in \mathbb R^m \setminus \{0\}$, and $f\colon X_2 \to \mathbb R^m$ be given for $x \in X_2$ by $f(x) = \mathbbm 1_N(x) v$. Then $f \in L^\infty(X_2, \mathbb R^m)$, $f = 0$ a.e.\ in $X_2$, and thus $f \circ g = 0$ a.e.\ in $X_1$, showing that $\mathcal L(g^{-1}(N)) = \mathcal L(\{x \in X_1 \suchthat f \circ g(x) = v\}) = 0$. This proves that $g_\# \mathcal L$ is absolutely continuous with respect to $\mathcal L$.

Let us now prove \eqref{item:composition-Lp-equivs}. Assume that \ref{item:g-radon-nikodym} holds true and let $f \in L^p(X_2, \mathbb R^m)$. As before, the $(\mathfrak L, \mathfrak B)$-measurability of $f \circ g$ follows from Proposition~\ref{prop:preserves-measurability} and, arguing as in the proof of \eqref{item:composition-Linfty-equivs}, we obtain that, if $\widetilde f = f$ a.e.\ in $X_2$, then $\widetilde f \circ g = f \circ g$ a.e.\ in $X_1$. In addition, denoting by $\varphi\colon X_2 \to \mathbb R^m$ the Radon--Nikodym derivative of $g_\# \mathcal L$ with respect to $\mathcal L$, we have
\[
\int_{X_1} \abs{f \circ g(x)}^p \diff x = \int_{X_2} \abs{f(x)}^p \diff g_\# \mathcal L(x) = \int_{X_2} \abs{f(x)}^p \varphi(x) \diff x < +\infty,
\]
since $\varphi$ is bounded. Hence $f \circ g \in L^p(X_1, \mathbb R^m)$.

Assume now that \ref{item:composition-Lp} holds true. By Proposition~\ref{prop:preserves-measurability}, $g$ is $(\mathfrak L, \mathfrak L)$-measurable and, arguing as in the proof of \eqref{item:composition-Linfty-equivs}, $g_\# \mathcal L$ is absolutely continuous with respect to $\mathcal L$. Since $\mathcal L(X_1) < +\infty$, the measure $g_{\#} \mathcal L$ is finite, and hence its Radon--Nikodym derivative $\varphi\colon X_2 \to \mathbb R_+ \cup \{+\infty\}$ with respect to $\mathcal L$ exists.

Let $L\colon L^1(X_2, \mathbb R) \to \mathbb R$ be the linear map defined for $\psi \in L^1(X_2, \mathbb R)$ by
\[
L \psi = \int_{X_2} \psi(x) \varphi(x) \diff x.
\]
Let us prove first that $L$ is well-defined. Indeed, write $\psi = \psi_+ - \psi_-$, where $\psi_+ = \max(\psi, 0)$ and $\psi_- = \max(-\psi, 0)$, and note that $\psi_+$ and $\psi_-$ are nonnegative functions in $L^1(X_2, \mathbb R)$. Let $f_+$ and $f_-$ be functions in $L^p(X_2, \mathbb R^m)$ such that $\psi_+ = \abs{f_+}^p$ and $\psi_- = \abs{f_-}^p$. By \ref{item:composition-Lp}, both $f_+ \circ g$ and $f_- \circ g$ belong to $L^p(X_1, \mathbb R^m)$, and we compute
\[
\int_{X_2} \psi_{\pm}(x) \varphi(x) \diff x = \int_{X_2} \abs{f_{\pm}(x)}^p \diff g_\# \mathcal L(x) = \int_{X_1} \abs{f_{\pm} \circ g(x)}^p \diff x < +\infty.
\]
Hence, $\psi \varphi \in L^1(X_2, \mathbb R)$, and thus $L \psi$ is well-defined.

For $n \in \mathbb N^\ast$, let $\varphi_n = \min(\varphi, n)$. Let $L_n\colon L^1(X_2, \mathbb R) \to \mathbb R$ be the continuous linear operator defined for $\psi \in L^1(X_2, \mathbb R)$ by $L_n \psi = \int_{X_2} \psi(x) \varphi_n(x) \diff x$. For a given $\psi \in L^1(X_2, \mathbb R)$, we have
\begin{align*}
L_n \psi & = \int_{X_2} \psi(x) \varphi_n(x) \diff x \\
& = \int_{X_2} \psi_+(x) \varphi_n(x) \diff x - \int_{X_2} \psi_-(x) \varphi_n(x) \diff x \\
& \xrightarrow[n \to +\infty]{} \int_{X_2} \psi_+(x) \varphi(x) \diff x - \int_{X_2} \psi_-(x) \varphi(x) \diff x \\
& = \int_{X_2} \psi(x) \varphi(x) \diff x = L \psi,
\end{align*}
where the convergence holds true thanks to the Lebesgue monotone convergence theorem. Hence, by the Banach--Steinhaus theorem, $L$ is a continuous linear operator, which implies that $L$ belongs to the topological dual of $L^1(X_2, \mathbb R)$, and thus $\varphi \in L^\infty(X_2, \mathbb R)$, as required.
\end{proof}

\section{Hypotheses and definitions}
\label{sec:Hypo-def}

Let us now present and comment the main definitions and hypotheses used in the sequel of the paper. We start with the definition of delay function used in this paper.

\begin{definition}
A function $\tau\colon \mathbb R_+ \to \mathbb R_+^\ast$ is said to be a \emph{delay function}.
\end{definition}

The study of the well-posedness and the asymptotic behavior of \eqref{eq:diff-eqn-N1-time} requires several assumptions. For simplicity, we state all such assumptions here, and refer to them whenever needed in the sequel.

\begin{hypothesis}
\label{hypo:hypos}
Let $A \in \mathcal M_d(\mathbb R)$, $p \in [1, +\infty]$, and $\tau$ be a delay function and define $\sigma_1\colon \mathbb R_+ \to \mathbb R$ by $\sigma_1(t) = t - \tau(t)$.
\begin{hypolist}
\item\label{hypo:tau-infimum} For every $t \geq 0$, we have $\inf_{s \in [0, t]} \tau(s) > 0$.
\item\label{hypo:tau-cont} The function $\tau$ is continuous.
\item\label{hypo:regulated} The function $\sigma_1$ is well-regulated.
\item\label{hypo:tau-measurable} The function $\tau$ is $(\mathfrak L, \mathfrak L)$-measurable.
\item\label{hypo:null} We have $\mathcal L(\sigma_1^{-1}(N)) = 0$ for every $N \in \mathfrak L$ such that $\mathcal L(N) = 0$.
\item\label{hypo:radon-nik-bdd} Hypotheses~\ref{hypo:tau-measurable} and \ref{hypo:null} hold true and, for every $T > 0$, the Radon--Nikodym derivative of the measure $\left(\sigma_1\rvert_{[0, T]}\right)_{\#}\mathcal L$ with respect to $\mathcal L$ is bounded.
\item\label{hypo:spr-A-less-1} We have $\rho(A) < 1$.
\item\label{hypo:delay-infinity} We have $\lim_{t \to +\infty} \sigma_1(t) = +\infty$.
\item\label{hypo:radon-nik-and-rho} Hypotheses~\ref{hypo:tau-measurable} and \ref{hypo:null} hold true, $p < +\infty$, ${\sigma_1}_{\#} \mathcal L$ is $\sigma$-finite, and the Radon--Nikodym derivative $\varphi$ of ${\sigma_1}_{\#}\mathcal L$ with respect to $\mathcal L$ is bounded and satisfies the inequality $\norm{\varphi}_{L^\infty(\mathbb R, \mathbb R)} \rho(A)^p < 1$.
\item\label{hypo:delay-bounded} The delay function $\tau$ is bounded.
\end{hypolist}
\end{hypothesis}

Assumption~\ref{hypo:tau-infimum} is used to obtain the representation formula \eqref{eq:explicit-formula} from Theorem~\ref{theorem-1}, and it is important for the uniqueness of solutions of \eqref{eq:diff-eqn-N1-time} (see Remark~\ref{remk:H2}). The continuity of $\tau$ assumed in \ref{hypo:tau-cont} is important if one wants to ensure the existence of continuous solutions of \eqref{eq:diff-eqn-N1-time} (see Theorem~\ref{thm:exist-continuous}), while, in view of Theorem~\ref{thm:well-regulated}, \ref{hypo:regulated} is used to ensure existence of solutions of \eqref{eq:diff-eqn-N1-time} in the class of regulated functions (see Theorem~\ref{thm:exist-regulated}).

In view of Theorem~\ref{prop:preserves-measurability}, the measurability of $\tau$ assumed in \ref{hypo:tau-measurable} is important to obtain existence of measurable solutions of \eqref{eq:diff-eqn-N1-time} (see Theorem~\ref{thm:exist-measurable}). Assumption~\ref{hypo:null} states that $\sigma_1$ preserves the class of sets of Lebesgue measure zero, and it turns out to be a necessary assumption to provide a notion of solution verifying \eqref{eq:diff-eqn-N1-time} only for almost every $t \in \mathbb R_+$ (see Theorem~\ref{thm:exist-ae} and Proposition~\ref{prop:h2-necessary}). This is fundamental in order to consider solutions of \eqref{eq:diff-eqn-N1-time} in $L^p$ spaces. Note that \ref{hypo:tau-measurable} ensures that ${\sigma_1}_{\#} \mathcal L$ is a measure in $(\mathbb R, \mathfrak L)$ and, under this assumption, \ref{hypo:null} is equivalent to requiring ${\sigma_1}_{\#} \mathcal L$ to be absolutely continuous with respect to the Lebesgue measure, which, by Theorem~\ref{thm:composition-Lp}\eqref{item:composition-Linfty-equivs}, is fundamental for ensuring that any $L^\infty$ function remains in $L^\infty$ after a right composition with $\sigma_1$. In addition, \ref{hypo:tau-measurable} and \ref{hypo:null} ensure the existence of the Radon--Nikodym derivative in \ref{hypo:radon-nik-bdd}, since $\left(\sigma_1\rvert_{[0, T]}\right)_{\#}\mathcal L$ is a finite measure. Finally, taking into account Theorem~\ref{thm:composition-Lp}, \ref{hypo:radon-nik-bdd} allows one to prove existence of solutions of \eqref{eq:diff-eqn-N1-time} in $L^p$ for $p \in [1, +\infty)$.

Note that \ref{hypo:tau-cont} is equivalent to assuming continuity of $\sigma_1$, and \ref{hypo:tau-measurable} is equivalent to assuming $(\mathfrak L, \mathfrak L)$-measurability of $\sigma_1$, but, due to Remark~\ref{remk:sum-not-well-regulated}, \ref{hypo:regulated} is not equivalent to assuming well-regulatedness of $\tau$. 

Assumptions~\ref{hypo:spr-A-less-1}--\ref{hypo:delay-bounded} are useful for the analysis of the asymptotic behavior of solutions of \eqref{eq:diff-eqn-N1-time} carried out in Section~\ref{sec:stability}. Indeed, in order to obtain that solutions of \eqref{eq:diff-eqn-N1-time} converge to $0$ as time tends to infinity, it seems natural from the form of \eqref{eq:diff-eqn-N1-time} to require \ref{hypo:spr-A-less-1}, since one may expect the existence of solutions diverging to infinity in norm in the case $\rho(A) > 1$, and not converging to $0$ in the case $\rho(A) = 1$. In addition, if $\sigma_1(t) = t - \tau(t)$ were to remain bounded as $t \to +\infty$, then, from \eqref{eq:diff-eqn-N1-time}, $x(t)$ would depend, as $t \to +\infty$, only on the past values of $x(\cdot)$ computed in a bounded time interval, and one might expect solutions of \eqref{eq:diff-eqn-N1-time} to not necessarily converge to $0$ in this case. These intuitions will be made precise later on in Theorem~\ref{thm:cv-to-0} and Proposition~\ref{prop:necessary-asympt}, which show that, under \ref{hypo:tau-infimum}, assumptions~\ref{hypo:spr-A-less-1} and \ref{hypo:delay-infinity} are necessary\footnote{The necessity of \ref{hypo:delay-infinity} requires also the assumption that $A$ is not nilpotent.} and sufficient to have $\abs{x(t)} \to 0$ as $t \to +\infty$ for any solution $x$ of \eqref{eq:diff-eqn-N1-time} with bounded initial condition.

In addition to the convergence of $\abs{x(t)}$ to $0$ as time tends to infinity, we are also interested in Section~\ref{sec:stability} in the convergence to $0$ of functional norms of $x(\cdot)$ computed at time intervals tending to infinity. While \ref{hypo:spr-A-less-1} and \ref{hypo:delay-infinity} remain sufficient for the $L^\infty$ norm, we will show (Example~\ref{example-3.1}) that this is not the case for $L^p$ norms with $p \in [1, +\infty)$, and it turns out that requiring also \ref{hypo:radon-nik-and-rho} to hold true is sufficient to retrieve convergence to $0$ also in $L^p$. Note that, in \ref{hypo:radon-nik-and-rho}, the assumption that ${\sigma_1}_{\#}\mathcal L$ is $\sigma$-finite is used to ensure that the Radon--Nikodym derivative $\varphi$ exists, and it is satisfied if \ref{hypo:delay-infinity} holds true.

Finally, assumption~\ref{hypo:delay-bounded} is useful in Section~\ref{sec:exponential} to obtain exponential decay rates of solutions of \eqref{eq:diff-eqn-N1-time}. Note that, trivially, \ref{hypo:delay-bounded} implies \ref{hypo:delay-infinity}.

\begin{remark}
A sufficient condition for \ref{hypo:radon-nik-bdd} to hold is that, in addition to \ref{hypo:tau-measurable} and \ref{hypo:null}, ${\sigma_1}_{\#} \mathcal L$ is $\sigma$-finite and its Radon--Nikodym derivative $\varphi$ with respect to $\mathcal L$ is bounded on $(-\infty, T]$ for every $T > 0$. Indeed, using the fact that $\sigma_1(t) < t$ for every $t \in \mathbb R_+$, we deduce that, for every $T > 0$, the Radon--Nikodym derivative of $\left(\sigma_1\rvert_{[0, T]}\right)_{\#} \mathcal L$ is supported on $(-\infty, T]$ and upper bounded by $\varphi$ on this set.

However, the previous condition is not necessary for \ref{hypo:radon-nik-bdd} to hold true: for instance, if $\tau(t) = \floor{t + 1}$ for $t \in \mathbb R_+$, we have $\sigma_1(t) = t - \floor{t + 1}$ for $t \in \mathbb R_+$, and we can verify that \ref{hypo:radon-nik-bdd} is satisfied, but ${\sigma_1}_{\#}\mathcal L$ is not $\sigma$-finite, and hence it does not admit a Radon--Nikodym derivative with respect to $\mathcal L$.
\end{remark}

\begin{remark}
\label{remk:radon-nikodym-derivative-tau}
If $\sigma_1$ is defined from a delay function $\tau$ as in Hypotheses~\ref{hypo:hypos} and it is a diffeomorphism between $\mathbb R_+$ and an interval $I \subset \mathbb R$, then ${\sigma_1}_{\#} \mathcal L$ is absolutely continuous with respect to $\mathcal L$, it is $\sigma$-finite, and its Radon--Nikodym derivative with respect to $\mathcal L$ is the function $\varphi\colon \mathbb R \to \mathbb R_+$ defined by $\varphi(t) = \abs{(\sigma_1^{-1})'(t)} = \frac{1}{\abs{\sigma_1'(\sigma_1^{-1}(t))}}$ if $t \in I$ and $\varphi(t) = 0$ otherwise. Given $T > 0$, $\varphi$ is bounded on $(-\infty, T]$ if and only if $\abs{\sigma_1'}$ is bounded from below on $(-\infty, T]$ by a positive constant.

In terms of $\tau$, a sufficient condition for \ref{hypo:radon-nik-bdd} to hold true is that $\tau$ is differentiable and, for every $T > 0$ there exists $\alpha \in [0, 1)$ such that $\tau'(t) \leq \alpha$ for every $t \in [0, T]$. In that case, $\varphi$ is bounded on $(-\infty, T]$ by $\frac{1}{1 - \alpha}$.
\end{remark}

We now introduce several definitions used in the sequel of the paper.

\begin{definition}
\label{def:h}
Let $\tau$ be a delay function. The \emph{largest delay function} associated with $\tau$ is the function $h\colon \mathbb R_+ \to \mathbb R_+^\ast \cup \{+\infty\}$ defined for $t \in \mathbb R_+$ by
\[
h(t) = t - \inf_{s \in [t, +\infty)} (s - \tau(s)).
\]
\end{definition}

The largest delay function $h$ introduced in Definition~\ref{def:h} can be interpreted as follows: if a function $x$ satisfies \eqref{eq:diff-eqn-N1-time} for every $t \geq t_0$, for some $t_0 \in \mathbb R$, then the right-hand side of \eqref{eq:diff-eqn-N1-time} depends on evaluations of $x$ at the interval $\lvert t_0 - h(t_0), +\infty)$. In particular, if one is to define a solution of \eqref{eq:diff-eqn-N1-time} in an interval of the form $[t_0, +\infty)$, then the corresponding initial condition should be a function defined in an interval of the form $\lvert t_0 - h(t_0), t_0)$. Notice also that, from the definition of $h$, the function $t_0 \mapsto t_0 - h(t_0)$ is nondecreasing.

\begin{definition}
Let $\tau$ be a delay function and $h$ be its associated largest delay function.
\begin{enumerate}
\item Given $x\colon \lvert -h(0), +\infty) \to \mathbb R^d$ and $t \geq 0$, we denote by $x_t\colon \lvert -h(t), 0) \to \mathbb R^d$ the function defined by $x_t(s) = x(t + s)$ for every $s \in \lvert -h(0), 0)$. Whenever needed, we also make the slight abuse of notation of considering $x_t$ to be defined on $\lvert -h(t), 0]$ by setting $x_t(0) = x(t)$.

\item For every $t \in \mathbb R_+$, we set $\mathsf C_t = \{x \in C(\lvert -h(t), 0], \mathbb R^d) \suchthat x(0) = A x(-\tau(t))\}$ and $\mathsf C_t^{\mathrm b} = \{x \in \mathsf C_t \suchthat x \text{ is bounded}\}$.

\item For every $t \in \mathbb R_+$, we set $\mathsf G_t = \{x \in G(\lvert -h(t), 0], \mathbb R^d) \suchthat x(0) = A x(-\tau(t))\}$ and $\mathsf G_t^{\mathrm b} = \{x \in \mathsf G_t \suchthat x \text{ is bounded}\}$.
\end{enumerate}
\end{definition}

Note that the sets $\mathsf C_t$ and $\mathsf G_t$ depend on $\tau$ and $A$, but we do not make this dependence explicit in their notation. We also remark that the interval $\lvert -h(t), 0]$ is not compact when $h(t) = +\infty$, motivating the boundedness requirements in the definitions of $\mathsf C_t^{\mathrm b}$ and $\mathsf G_t^{\mathrm b}$. In particular, $\mathsf C_t^{\mathrm b}$ and $\mathsf G_t^{\mathrm b}$ are Banach spaces with norms
\[
\norm{x}_{\mathsf C_t^{\mathrm b}} = \sup_{t \in \lvert -h(t), 0]} \abs{x(t)} \quad \text{ and } \quad \norm{x}_{\mathsf G_t^{\mathrm b}} = \sup_{t \in \lvert -h(t), 0]} \abs{x(t)},
\]
respectively.

\section{Well-posedness}
\label{sec:well-posed}

We next address questions on the existence and uniqueness of solutions of \eqref{eq:diff-eqn-N1-time}, and we start by introducing the three main notions of solutions used along this paper, namely continuous, regulated, and $L^p$ solutions, as well as some other auxiliary notions.

\begin{definition}
\label{def:solutions}
Let $\tau$ be a delay function and $h$ be its associated maximal delay function.
\begin{enumerate}
\item We say that $x\colon \lvert -h(0), +\infty) \to \mathbb R^d$ is a \emph{solution} of \eqref{eq:diff-eqn-N1-time} if the equality $x(t) = A x(t - \tau(t))$ is satisfied for every $t \in \mathbb R_+$.

\item A solution $x$ of \eqref{eq:diff-eqn-N1-time} is said to be a \emph{continuous} (respectively, \emph{regulated}) \emph{solution} of \eqref{eq:diff-eqn-N1-time} if it is continuous (respectively, regulated).

\item A solution $x$ of \eqref{eq:diff-eqn-N1-time} is said to be a \emph{measurable solution} of \eqref{eq:diff-eqn-N1-time} if it is $(\mathfrak L, \mathfrak B)$-measurable.

\item We say that $x\colon\lvert -h(0), +\infty) \to \mathbb R^d$ is an \emph{almost everywhere (a.e.) solution} of \eqref{eq:diff-eqn-N1-time} if, for every function $\widetilde x\colon \lvert -h(0), +\infty) \to \mathbb R^d$ equal to $x$ almost everywhere on $\lvert -h(0), +\infty)$, we have $\widetilde x(t) = A \widetilde x(t - \tau(t))$ for almost every $t \in \mathbb R_+$.

\item For $p \in [1, +\infty]$, a measurable a.e.\ solution $x$ of \eqref{eq:diff-eqn-N1-time} is said to be an \emph{$L^p$ solution} of \eqref{eq:diff-eqn-N1-time} if $x \in L^p_{\mathrm{loc}}(\lvert -h(0), +\infty), \mathbb R^d)$.
\end{enumerate}

If $x$ is a solution, a measurable solution, an a.e.\ solution, or an $L^p$ solution of \eqref{eq:diff-eqn-N1-time}, we say that its \emph{initial condition} is its restriction to the interval $\lvert -h(0), 0)$. If $x$ is a continuous or regulated solution of \eqref{eq:diff-eqn-N1-time}, we say that its \emph{initial condition} is its restriction to the interval $\lvert -h(0), 0]$.
\end{definition}

\subsection{Representation formula for solutions}
\label{sec:explicit-formula}

The goal of this section is to provide a representation formula for solutions of \eqref{eq:diff-eqn-N1-time}, expressing a solution $x$ at a given time $t$ in terms of its initial condition $x_0$ evaluated at well-identified times and suitable powers of the matrix $A$. For that purpose, we start by a suitable definition of iterates of the function $t \mapsto t - \tau(t)$.

\begin{definition}
\label{def:Dn-sigman}
Let $\tau$ be a delay function. For $n \in \mathbb N$, we define the \emph{iterated delay functions} $\sigma_n\colon D_n \to \mathbb R$ inductively by setting $D_0 = \mathbb R$ and $\sigma_0(t) = t$ for $t \in \mathbb R$ and defining, for $n \in \mathbb N^\ast$,
\begin{equation}
\label{eq:defi-Dn}
D_n = \{t \in \mathbb R_+ \suchthat t - \tau(t) \in D_{n-1}\}
\end{equation}
and
\begin{equation}
\label{eq:defi-sigman}
\sigma_n(t) = \sigma_{n-1}(t - \tau(t)) \qquad \text{ for } t \in D_n.
\end{equation}
\end{definition}

Note that $D_1 = \mathbb R_+$ and $\sigma_1(t) = t - \tau(t)$. Moreover, we may have $D_n = \emptyset$ for some $n \in \mathbb N$: this is the case, for instance, if the delay function $\tau$ is defined for $t \in \mathbb R_+$ by $\tau(t) = t + 1$, in which case $\sigma_1(t) = -1$ for every $t \in \mathbb R_+$ and hence $D_2 = \emptyset$. The empty set, however, is not necessarily attained by some $D_n$: for instance, if $\tau(t) = 1$ for every $t \in \mathbb R_+$, one immediately computes that $D_n = [n-1, +\infty)$ for every $n \in \mathbb N^\ast$.

We now provide immediate properties of iterated delay functions. The first one is an immediate consequence of \eqref{eq:defi-Dn}.

\begin{lemma}
\label{lemma:D_n-sigma1}
Let $\tau$ be a delay function, $t \in \mathbb R_+$, and $n \in \mathbb N$. Then $t \in D_{n+1}$ if and only if $\sigma_1(t) \in D_n$.
\end{lemma}

The next result shows that the sequence of sets $(D_n)_{n \in \mathbb N}$ is nonincreasing.

\begin{lemma}
\label{lemma:Dn-decreasing}
Let $\tau$ be a delay function. Then, for every $n \in \mathbb N^\ast$, we have $D_n \subset D_{n-1}$.
\end{lemma}

\begin{proof}
We prove the result by induction on $n$. The case $n = 1$ is trivially true. Now, let $n \in \mathbb N^\ast$ be such that $D_n \subset D_{n-1}$. Let $t \in D_{n+1}$ and note that, by definition, $t \in \mathbb R_+$ and $t - \tau(t) \in D_n$, and hence $t - \tau(t) \in D_{n-1}$, implying that $t \in D_n$. Thus $D_{n+1} \subset D_n$, yielding the conclusion.
\end{proof}

\begin{lemma}
\label{lemma:caract-Dn-sigman}
Let $\tau$ be a delay function. Then, for every $n \in \mathbb N$ and $k \in \llbracket 0, n\rrbracket$, we have
\begin{equation}
\label{eq:rec-Dn}
D_n = \{t \in D_k \suchthat \sigma_k(t) \in D_{n-k}\}
\end{equation}
and
\begin{equation}
\label{eq:rec-sigman}
\sigma_n(t) = \sigma_{n-k}(\sigma_k(t)) \qquad \text{ for every } t \in D_n.
\end{equation}
\end{lemma}

\begin{proof}
We prove the statement by strong induction on $n$. Notice first that the result is immediate for $n = 0$ and $n = 1$. Let $N \in \mathbb N^\ast$ be such that \eqref{eq:rec-Dn} and \eqref{eq:rec-sigman} hold for every $n \in \llbracket 0, N-1\rrbracket$ and $k \in \llbracket 0, n\rrbracket$. We notice that \eqref{eq:rec-Dn} and \eqref{eq:rec-sigman} trivially hold with $n$ replaced by $N$ and $k$ replaced by either $0$ or $N$. We then fix $k \in \llbracket 1, N-1\rrbracket$.

By \eqref{eq:defi-Dn}, we have
\[
D_{N} = \{t \in D_1 \suchthat \sigma_1(t) \in D_{N-1}\}.
\]
Using the induction assumption \eqref{eq:rec-Dn} with $n$ replaced by $N-1$ and $k$ replaced by $k-1$, we deduce that $\sigma_1(t) \in D_{N-1}$ if and only if $\sigma_1(t) \in D_{k-1}$ and $\sigma_{k-1}(\sigma_1(t)) \in D_{N - k}$, i.e.,
\[
D_{N} = \{t \in D_1 \suchthat \sigma_1(t) \in D_{k-1} \text{ and } \sigma_{k-1}(\sigma_1(t)) \in D_{N - k}\}.
\]
By \eqref{eq:defi-Dn}, we have $t \in D_k$ if and only if $t \in D_1$ and $\sigma_1(t) \in D_{k-1}$, and thus
\[
D_{N} = \{t \in D_k \suchthat \sigma_{k-1}(\sigma_1(t)) \in D_{N - k}\}.
\]
Since $k \in \llbracket 1, N-1\rrbracket$, the induction assumption \eqref{eq:rec-sigman} also shows that $\sigma_{k-1}(\sigma_1(t))$ for every $t \in D_k$, yielding
\[
D_{N} = \{t \in D_k \suchthat \sigma_{k}(t) \in D_{N - k}\},
\]
as required.

Finally, to obtain \eqref{eq:rec-sigman} with $n$ replaced by $N$, notice that, by definition, for every $t \in D_N$, we have $\sigma_N(t) = \sigma_{N-1}(\sigma_1(t))$. On the other hand, by the induction assumption \eqref{eq:rec-sigman} with $n$ replaced by $N-1$ and $k$ replaced by $k-1$, we deduce that $\sigma_{N-1}(s) = \sigma_{N-k}(\sigma_{k-1}(s))$ for every $s \in D_{N-1}$. Since $\sigma_1(t) \in D_{N-1}$ for every $t \in D_N$, we finally deduce that, for every $t \in D_N$, we have
\[
\sigma_{N}(t) = \sigma_{N-1}(\sigma_1(t)) = \sigma_{N-k}(\sigma_{k-1}(\sigma_1(t))) = \sigma_{N-k}(\sigma_k(t)),
\]
as required.
\end{proof}

\begin{lemma}
\label{lemma:intersect-Dn-empty}
Let $\tau$ be a delay function satisfying \ref{hypo:tau-infimum}. Then
\[\bigcap_{n \in \mathbb N} D_n = \emptyset.\]
\end{lemma}

\begin{proof}
Assume, to obtain a contradiction, that there exists $t \in \bigcap_{n \in \mathbb N} D_n$, and consider the sequence $(t_k)_{k \in \mathbb N}$ defined for $k \in \mathbb N$ by $t_k = \sigma_k(t)$. By Lemma~\ref{lemma:caract-Dn-sigman}, we deduce that $t_k \in \bigcap_{n \in \mathbb N} D_n$ for every $k \in \mathbb N$. In particular, $t_k \geq 0$ for every $k \in \mathbb N$. Since $t_k = \sigma_1(\sigma_{k-1}(t)) = t_{k-1} - \tau(t_{k-1})$ and $\tau$ takes values in $\mathbb R_+^\ast$, we also have that $(t_k)_{k \in \mathbb N}$ is a decreasing sequence.

Let $\delta = \inf_{s \in [0, t]} \tau(s) > 0$. Since $t_k \in [0, t]$ for every $k \in \mathbb N$, we deduce that $\tau(t_{k-1}) \geq \delta$ for every $k \in \mathbb N^\ast$, and thus $t_k \leq t_{k-1} - \delta$, which yields $t_k \leq t - k \delta$ for every $k \in \mathbb N$. This, however, contradicts the fact that $t_k \geq 0$ for every $k \in \mathbb N$, establishing the result.
\end{proof}

As an immediate consequence of Lemmas~\ref{lemma:Dn-decreasing} and \ref{lemma:intersect-Dn-empty}, we obtain the following result.

\begin{corollary}
\label{coro:integer-iterations}
Let $\tau$ be a delay function satisfying \ref{hypo:tau-infimum}. Then there exists a unique function $\mathbf n\colon \mathbb R \to \mathbb N$ such that, for every $t \in \mathbb R$, we have $t \in D_{\mathbf n(t)} \setminus D_{\mathbf n(t) + 1}$.
\end{corollary}

Our next result provides several properties of the function $\mathbf n$.

\begin{lemma}
\label{lemma:properties_n}
Let $\tau$ be a delay function satisfying \ref{hypo:tau-infimum} and $\mathbf n\colon \mathbb R \to \mathbb N$ be the function from Corollary~\ref{coro:integer-iterations}. Then, for every $t \in \mathbb R_+$, we have
\begin{enumerate}
\item\label{item:prop-n-nonzero} $\mathbf n(t) \geq 1$,
\item\label{item:n_sigma_1} $\mathbf n(\sigma_1(t)) = \mathbf n(t) - 1$, 
\item\label{item:sigma-k-nonnegative} $\sigma_k(t) \in \mathbb R_+$ for every $k \in \llbracket 0, \mathbf n(t) - 1\rrbracket$, and
\item\label{item:sigma-nt-negative} $\sigma_{\mathbf n(t)}(t) \in \lvert -h(0), 0)$.
\end{enumerate}
\end{lemma}

\begin{proof}
Since $t \in \mathbb R_+$ and $D_0 \setminus D_1 = \mathbb R_-^\ast$, we have $t \notin D_0 \setminus D_1$, and hence $\mathbf n(t) \neq 0$, implying \ref{item:prop-n-nonzero}.

To prove \ref{item:n_sigma_1}, note that, by Corollary~\ref{coro:integer-iterations}, we have
\begin{equation}
\label{eq:n_sigma_1}
\sigma_1(t) \in D_{\mathbf n(\sigma_1(t))} \setminus D_{\mathbf n(\sigma_1(t)) + 1}.
\end{equation}
Recall that, by Lemma~\ref{lemma:D_n-sigma1}, for every $n \in \mathbb N$, we have $t \in D_{n+1}$ if and only if $\sigma_1(t) \in D_n$. Hence, \eqref{eq:n_sigma_1} is equivalent to
\[
t \in D_{\mathbf n(\sigma_1(t)) + 1} \setminus D_{\mathbf n(\sigma_1(t)) + 2}.
\]
Since the function $\mathbf n$ from Corollary~\ref{coro:integer-iterations} is unique, we deduce that $\mathbf n(t) = \mathbf n(\sigma_1(t)) + 1$, yielding the conclusion.

For $k \in \llbracket 0, \mathbf n(t) - 1\rrbracket$, it follows from the definition of $\mathbf n$ and Lemma~\ref{lemma:Dn-decreasing} that $t \in D_{k+1}$. By \eqref{eq:rec-Dn}, $D_{k+1} = \{ s \in D_k \suchthat \sigma_k (s) \in D_1\}$, which implies that $\sigma_k (t) \in D_1 = \mathbb R_+$, proving \ref{item:sigma-k-nonnegative}. 

Let us finally prove \ref{item:sigma-nt-negative}. Note that $t \in D_{\mathbf n(t)}$, but $t \notin D_{\mathbf n(t) +1}$. By \eqref{eq:rec-Dn}, we have $D_{\mathbf n(t) + 1} = \{ s \in D_{\mathbf n(t)} \suchthat \sigma_{\mathbf n(t)} (s) \in D_1 \}.$ Therefore, $t \notin D_{\mathbf n(t) + 1}$ implies that $t \notin D_{\mathbf n(t)}$ or $\sigma_{\mathbf n(t)} (t) \notin D_1 = \mathbb R_+$. Since $t \in D_{\mathbf n(t)}$, then the only possibility is $\sigma_{\mathbf n(t)} (t) \notin D_1 = \mathbb R_+$. This implies that $\sigma_{\mathbf n(t)} (t) < 0$.

On the other hand, $h(0) = - \inf_{s \geq 0} \sigma_1 (s) = - \inf_{s \geq 0} (s - \tau(s)) \in (0, +\infty]$. Therefore, $\sigma_1 (s) \geq -h(0)$ for every $s \in \mathbb R_+$. By \ref{item:prop-n-nonzero}, we have $\mathbf n(t) \geq 1$, then, by \eqref{eq:rec-sigman},
\[\sigma_{\mathbf n(t)} (t) = \sigma_1 (\sigma_{\mathbf n(t) - 1} (t)) \geq - h(0),\]
proving the desired result.
\end{proof}

Another useful result regarding the function $\mathbf n$ is the following.

\begin{lemma}
\label{lemma:n-loc-bounded}
Let $\tau$ be a delay function satisfying \ref{hypo:tau-infimum} and $\mathbf n\colon \mathbb R \to \mathbb N$ be the function from Corollary~\ref{coro:integer-iterations}. Then, for every $T > 0$, there exists $N_T > 0$ such that $\mathbf n(t) \leq N_T$ for every $t \in (-\infty, T]$.
\end{lemma}

\begin{proof}
Let $T > 0$ and set $\delta = \inf_{t \in [0, T]} \tau(t)$. Note that $\delta > 0$ thanks to \ref{hypo:tau-infimum}.

We claim that, for every $k \in \mathbb N^\ast$, we have $D_k \cap [0, T] \subset [(k-1)\delta, T]$. Indeed, this is trivially verified for $k = 1$. Now, let $k \in \mathbb N^\ast$ bet such that $D_k \cap [0, T] \subset [(k-1)\delta, T]$ and take $t \in D_{k+1} \cap [0, T]$. By \eqref{eq:defi-Dn}, we have $\sigma_1(t) \in D_k$. Since $k \neq 0$, we have $D_k \subset \mathbb R_+$, hence $\sigma_1(t) \geq 0$. Moreover, $\sigma_1(t) = t - \tau(t) < t \leq T$, hence $\sigma_1(t) \in D_k \cap [0, T] \subset [(k-1)\delta, T]$. Thus $t = \sigma_1(t) + \tau(t) \geq (k-1)\delta + \delta = k \delta$, proving that $t \in [k\delta, T]$. Hence, by induction, $D_k \cap [0, T] \subset [(k-1)\delta, T]$ for every $k \in \mathbb N^\ast$.

Let $t \in (-\infty, T]$. If $t < 0$, we have $\mathbf n(t) = 0$. Otherwise, for $t \in [0, T]$, we have $t \in D_{\mathbf n(t)} \cap [0, T] \subset [(\mathbf n(t) - 1)\delta, T]$, and thus $t \geq (\mathbf n(t) - 1)\delta$, implying that $\mathbf n(t) \leq \frac{t}{\delta} + 1 \leq \frac{T}{\delta} + 1$. This implies the conclusion with $N_T = \frac{T}{\delta} + 1$.
\end{proof}

Our next result provides an expression for a solution $x$ of \eqref{eq:diff-eqn-N1-time} in terms of the $k$-th power of $A$ and $\sigma_k$.

\begin{lemma}\label{lemma-solution} Let $\tau$ be a delay function satisfying \ref{hypo:tau-infimum} and $\mathbf n\colon \mathbb R \to \mathbb N$ be the function from Corollary~\ref{coro:integer-iterations}. If $x$ is a solution of \eqref{eq:diff-eqn-N1-time}, then 
\begin{equation}\label{equation-lemma}
x(t) = A^k x(\sigma_k (t)), \qquad \text{for every } t \in \mathbb R_+ \text{ and } k \in \llbracket 0, \mathbf n(t)\rrbracket.
\end{equation}
\end{lemma}

\begin{proof}
Fix $t \in \mathbb R_+$. We prove \eqref{equation-lemma} by induction on $k$. Consider first $k =0$, then $A^0 = \Id$ and $\sigma_0 = \id$, thus clearly we get that \eqref{equation-lemma} is satisfied.

Now, let $k \in \llbracket 0, \mathbf n(t) - 1\rrbracket$ be such that the equation \eqref{equation-lemma} is fulfilled, that is,
\begin{equation}\label{equation-induction-sigma-k}
x(t) = A^k x(\sigma_k(t)).
\end{equation}
By Lemma~\ref{lemma:properties_n}\ref{item:sigma-k-nonnegative}, we have $\sigma_k(t) \geq 0$. On the other hand, since $x$ is a solution of \eqref{eq:diff-eqn-N1-time}, we have
\begin{equation*}
x(\sigma_k (t)) = A x(\sigma_k(t) - \tau(\sigma_k (t)))) = A x(\sigma_1 \circ \sigma_k(t)) = A x(\sigma_{k+1} (t)).
\end{equation*}
Combining with \eqref{equation-induction-sigma-k}, we get
\[x(t) = A^{k+1} x(\sigma_{k+1} (t)),\]
proving our desired result.
\end{proof}

Taking $k = \mathbf n(t)$ in \eqref{equation-lemma} and noticing that $\sigma_{\mathbf n(t)} \in \lvert -h(0), 0)$ thanks to Lem\-ma~\ref{lemma:properties_n}\ref{item:sigma-nt-negative}, we obtain our main result of this section, containing a representation formula for solutions of \eqref{eq:diff-eqn-N1-time} in terms of the initial condition $x_0$, the matrix $A$, and the delay $\tau$.

\begin{theorem}\label{theorem-1} Let $\tau$ be a delay function satisfying \ref{hypo:tau-infimum} and $\mathbf n\colon \mathbb R \to \mathbb N$ be the function from Corollary~\ref{coro:integer-iterations}. If $x$ is a solution of \eqref{eq:diff-eqn-N1-time} with initial condition $x_0$, then
\begin{equation}
\label{eq:explicit-formula}
x(t) = A^{\mathbf n(t)} x_0(\sigma_{\mathbf n(t)}(t)), \qquad \text{for every } t \in \mathbb R_+.
\end{equation}
\end{theorem}

Note that the right-hand side of \eqref{eq:explicit-formula} is a function of $t$ uniquely determined from $x_0$, $A$, and the delay $\tau$ satisfying \ref{hypo:tau-infimum}. In particular, we deduce at once the following uniqueness result.

\begin{corollary}
\label{coro:uniqueness}
Let $\tau$ be a delay function satisfying \ref{hypo:tau-infimum}, $h$ be its associated largest delay function, and $x_0\colon \lvert -h(0), 0) \to \mathbb R^d$. Then \eqref{eq:diff-eqn-N1-time} admits at most one solution with initial condition $x_0$.
\end{corollary}

\begin{remark}
\label{remk:H2}
Assumption~\ref{hypo:tau-infimum} is important to ensure uniqueness of solutions in Corollary~\ref{coro:uniqueness}. Consider, for instance, the delay function $\tau$ defined by $\tau(0) = 1$ and $\tau(t) = \left(1 - \frac{1}{e}\right)t$ for $t > 0$. In particular, \eqref{eq:diff-eqn-N1-time} reads $x(t) = A x(\frac{t}{e})$ for $t > 0$. Assume that $A$ is not nilpotent. Let $x_0\colon [-1, 0) \to \mathbb R^d$ be a given initial condition. Let $\lambda \in \mathbb C \setminus \{0\}$ be a nonzero eigenvector of $A$, $v \in \mathbb C^d \setminus \{0\}$ be a corresponding eigenvector, $\alpha \in \mathbb C$ be such that $e^\alpha = \lambda$, and $\rho \in \mathbb C$. Let $x\colon [-1, +\infty) \to \mathbb R^d$ be the function defined by
\[
x(t) = \begin{dcases*}
x_0(t) & if $t \in [-1, 0)$, \\
A x_0(-1) & if $t = 0$, \\
\Real\left(\rho e^{\alpha \ln t} v\right) & if $t > 0$.
\end{dcases*}
\]
It is immediate to verify that $x$ is a solution of \eqref{eq:diff-eqn-N1-time} with initial condition $x_0$ and, since $\rho \in \mathbb C$ is arbitrary, \eqref{eq:diff-eqn-N1-time} admits infinitely many solutions with the same initial condition $x_0$.

Assumption~\ref{hypo:tau-infimum}, however, is not necessary for the uniqueness of solutions of \eqref{eq:diff-eqn-N1-time}. Indeed, consider the delay function $\tau$ given by $\tau(0) = 1$ and $\tau(t) = t$ for $t > 0$. In this case, \eqref{eq:diff-eqn-N1-time} reads $x(0) = A x(-1)$ and $x(t) = A x(0)$ for $t > 0$. It is then immediate to verify that, given an initial condition $x_0\colon [-1, 0) \to \mathbb R^d$, \eqref{eq:diff-eqn-N1-time} admits a unique solution with initial condition $x_0$, which is given by $x(t) = x_0(t)$ for $t \in [-1, 0)$, $x(0) = A x_0(-1)$, and $x(t) = A^2 x_0(-1)$ for $t > 0$.
\end{remark}

Since continuous, measurable, and regulated solutions of \eqref{eq:diff-eqn-N1-time} are particular cases of solutions of \eqref{eq:diff-eqn-N1-time}, the representation formula \eqref{eq:explicit-formula} from Theorem~\ref{theorem-1} also applies to such notions of solutions, and Corollary~\ref{coro:uniqueness} also implies uniqueness in such classes of solutions. It is also interesting to consider the counterparts of Theorem~\ref{theorem-1} and Corollary~\ref{coro:uniqueness} for a.e.\ solutions of \eqref{eq:diff-eqn-N1-time}. This is the topic of our next result.

\begin{theorem}
\label{thm:explicit-ae}
Let $\tau$ be a delay function satisfying \ref{hypo:tau-infimum} and \ref{hypo:null} and $\mathbf n\colon \mathbb R \to \mathbb N$ be the function from Corollary~\ref{coro:integer-iterations}. If $x$ is an a.e.\ solution of \eqref{eq:diff-eqn-N1-time} with initial condition equal to $x_0$ a.e., then
\begin{equation}
\label{eq:explicit-formula-ae}
x(t) = A^{\mathbf n(t)} x_0(\sigma_{\mathbf n(t)}(t)), \qquad \text{for almost every } t \in \mathbb R_+.
\end{equation}
In particular, one has uniqueness of solutions of \eqref{eq:diff-eqn-N1-time} with respect to a.e.\ equality: if $x_1, x_2\colon \lvert -h(0), +\infty) \to \mathbb R^d$ are a.e.\ solutions of \eqref{eq:diff-eqn-N1-time} with $x_1(t) = x_2(t)$ for a.e.\ $t \in \lvert -h(0), 0)$, then $x_1(t) = x_2(t)$ for a.e.\ $t \in \lvert -h(0), +\infty)$.
\end{theorem}

\begin{proof}
Let
\begin{equation}
\label{eq:def-N0-N1}
\begin{aligned}
N_0 & = \{t \in \lvert -h(0), 0) \suchthat x(t) \neq x_0(t)\}, \\
N_1 & = \{t \in \mathbb R_+ \suchthat x(t) \neq A x(t - \tau(t))\}, \\
N & = N_0 \cup N_1, \\
M & = \{t \in \lvert -h(0), +\infty) \suchthat x(t) \neq A^{\mathbf n(t)} x_0(\sigma_{\mathbf n(t)}(t))\},
\end{aligned}
\end{equation}
and note that, by assumption, $\mathcal L(N_0) = \mathcal L(N_1) = \mathcal L(N) = 0$. Showing \eqref{eq:explicit-formula-ae} amounts to showing that $\mathcal L(M) = 0$. We claim that
\begin{equation}
\label{eq:M-inclusion}
M \subset \bigcup_{k = 0}^\infty \sigma_k^{-1}(N).
\end{equation}
Indeed, assume that $t \in \mathbb R$ is such that $t \notin \bigcup_{k = 0}^\infty \sigma_k^{-1}(N)$. Then, for every $k \in \mathbb N$, we have $t \notin D_k$ or $\sigma_k(t) \notin N$. In particular, for every $k \in \llbracket 0, \mathbf n(t)\rrbracket$, we have $\sigma_k(t) \notin N$. Using Lemma~\ref{lemma:properties_n}\ref{item:sigma-k-nonnegative} and \ref{item:sigma-nt-negative}, \eqref{eq:rec-sigman}, and \eqref{eq:def-N0-N1}, we deduce that $x(\sigma_k(t)) = A x(\sigma_{k+1}(t))$ for every $k \in \llbracket 0, \mathbf n(t) - 1\rrbracket$ and $x(\sigma_{\mathbf n(t)}(t)) = x_0(\sigma_{\mathbf n(t)}(t))$. An immediate inductive argument shows that $x(t) = x(\sigma_0(t)) = A^{\mathbf n(t)} x(\sigma_{\mathbf n(t)}(t)) = A^{\mathbf n(t)} x_0(\sigma_{\mathbf n(t)}(t))$, showing that $t \notin M$, and hence concluding the proof of \eqref{eq:M-inclusion}.

Notice now that, by \eqref{eq:rec-sigman}, we have $\sigma_k^{-1}(N) = \sigma_{k-1}^{-1}(\sigma_1^{-1}(N))$ for every $k \in \mathbb N^\ast$, and thus, using \ref{hypo:null} and an immediate inductive argument, we deduce that $\mathcal L(\sigma_k^{-1}(N)) = 0$ for every $k \in \mathbb N^\ast$. Hence, \eqref{eq:M-inclusion} implies that $\mathcal L(N) = 0$, yielding the conclusion.

To prove uniqueness, note that, since \eqref{eq:diff-eqn-N1-time} is linear, if $x_1, x_2\colon \lvert -h(0), +\infty) \to \mathbb R^d$ are a.e.\ solutions of \eqref{eq:diff-eqn-N1-time} with $x_1(t) = x_2(t)$ for a.e.\ $t \in \lvert -h(0), 0)$, then $x_1 - x_2$ is an a.e.\ solution of \eqref{eq:diff-eqn-N1-time} with initial condition equal a.e.\ to $0$. Hence, by \eqref{eq:explicit-formula-ae}, $x_1(t) - x_2(t) = 0$ for a.e.\ $t \in \lvert -h(0), +\infty)$, yielding the conclusion.
\end{proof}

\subsection{Existence of solutions}
\label{sec:exist}

We now turn to the question of existence of solutions of \eqref{eq:diff-eqn-N1-time}. Our first result is the following consequence of the representation formula for the solutions from Theorem~\ref{theorem-1}.

\begin{theorem}
\label{thm:exist-general}
Let $\tau$ be a delay function satisfying \ref{hypo:tau-infimum}, $h$ be its associated largest delay function, and $x_0\colon \lvert -h(0), 0) \to \mathbb R^d$. Then there exists a unique solution $x\colon \lvert -h(0), +\infty) \to \mathbb R^d$ of \eqref{eq:diff-eqn-N1-time} with initial condition $x_0$.
\end{theorem}

\begin{proof}
Uniqueness of the solution was already established in Corollary~\ref{coro:uniqueness}, and so we are left to prove existence. Let $x\colon \lvert -h(0), +\infty) \to \mathbb R^d$ be defined by
\begin{equation}
\label{eq:def-x-sol}
x(t) = \begin{dcases*}
x_0(t) & if $t \in \lvert -h(0), 0)$, \\
A^{\mathbf n(t)} x_0(\sigma_{\mathbf n(t)}(t)) & if $t \geq 0$,
\end{dcases*}
\end{equation}
where $\mathbf n\colon \mathbb R \to \mathbb N$ is the function from Corollary~\ref{coro:integer-iterations}. Note that, thanks to Lem\-ma~\ref{lemma:properties_n}\ref{item:sigma-nt-negative}, we have $\sigma_{\mathbf n(t)}(t) \in \lvert -h(0), 0)$ for every $t \geq 0$, and thus $x$ is well-defined. We want to prove that $x$ satisfies $x(t) = A x(t - \tau(t))$ for every $t \geq 0$.

Let $t \geq 0$ and assume first that $\sigma_1(t) < 0$. Then $\mathbf n(t) = 1$, and we have $x(\sigma_1(t)) = x_0(\sigma_{\mathbf n(t)}(t))$. Thus, using \eqref{eq:def-x-sol}, we have
\[
x(t) = A^{\mathbf n(t)} x_0(\sigma_{\mathbf n(t)}(t)) = A x(\sigma_1(t)) = A x(t - \tau(t)).
\]
Consider now the case $\sigma_1(t) \geq 0$. Using \eqref{eq:def-x-sol} and Lemmas~\ref{lemma:caract-Dn-sigman} and \ref{lemma:properties_n}\ref{item:n_sigma_1}, we have
\begin{align*}
x(\sigma_1(t)) & = A^{\mathbf n(\sigma_1(t))} x_0(\sigma_{\mathbf n(\sigma_1(t))} (\sigma_1(t))) \\
& = A^{\mathbf n(t) - 1} x_0(\sigma_{\mathbf n(t) - 1} (\sigma_1(t))) \\
& = A^{\mathbf n(t) - 1} x_0(\sigma_{\mathbf n(t)}(t)).
\end{align*}
Hence
\[
A x(t - \tau(t)) = A x(\sigma_1(t)) = A^{\mathbf n(t)} x_0(\sigma_{\mathbf n(t)}(t)) = x(t),
\]
where the last equality follows from \eqref{eq:def-x-sol}.
\end{proof}

\begin{remark}
As highlighted in Remark~\ref{remk:H2}, assumption~\ref{hypo:tau-infimum} is important in Theorem~\ref{thm:exist-general} to deduce uniqueness of solutions, but it is not needed for getting existence of solutions. Without \ref{hypo:tau-infimum}, one cannot guarantee that the result of Lemma~\ref{lemma:intersect-Dn-empty} is true. However, setting $D_\infty = \bigcap_{n \in \mathbb N} D_n \subset \mathbb R_+$, we still get the conclusion of Corollary~\ref{coro:integer-iterations} with a function $\mathbf n$ defined on $\mathbb R \setminus D_\infty$ instead of $\mathbb R$. The results of Lemmas~\ref{lemma:properties_n} and \ref{lemma-solution} and Theorem~\ref{theorem-1} remain true without assuming \ref{hypo:tau-infimum} after replacing $t \in \mathbb R_+$ by $t \in \mathbb R_+ \setminus D_\infty$ in their statements and proofs. Then, given an initial condition $x_0\colon \lvert -h(0), 0) \to \mathbb R^d$, we define $x\colon \lvert -h(0), +\infty) \to \mathbb R^d$ by
\[
x(t) = \begin{dcases*}
x_0(t) & if $t \in \lvert -h(0), 0)$, \\
A^{\mathbf n(t)} x_0(\sigma_{\mathbf n(t)}(t)) & if $t \in \mathbb R_+ \setminus D_\infty$, \\
0 & if $t \in D_\infty$.
\end{dcases*}
\]
By the same argument used in the proof of Theorem~\ref{thm:exist-general}, we obtain that $x(t) = A x(t - \tau(t))$ for every $t \in \mathbb R_+ \setminus D_\infty$. In addition, if $t \in D_\infty$, then $t - \tau(t) = \sigma_1(t) \in D_\infty$, and thus the identity $x(t) = A x(t - \tau(t))$ is verified for every $t \in D_\infty$. This shows that $x$ is a solution of \eqref{eq:diff-eqn-N1-time} with initial condition $x_0$, and hence proves the existence assertion of Theorem~\ref{thm:exist-general} without the need of \ref{hypo:tau-infimum}.

In this paper, we are only interested in cases where the solution of \eqref{eq:diff-eqn-N1-time} is unique, and hence we shall always assume \ref{hypo:tau-infimum} in the sequel. However, by arguments similar to those in this remark, some of our next results on existence of solutions can also be generalized to settings where \ref{hypo:tau-infimum} is not assumed to hold true.
\end{remark}

Our next result considers the existence of continuous solutions of \eqref{eq:diff-eqn-N1-time}.

\begin{theorem}
\label{thm:exist-continuous}
Let $\tau$ be a delay function satisfying \ref{hypo:tau-infimum} and \ref{hypo:tau-cont}, $h$ be its associated largest delay function, and $x_0 \in \mathsf C_0$. Then there exists a unique continuous solution $x \in C(\lvert -h(0), +\infty), \mathbb R^d)$ of \eqref{eq:diff-eqn-N1-time} with initial condition $x_0$. In addition, $x_t \in \mathsf C_t$ for every $t \in \mathbb R_+$.
\end{theorem}

\begin{proof}
By Theorem~\ref{thm:exist-general}, \eqref{eq:diff-eqn-N1-time} admits a unique solution $x\colon \lvert -h(0), +\infty) \to \mathbb R^d$ with initial condition $x_0$, and we are thus only left to show that such a solution is continuous. To do so, we will prove that $x$ is continuous in $\lvert -h(0), T]$ for every $T > 0$.

Fix $T > 0$ and let $\widetilde\delta = \frac{1}{2}\inf_{t \in [0, T]} \tau(t)$. Note that, by \ref{hypo:tau-infimum}, we have $\widetilde\delta > 0$. Set $\delta = \frac{T}{\ceil{T / \widetilde\delta}}$ and note that $0 < \delta \leq \widetilde\delta$. By definition, $\tau(t) > \delta$ for every $t \in [0, T]$ and $T/\delta$ is a positive integer, which we denote by $M$.

We prove by induction that, for every $k \in \llbracket 0, M\rrbracket$, $x$ is continuous on $\lvert -h(0), k \delta]$, which will conclude the proof of the result since $\lvert -h(0), M \delta] = \lvert -h(0), T]$. This is trivially true for $k = 0$ since $x$ coincides with $x_0$ in $\lvert -h(0), 0)$ by definition, their values also coincide at $t = 0$ since $x(0) = A x_0(-\tau(0)) = x_0(0)$, and $x_0$ is continuous by assumption.

Let $k \in \llbracket 0, M-1\rrbracket$ be such that $x$ is continuous on $\lvert -h(0), k \delta]$. Define $\widehat x = x \rvert_{\lvert -h(0), k \delta]}$ and $y = x \rvert_{[k \delta, (k+1) \delta]}$ and notice that, thanks to the definition of $\delta$, for every $t \in [k \delta, (k+1) \delta]$, we have $y(t) = A \widehat x(t - \tau(t))$. Since $\widehat x$ and $\tau$ are continuous, $y$ is thus continuous as well. Hence $x$ is continuous on $\lvert -h(0), (k+1) \delta]$, as required.

Finally, for every $t \in \mathbb R_+$, we have $x_t(0) = x(t) = A x(t - \tau(t)) = A x_t(-\tau(t))$, showing that $x_t \in \mathsf C_t$.
\end{proof}

It is immediate to verify the importance of \ref{hypo:tau-cont} in Theorem~\ref{thm:exist-continuous}, as one can easily come up with examples of discontinuous delay functions yielding a discontinuous solution of \eqref{eq:diff-eqn-N1-time} for some initial condition in $\mathsf C_0$. A natural question is whether \ref{hypo:tau-cont} is necessary, and our next example provides a negative answer to that question.

\begin{example}
\label{expl:exist-continuous}
Let $\tau\colon \mathbb R_+ \to \mathbb R_+^\ast$ be the delay function defined by
\[
\tau(t) = \begin{dcases*}
t + 1 & if $t \in [0, 1)$, \\
2 & if $t \in [1, 2)$, \\
t - 1 & if $t \in [2, +\infty)$.
\end{dcases*}
\]
Note that $\tau$ is discontinuous at $t = 2$. Let $x_0 \in \mathsf C_0$ and $x\colon [-1, +\infty) \to \mathbb R^d$ be the unique solution of \eqref{eq:diff-eqn-N1-time} with initial condition $x_0$ provided by Theorem~\ref{thm:exist-general}. Then $x$ is continuous in $[-1, +\infty)$. Indeed, a straightforward computation from \eqref{eq:diff-eqn-N1-time} shows that
\begin{equation}
\label{eq:x-continuous}
x(t) = \begin{dcases*}
x_0(t) & if $t \in [-1, 0)$, \\
A x_0(-1) & if $t \in [0, 1)$, \\
A x_0(t - 2) & if $t \in [1, 2)$, \\
A^2 x_0(-1) & if $t \in [2, +\infty)$,
\end{dcases*}
\end{equation}
and it is immediate to verify that $x$ is continuous using the compatibility condition $x_0(0) = A x_0(-1)$, since $x_0 \in \mathsf C_0$.
\end{example}

Our next result provides conditions for the existence and uniqueness of regulated solutions of \eqref{eq:diff-eqn-N1-time}.

\begin{theorem}
\label{thm:exist-regulated}
Let $\tau$ be a delay function satisfying \ref{hypo:tau-infimum} and \ref{hypo:regulated}, $h$ be its associated largest delay function, and $x_0 \in \mathsf G_0$. Then there exists a unique regulated solution $x \in G(\lvert -h(0), +\infty), \mathbb R^d)$ of \eqref{eq:diff-eqn-N1-time} with initial condition $x_0$. In addition, $x_t \in \mathsf G_t$ for every $t \in \mathbb R_+$.
\end{theorem}

\begin{proof}
As in the proof of Theorem~\ref{thm:exist-continuous}, we only need to show that unique solution $x\colon \lvert -h(0), +\infty) \to \mathbb R^d$ of \eqref{eq:diff-eqn-N1-time} with initial condition $x_0$ provided by Theorem~\ref{thm:exist-general} is regulated. The argument is very similar to that in the proof of Theorem~\ref{thm:exist-continuous}: we set $T > 0$, $\delta > 0$, and $M \in \mathbb N^\ast$ as in that proof, and we remark that the restriction of $x$ to $\lvert -h(0), 0]$ is regulated by assumption. If $k \in \llbracket 0, M-1\rrbracket$ is such that the restriction of $x$ to $\lvert -h(0), k \delta]$ is regulated, we define $\widehat x$ and $y$ as in the proof of Theorem~\ref{thm:exist-continuous}. For every $t \in [k\delta, (k+1)\delta]$, we have $y(t) = A \widehat x(t - \tau(t))$. Since $\widehat x$ is regulated and $t \mapsto t - \tau(t)$ is well-regulated, we deduce from Theorem~\ref{thm:well-regulated} that $y$ is regulated, yielding that the restriction of $x$ to $\lvert -h(0), (k+1)\delta]$ is regulated, as required.

Finally, for every $t \in \mathbb R_+$, we have $x_t(0) = x(t) = A x(t - \tau(t)) = A x_t(-\tau(t))$, showing that $x_t \in \mathsf G_t$.
\end{proof}

Assumption~\ref{hypo:regulated} is important in Theorem~\ref{thm:exist-regulated} to guarantee regulatedness of the solution $x$ by using Theorem~\ref{thm:well-regulated}. Note that, due to Remark~\ref{remk:sum-not-well-regulated}, the well-regulatedness assumption cannot be done in $\tau$, since this is not sufficient to guarantee the well-regulatedness of $t \mapsto t - \tau(t)$. Even though \ref{hypo:regulated} is important, such an assumption is not necessary, since one can construct an example similar in spirit to Example~\ref{expl:exist-continuous}, as shown below.

\begin{example}
Let $\varphi\colon [1, +\infty) \to [0, 1)$ be a continuous function that is not well-regulated (such a function can be constructed, for instance, from a suitable translation of the function from Remark~\ref{remk:cinfty-not-well-regulated}). We define the delay function $\tau\colon \mathbb R_+ \to \mathbb R_+^\ast$ by
\[
\tau(t) = \begin{dcases*}
t + 1 & if $t \in [0, 1)$, \\
t - \varphi(t) & if $t \in [1, \infty)$.
\end{dcases*}
\]
Note that $\tau$ does not satisfy \ref{hypo:regulated}. Let $x_0 \in \mathsf G_0$ and $x\colon [-1, +\infty) \to \mathbb R^d$ be the unique solution of \eqref{eq:diff-eqn-N1-time} with initial condition $x_0$ provided by Theorem~\ref{thm:exist-general}. A straightforward computation from \eqref{eq:diff-eqn-N1-time} shows that
\begin{equation*}
x(t) = \begin{dcases*}
x_0(t) & if $t \in [-1, 0)$, \\
A x_0(-1) & if $t \in [0, 1)$, \\
A x(\varphi(t)) = A^2 x_0(-1) & if $t \in [1, +\infty)$,
\end{dcases*}
\end{equation*}
where we use the fact that $\varphi$ takes values in $[0, 1)$. Hence $x$ is regulated in $[-1, +\infty)$.
\end{example}

We now turn to the existence of measurable solutions of \eqref{eq:diff-eqn-N1-time}.

\begin{theorem}
\label{thm:exist-measurable}
Let $\tau$ be a delay function satisfying \ref{hypo:tau-infimum} and \ref{hypo:tau-measurable}, $h$ be its associated largest delay function, and $x_0\colon \lvert -h(0), 0) \to \mathbb R^d$ be $(\mathfrak L, \mathfrak B)$-measurable. Then there exists a unique measurable solution $x\colon \lvert -h(0), +\infty) \to \mathbb R^d$ of \eqref{eq:diff-eqn-N1-time} with initial condition $x_0$.
\end{theorem}

\begin{proof}
As in the proof of Theorems~\ref{thm:exist-continuous} and \ref{thm:exist-regulated}, we only need to show that unique solution $x\colon \lvert -h(0), +\infty) \to \mathbb R^d$ of \eqref{eq:diff-eqn-N1-time} with initial condition $x_0$ provided by Theorem~\ref{thm:exist-general} is $(\mathfrak L, \mathfrak B)$-measurable. We proceed by an inductive argument similar to those of the proofs of Theorems~\ref{thm:exist-continuous} and \ref{thm:exist-regulated}.

Let $T > 0$, $\delta > 0$, and $M \in \mathbb N^\ast$ be as in those proofs. Clearly, the restriction of $x$ to $\lvert -h(0), 0]$ is $(\mathfrak L, \mathfrak B)$-measurable, since it coincides with $x_0$ in $\lvert -h(0), 0)$ and $x_0$ is $(\mathfrak L, \mathfrak B)$-measurable.

Now, let $k \in \llbracket 0, M-1\rrbracket$ be such that the restriction of $x$ to $\lvert -h(0), k \delta]$ is $(\mathfrak L, \mathfrak B)$-measurable and define $\widehat x$ and $y$ as in the proofs of Theorems~\ref{thm:exist-continuous} and \ref{thm:exist-regulated}. For every $t \in [k\delta, (k+1)\delta]$, we have $y(t) = A \widehat x(t - \tau(t))$. Since $\widehat x$ is $(\mathfrak L, \mathfrak B)$-measurable and $t \mapsto t - \tau(t)$ is $(\mathfrak L, \mathfrak L)$-measurable, we deduce that $y$ is $(\mathfrak L, \mathfrak B)$-measurable, yielding the $(\mathfrak L, \mathfrak B)$-measurability of the restriction of $x$ to $\lvert -h(0), (k+1)\delta]$, as required.
\end{proof}

The counterpart in Theorem~\ref{thm:exist-measurable} of assumption \ref{hypo:tau-cont} from Theorem~\ref{thm:exist-continuous} and of assumption \ref{hypo:regulated} from Theorem~\ref{thm:exist-regulated} is \ref{hypo:tau-measurable}. Similarly to what was shown in Example~\ref{expl:exist-continuous}, we provide now an example showing that \ref{hypo:tau-measurable} is not necessary for the existence of measurable solutions.

\begin{example}
\label{expl:exist-measurable}
Let $P \subset [2, +\infty)$ be a set that is not Lebesgue measurable. Let $\tau\colon \mathbb R_+ \to \mathbb R_+^\ast$ be the delay function defined by
\[
\tau(t) = \begin{dcases*}
t + 1 & if $t \in [0, 1)$, \\
2 & if $t \in [1, 2)$, \\
t - \frac{2}{3} + \frac{1}{3} \mathbbm 1_P(t) & if $t \in [2, +\infty)$,
\end{dcases*}
\]
where $\mathbbm 1_P$ is the indicator function of $P$. Since $P \notin \mathfrak L$, the function $\tau$ is not $(\mathfrak L, \mathfrak L)$-measurable. Let $x_0\colon [-1, 0) \to \mathbb R^d$ be a $(\mathfrak L, \mathfrak B)$-measurable function and $x\colon [-1, +\infty) \to \mathbb R^d$ be the unique solution of \eqref{eq:diff-eqn-N1-time} with initial condition $x_0$ provided by Theorem~\ref{thm:exist-general}. Then $x$ is $(\mathfrak L, \mathfrak B)$-measurable in $[-1, +\infty)$. Indeed, similarly to Example~\ref{expl:exist-continuous}, $x$ satisfies \eqref{eq:x-continuous} for $t \in [-1, 2)$. For $t \geq 2$, we have from \eqref{eq:diff-eqn-N1-time} that $x(t) = A x(\frac{2}{3} - \frac{1}{3} \mathbbm 1_P(t))$ and, since $\frac{2}{3} - \frac{1}{3} \mathbbm 1_P(t) \in [0, 1)$ for every $t \geq 2$, we have $x(t) = A^2 x_0(-1)$. Hence $x$ satisfies \eqref{eq:x-continuous} for every $t \in [-1, +\infty)$, and it is immediate to verify that the right-hand side of \eqref{eq:x-continuous} defines a $(\mathfrak L, \mathfrak B)$-measurable function of $t$.
\end{example}

Our next result concerns existence and uniqueness of a.e.\ solutions.

\begin{theorem}
\label{thm:exist-ae}
Let $\tau$ be a delay function satisfying \ref{hypo:tau-infimum} and \ref{hypo:null}, $h$ be its associated largest delay function, and $x_0\colon \lvert -h(0), 0) \to  \mathbb R^d$. Then the unique solution $x\colon \lvert -h(0), +\infty) \to \mathbb R^d$ of \eqref{eq:diff-eqn-N1-time} with initial condition $x_0$ from Theorem~\ref{thm:exist-general} is also an a.e.\ solution of \eqref{eq:diff-eqn-N1-time} with initial condition $x_0$. In addition, such an a.e.\ solution is unique with respect to a.e.\ equality (in the sense of Theorem~\ref{thm:explicit-ae}).
\end{theorem}

\begin{proof}
Notice first that the assertion on uniqueness follows from Theorem~\ref{thm:explicit-ae}.

Let $x\colon \lvert -h(0), +\infty) \allowbreak \to \mathbb R^d$ be the solution of \eqref{eq:diff-eqn-N1-time} with initial condition $x_0$, whose existence and uniqueness is asserted in Theorem~\ref{thm:exist-general}. Let us prove that $x$ is also an a.e.\ solution of \eqref{eq:diff-eqn-N1-time}. Let $\widetilde x\colon \lvert -h(0), +\infty) \to \mathbb R^d$ be such that $\widetilde x(t) = x(t)$ for a.e.\ $t \in \lvert -h(0), +\infty)$. Let $N_0 = \{t \in \lvert -h(0), +\infty) \suchthat x(t) \neq \widetilde x(t)\}$ and notice that $x(t - \tau(t)) \neq \widetilde x(t - \tau(t))$ if and only if $t \in \sigma_1^{-1}(N_0)$. Let $N = N_0 \cup \sigma_1^{-1}(N_0)$ and note that, thanks to \ref{hypo:null}, $\mathcal L(N) = 0$. For every $t \in \mathbb R_+ \setminus N$, we have
\[
\widetilde x(t) = x(t) = A x(t - \tau(t)) = A \widetilde x(t - \tau(t)).
\]
Hence $x$ is an a.e.\ solution of \eqref{eq:diff-eqn-N1-time} with initial condition $x_0$.
\end{proof}

It turns out that \ref{hypo:null} is also a necessary condition for the existence of a.e.\ solutions of \eqref{eq:diff-eqn-N1-time} in the nontrivial case where $A \neq 0$, as detailed in the next result.

\begin{proposition}
\label{prop:h2-necessary}
Let $\tau$ be a delay function and assume that $A \neq 0$ and there exists an a.e.\ solution of \eqref{eq:diff-eqn-N1-time}. Then $\tau$ satisfies \ref{hypo:null}.
\end{proposition}

\begin{proof}
Fix an a.e.\ solution $x\colon \lvert -h(0), +\infty) \to \mathbb R^d$ of \eqref{eq:diff-eqn-N1-time}. Let $v \in \mathbb R^d \setminus \{0\}$ be such that $A v \neq 0$. Since the image of $\sigma_1$ is included in $\lvert -h(0), +\infty)$, it suffices to show that $\mathcal L(\sigma_1^{-1}(N)) = 0$ for every Lebesgue-measurable set $N \subset \lvert -h(0), +\infty)$ such that $\mathcal L(N) = 0$.

Let $N \subset \lvert -h(0), +\infty)$ be Lebesgue measurable with $\mathcal L(N) = 0$. Let $\widetilde x\colon \lvert -h(0),\allowbreak +\infty) \to \mathbb R^d$ be given by $\widetilde x(t) = x(t) + v$ if $t \in N$ and $\widetilde x(t) = x(t)$ otherwise. Hence $\widetilde x$ is equal a.e.\ to $x$ and, since $x$ is an a.e.\ solution of \eqref{eq:diff-eqn-N1-time}, there exist Lebesgue-measurable sets $N_0$ and $\widetilde N_0$ included in $\mathbb R_+$ with $\mathcal L(N_0) = \mathcal L(\widetilde N_0) = 0$ and such that
\begin{equation*}
\begin{aligned}
\widetilde x(t) & = A \widetilde x(\sigma_1(t)) & \qquad & \text{ for every } t \in \mathbb R_+ \setminus \widetilde N_0, \\
x(t) & = A x(\sigma_1(t)) & & \text{ for every } t \in \mathbb R_+ \setminus N_0. \\
\end{aligned}
\end{equation*}
For every $t \in \sigma_1^{-1}(N) \setminus (N_0 \cup \widetilde N_0)$, we then have
\[
\widetilde x(t) = A \widetilde x(\sigma_1(t)) = A x(\sigma_1(t)) + A v = x(t) + A v \neq x(t).
\]
Since $\widetilde x$ is equal a.e.\ to $x$, we deduce that $\mathcal L(\sigma_1^{-1}(N) \setminus (N_0 \cup \widetilde N_0)) = 0$. In addition,
\[
\sigma_1^{-1}(N) = \left(\sigma_1^{-1}(N) \setminus (N_0 \cup \widetilde N_0)\right) \cup \left(\sigma_1^{-1}(N) \cap (N_0 \cup \widetilde N_0)\right).
\]
Since $\mathcal L(\sigma_1^{-1}(N) \cap (N_0 \cup \widetilde N_0)) = 0$, we finally deduce that $\mathcal L(\sigma_1^{-1}(N)) = 0$.
\end{proof}

We now turn to the question of existence of $L^p$ solutions of \eqref{eq:diff-eqn-N1-time}. The conditions for the existence of such solutions turn out to depend on whether $p$ is finite or infinite, and we start with the case $p = +\infty$.

{\sloppy
\begin{theorem}
\label{thm:exist-L-infty}
Let $\tau$ be a delay function satisfying \ref{hypo:tau-infimum}, \ref{hypo:tau-measurable}, and \ref{hypo:null}, $h$ be its associated largest delay function, and $x_0 \in L^\infty(\lvert -h(0), 0], \mathbb R^d)$. Then there exists a unique $L^\infty$ solution $x \in L^\infty_{\mathrm{loc}}(\lvert -h(0), +\infty),\allowbreak \mathbb R^d)$ of \eqref{eq:diff-eqn-N1-time} with initial condition $x_0$.
\end{theorem}
}

\begin{proof}
Since $L^\infty$ solutions are particular cases of a.e.\ solutions, the assertion on the uniqueness of an $L^\infty$ solution follows from the uniqueness assertion for a.e.\ solutions from Theorem~\ref{thm:exist-ae}. We are thus left to prove the existence of an $L^\infty$ solution with a given initial condition.

Let $x_0 \in L^\infty(\lvert -h(0), 0], \mathbb R^d)$. Up to replacing $x_0$ by a suitable element in its equivalence class with respect to a.e.\ equality, we can assume that $x_0$ is bounded, i.e., there exists $M > 0$ such that $\abs{x_0(t)} \leq M$ for every $t \in \lvert -h(0), 0]$.

Since $x_0 \in L^\infty(\lvert -h(0), 0], \mathbb R^d)$, we have in particular that $x_0$ is $(\mathfrak L, \mathfrak B)$-measurable and thus, by Theorem~\ref{thm:exist-measurable}, \eqref{eq:diff-eqn-N1-time} admits a unique measurable solution $x\colon \lvert -h(0),\allowbreak +\infty) \to \mathbb R^d$ with initial condition $x_0$. By Theorem~\ref{thm:exist-ae}, $x$ is also an a.e.\ solution of \eqref{eq:diff-eqn-N1-time} with initial condition $x_0$. By Theorem~\ref{theorem-1}, we have $x(t) = A^{\mathbf n(t)} x_0(\sigma_{\mathbf n(t)}(t))$ for every $t \in \mathbb R_+$, and thus $\abs{x(t)} \leq \abs{A}^{\mathbf n(t)} M$ for every $t \geq 0$. The conclusion now follows thanks to Lemma~\ref{lemma:n-loc-bounded}.
\end{proof}

We conclude the main results of this section by the following result on the existence and uniqueness of solutions in $L^p$ for $p \in [1, +\infty)$.

\begin{theorem}
\label{thm:exist-L-p}
Let $p \in [1, +\infty)$,  $\tau$ be a delay function satisfying \ref{hypo:tau-infimum} and \ref{hypo:radon-nik-bdd}, $h$ be its associated largest delay function, and $x_0 \in L^p(\lvert -h(0), 0], \mathbb R^d)$. Then there exists a unique $L^p$ solution $x \in L^p_{\mathrm{loc}}(\lvert -h(0), +\infty),\allowbreak \mathbb R^d)$ of \eqref{eq:diff-eqn-N1-time} with initial condition $x_0$.
\end{theorem}

\begin{proof}
As in the proof of Theorem~\ref{thm:exist-L-infty}, uniqueness of solutions is a consequence of Theorem~\ref{thm:exist-ae}.

Let $x_0 \in L^p(\lvert -h(0), 0], \mathbb R^d)$. Arguing as in the proof of Theorem~\ref{thm:exist-L-infty}, \eqref{eq:diff-eqn-N1-time} admits a unique solution $x\colon \lvert -h(0), +\infty) \to \mathbb R^d$ with initial condition $x_0$, which is both a measurable and an a.e.\ solution. We now take $T > 0$ and define $\delta > 0$ and $M \in \mathbb N^\ast$ as in the proof of Theorem~\ref{thm:exist-continuous}. Let $\varphi\colon \mathbb R \to \mathbb R_+$ be the Radon--Nikodym derivative of $\left(\sigma_1\rvert_{[0, T]}\right)_{\#}\mathcal L$ with respect to $\mathcal L$ and denote by $K > 0$ an upper bound of $\varphi$. We prove by induction on $k$ that $x \in L^p(\lvert -h(0), k \delta], \mathbb R^d)$ for every $k \in \llbracket 0, M\rrbracket$, which will conclude the proof since $M\delta = T$ and $T > 0$ is arbitrary. The case $k = 0$ is clear since $x$ coincides with $x_0$ in $\lvert -h(0), 0)$.

Let $k \in \llbracket 0, M-1\rrbracket$ be such that $x \in L^p(\lvert -h(0), k \delta], \mathbb R^d)$. Define $\widehat x$ and $y$ as in the proof of Theorem~\ref{thm:exist-continuous} and note that, as in that proof, $\sigma_1([k\delta, (k+1)\delta]) \subset \lvert -h(0), k\delta]$ and $y(t) = A \widehat x(\sigma_1(t))$ for every $t \in [k\delta, (k+1)\delta]$. By \ref{hypo:radon-nik-bdd} and Theorem~\ref{thm:composition-Lp}\eqref{item:composition-Lp-equivs}, we deduce that $y \in L^p([k\delta, (k+1)\delta], \mathbb R^d)$. Hence $x \in L^p(\lvert -h(0), (k+1) \delta], \mathbb R^d)$, yielding the conclusion.
\end{proof}

Contrarily to the case of $L^\infty$ from Theorem~\ref{thm:exist-L-infty}, assumptions \ref{hypo:tau-infimum}, \ref{hypo:tau-measurable}, and \ref{hypo:null} are not sufficient to ensure existence of solutions in $L^p$ with $p \in [1, +\infty)$, as we detail in the following example.

\begin{example}
\label{expl:H6-important-Lp}
Let $\tau$ be the delay function defined for $t \in \mathbb R_+$ by $\tau(t) = 2 - \frac{1}{t + 1}$. Then $\sigma_1(t) = t - 2 + \frac{1}{t + 1}$ for $t \in \mathbb R_+$. One easily verifies that $\sigma_1$ is increasing, maps $\mathbb R_+$ onto $[-1, +\infty)$, and is a diffeomorphism between $\mathbb R_+^\ast$ and $(-1, +\infty)$, with $\sigma_1^{-1}(s) = \frac{s + 1 + \sqrt{(s + 1)(s + 5)}}{2}$ and $(\sigma_1^{-1})'(s) = \frac{1}{2} + \frac{\sqrt{s + 5}}{4 \sqrt{s + 1}} + \frac{\sqrt{s + 1}}{4 \sqrt{s + 5}}$ for every $s \in (-1, +\infty)$. These facts also allow one to show that \ref{hypo:tau-infimum}, \ref{hypo:tau-measurable}, and \ref{hypo:null} are verified, but \ref{hypo:radon-nik-bdd} is not. One computes in addition that $\mathbf n(t) = 1$ if and only if $t \in [0, \frac{1 + \sqrt{5}}{2})$.

Let $p \in [1, +\infty)$ and $x_0 \in L^p([-1, 0], \mathbb R^d)$. By combining Theorems~\ref{thm:exist-general}, \ref{thm:exist-measurable}, and \ref{thm:exist-ae}, \eqref{eq:diff-eqn-N1-time} admits a unique solution $x\colon [-1, +\infty) \to \mathbb R^d$ with initial condition $x_0$, which is also both a measurable and a.e.\ solution. An immediate computation from \eqref{eq:diff-eqn-N1-time} shows that, for every $t \in [0, \frac{1 + \sqrt{5}}{2})$, we have $x(t) = A x_0\left(t - 2 + \frac{1}{t + 1}\right)$. Hence,
\begin{align*}
\int_0^{\frac{1 + \sqrt{5}}{2}} \abs*{x(t)}^p \diff t & = \int_0^{\frac{1 + \sqrt{5}}{2}} \abs*{A x_0\left(t - 2 + \frac{1}{t + 1}\right)}^p \diff t \\
& = \int_{-1}^{0} \abs*{A x_0(s)}^p \left(\frac{1}{2} + \frac{\sqrt{s + 5}}{4 \sqrt{s + 1}} + \frac{\sqrt{s + 1}}{4 \sqrt{s + 5}}\right) \diff s,
\end{align*}
and the above integral may be equal to $+\infty$, in which case $x \notin L^p_{\mathrm{loc}}([-1, +\infty), \mathbb R^d)$.

For instance, assume that $A \neq 0$ and let $\alpha \in \left[\frac{1}{2p}, \frac{1}{p}\right)$ and $v \in \mathbb R^d \setminus \{0\}$ be such that $A v \neq 0$. Let $x_0$ be defined by $x_0(t) = \frac{1}{(t + 1)^\alpha} v$ for $t \in (-1, 0)$ and $x_0(-1) = 0$. Then $x_0 \in L^p([-1, 0], \mathbb R^d)$ but
\[
\int_{-1}^{0} \abs*{A x_0(s)}^p \frac{\sqrt{s + 5}}{4 \sqrt{s + 1}} \diff s = \abs{A v}^p \int_{-1}^0 \frac{\sqrt{s + 5}}{4 (s + 1)^{\alpha p + \frac{1}{2}}} \diff s = +\infty
\]
since $\alpha p + \frac{1}{2} \geq 1$.
\end{example}

Despite the importance of \ref{hypo:radon-nik-bdd} for Theorem~\ref{thm:exist-L-p}, illustrated by Theorem~\ref{thm:composition-Lp}\eqref{item:composition-Lp-equivs} and Example~\ref{expl:H6-important-Lp}, \ref{hypo:radon-nik-bdd} is not necessary for the existence of $L^p$ solutions of \eqref{eq:diff-eqn-N1-time}, as illustrated by the following example.

\begin{example}
Let $\tau$ be a delay function such that the function $\sigma_1 = \id - \tau$ is given, for $t \in \mathbb R_+$, by
\[
\sigma_1(t) = \begin{dcases*}
\sqrt[3]{t - 1} - 1 & if $t \in [0, 2)$, \\
(t - 3)^3 + 1 & if $t \in [2, 4)$, \\
t - 2 & if $t \in [4, +\infty)$.
\end{dcases*}
\]
Note that ${\sigma_1}_\# \mathcal L$ is $\sigma$-finite and absolutely continuous with respect to $\mathcal L$ and $\sigma_1\colon \mathbb R_+ \to [-2, +\infty)$ is increasing and invertible with
\[
\sigma_1^{-1}(s) = \begin{dcases*}
(s + 1)^3 + 1 & if $s \in [-2, 0)$, \\
\sqrt[3]{s - 1} + 3 & if $s \in [0, 2)$, \\
s + 2 & if $s \in [2, +\infty)$.
\end{dcases*}
\]
Hence, one can compute that the Radon--Nikodym derivative $\varphi$ of ${\sigma_1}_\# \mathcal L$ with respect to $\mathcal L$ is the unbounded function $\varphi$ defined for $s \in \mathbb R \setminus \{1\}$ by
\[
\varphi(s) = \begin{dcases*}
0 & if $s \in (-\infty, -2)$, \\
3 (s + 1)^2 & if $s \in [-2, 0)$, \\
\frac{1}{3 (s - 1)^{2/3}} & if $s \in [0, 2) \setminus \{1\}$, \\
1 & if $s \in [2, +\infty)$.
\end{dcases*}
\]

Even though $\varphi$ is unbounded, \eqref{eq:diff-eqn-N1-time} admits a unique $L^p$ solution for a given initial condition in $L^p$. Indeed, given $x_0 \in L^p([-2, 0], \mathbb R^d)$, it follows from Theorems~\ref{thm:exist-measurable} and \ref{thm:exist-ae} that \eqref{eq:diff-eqn-N1-time} admits a unique solution with initial condition $x_0$, which is both a measurable solution and an a.e.\ solution, and, by Theorem~\ref{thm:explicit-ae}, this solution satisfies \eqref{eq:explicit-formula-ae}. Explicit computations show that, for every $n \in \mathbb N^\ast$, we have $D_n = [2(n-1), +\infty)$, so that $\mathbf n(t) = \floor{t / 2} + 1$ for every $t \in \mathbb R_+$. In addition, by an immediate inductive argument, we have, for $n \geq 2$ and $t \in D_n$,
\[
\sigma_n(t) = \begin{dcases*}
t - 2n & if $t \in [2(n-1), 2n) \cup [2(n+1), +\infty)$, \\
(t - 2n - 1)^3 + 1 & if $t \in [2n, 2(n+1))$.
\end{dcases*}
\]
In particular, for $t \in \mathbb R_+$, we have
\[
\sigma_{\mathbf n(t)}(t) = \begin{dcases*}
\sqrt[3]{t - 1} - 1 & if $t \in [0, 2)$, \\
t - 2n & if $t \in [2(n-1), 2n)$ for some $n \in \mathbb N$ with $n \geq 2$.
\end{dcases*}
\]
Hence, by \eqref{eq:explicit-formula-ae}, for a.e.\ $t \in \mathbb R_+$, we have
\[
x(t) = \begin{dcases*}
A x_0(\sqrt[3]{t - 1} - 1) & if $t \in [0, 2)$, \\
A^n x_0(t - 2n) & if $t \in [2(n-1), 2n)$ for some $n \in \mathbb N$ with $n \geq 2$,
\end{dcases*}
\]
and thus one immediately verifies that $x \in L^p_{\mathrm{loc}}([-2, +\infty), \mathbb R^d)$, yielding that $x$ is an $L^p$ solution of \eqref{eq:diff-eqn-N1-time} with initial condition $x_0$.
\end{example}

We conclude this section by remarking that, by requiring the additional assumption \ref{hypo:radon-nik-and-rho} in Theorem~\ref{thm:exist-L-p}, one obtains the stronger conclusion that the solution $x$ belongs to $L^p(\lvert -h(0), +\infty), \mathbb R^d)$.

\begin{lemma}
\label{lemma:x-in-Lp}
Let $p \in [1, +\infty)$, $A \in \mathcal M_d(\mathbb R)$, $\tau$ be a delay function, $h$ be its associated largest delay function, and $x_0 \in L^p(\lvert -h(0), 0], \mathbb R^d)$. Assume that \ref{hypo:tau-infimum} and \ref{hypo:radon-nik-and-rho} are satisfied. Then the solution $x$ of \eqref{eq:diff-eqn-N1-time} with initial condition $x_0$ satisfies $x \in L^p(\lvert -h(0), +\infty),\allowbreak \mathbb R^d)$.
\end{lemma}

\begin{proof}
Let $\varphi$ denote the Radon--Nikodym derivative of ${\sigma_1}_{\#} \mathcal L$ with respect to $\mathcal L$. Since $\norm{\varphi}_{L^\infty(\mathbb R, \mathbb R)} \rho(A)^p < 1$, there exists a norm $\abs{\cdot}$ in $\mathbb R^d$ for which we have $\norm{\varphi}_{L^\infty(\mathbb R, \mathbb R)} \abs{A}^p \allowbreak < 1$ (see, e.g., \cite[Lemma~5.6.10]{Horn2013Matrix}). We fix such a norm in the sequel.

Let $(D_n)_{n \in \mathbb N}$ be the nonincreasing sequence of sets from Definition~\ref{def:Dn-sigman} and $D_\ast = \lvert -h(0), +\infty)$. Thanks to Lemma~\ref{lemma:D_n-sigma1}, we have $\sigma_1^{-1}(D_{n-1} \setminus D_n) = D_n \setminus D_{n+1}$ for every $n \in \mathbb N^\ast$. In addition, since $\sigma_1$ takes values in $D_\ast$ and $D_n \subset \mathbb R_+ \subset D_\ast$ for every $n \in \mathbb N^\ast$, we have $\sigma_1^{-1}(D_\ast \cap (D_{n-1} \setminus D_n)) = \sigma_1^{-1}(D_{n-1} \setminus D_n) = D_n \setminus D_{n+1} = D_\ast \cap (D_n \setminus D_{n+1})$ for every $n \in \mathbb N^\ast$.

Let $n \in \mathbb N^\ast$. Using \eqref{eq:diff-eqn-N1-time}, we have
\begin{align*}
\int_{D_\ast \cap (D_n \setminus D_{n+1})} \abs{x(t)}^p \diff t & \leq \abs{A}^p \int_{D_\ast \cap (D_n \setminus D_{n+1})} \abs{x(\sigma_1(t))}^p \diff t \displaybreak[0] \\
& = \abs{A}^p \int_{\sigma_1^{-1}(D_\ast \cap (D_{n-1} \setminus D_{n}))} \abs{x(\sigma_1(t))}^p \diff t \displaybreak[0] \\
& = \abs{A}^p \int_{D_\ast \cap (D_{n-1} \setminus D_{n})} \abs{x(t)}^p \diff {\sigma_1}_{\#} \mathcal L(t) \displaybreak[0] \\
& = \abs{A}^p \int_{D_\ast \cap (D_{n-1} \setminus D_{n})} \varphi(t) \abs{x(t)}^p \diff t \displaybreak[0] \\
& \leq \norm{\varphi}_{L^\infty(\mathbb R, \mathbb R)} \abs{A}^p \int_{D_\ast \cap (D_{n-1} \setminus D_{n})} \abs{x(t)}^p \diff t.
\end{align*}
Hence, for every $n \in \mathbb N$, we have
\begin{align*}
\int_{D_\ast \cap (D_n \setminus D_{n+1})} \abs{x(t)}^p \diff t & \leq \left(\norm{\varphi}_{L^\infty(\mathbb R, \mathbb R)} \abs{A}^p\right)^n \int_{D_\ast \cap (D_{0} \setminus D_{1})} \abs{x(t)}^p \diff t \displaybreak[0] \\
& = \left(\norm{\varphi}_{L^\infty(\mathbb R, \mathbb R)} \abs{A}^p\right)^n \norm{x_0}_{L^p(\lvert -h(0), 0), \mathbb R^d)}^p.
\end{align*}
Thus
\begin{align*}
\int_{-h(0)}^{+\infty} \abs{x(t)}^p \diff t = \sum_{n = 0}^{\infty} \int_{D_\ast \cap (D_n \setminus D_{n+1})} \abs{x(t)}^p \diff t \leq \frac{\norm{x_0}_{L^p(\lvert -h(0), 0), \mathbb R^d)}^p}{1 - \norm{\varphi}_{L^\infty(\mathbb R, \mathbb R)} \abs{A}^p},
\end{align*}
yielding the conclusion.
\end{proof}

\subsection{Continuous dependence on parameters}
\label{sec:continuous}

In order to complete the well-posedness analysis of \eqref{eq:diff-eqn-N1-time}, we now address the problem of the continuous dependence of the solution with respect to its initial condition and to the matrix $A$ from \eqref{eq:diff-eqn-N1-time}.

We start with the following preliminary lemma, whose proof, based on simple linear-algebraic computations, is provided here only for sake of completeness.

\begin{lemma}
\label{lemma:Akn-cv-uniform}
Let $A \in \mathcal M_d(\mathbb R)$ be a matrix satisfying \ref{hypo:spr-A-less-1} and $(A_k)_{k \in \mathbb N}$ be a sequence of matrices in $\mathcal M_d(\mathbb R)$ converging to $A$. Then $A_k^n \to A^n$ as $k \to +\infty$ uniformly on $n \in \mathbb N$.
\end{lemma}

\begin{proof}
Take a norm $\abs{\cdot}$ in $\mathbb R^d$ such that $\abs{A} < 1$, which exists thanks to \ref{hypo:spr-A-less-1}. Let $\rho_\ast = \frac{\abs{A} + 1}{2} \in (\abs{A}, 1)$. Let $\varepsilon > 0$ and pick $N \in \mathbb N^\ast$ such that $\rho_\ast^n < \frac{\varepsilon}{2}$ for every $n \in \mathbb N$ with $n > N$. We then take $K \in \mathbb N$ such that $\abs*{A_k^n - A^n} < \min\left\{\varepsilon, \frac{1 - \abs{A}}{2}\right\}$ for every $n \in \llbracket 0, N\rrbracket$ and $k \in \mathbb N$ with $k \geq K$. In particular, $\abs*{A_k} < \abs{A} + \frac{1 - \abs{A}}{2} = \rho_\ast$ for every $k \in \mathbb N$ with $k \geq K$. Hence, for every $n \in \mathbb N$ and $k \in \mathbb N$, if $k \geq K$, we have
\[
\begin{aligned}
\abs*{A_k^n - A^n} & < \varepsilon, & \qquad & \text{if } n \in \llbracket 0, N\rrbracket, \\
\abs*{A_k^n - A^n} & \leq \abs*{A_k}^n + \abs*{A}^n < 2 \rho_\ast^n < \varepsilon, & & \text{if } n > N,
\end{aligned}
\]
yielding the conclusion.
\end{proof}

We now turn to the main topic of this section. Let us consider a sequence of problems
\begin{equation}\label{eq:diff-eqn-N1-time-k}
x_k (t) = A_k x_k (t - \tau(t)), \qquad t \geq 0,
\end{equation}
where $(A_k)_{k \in \mathbb N}$ is a sequence of matrices in $\mathcal M_d(\mathbb R)$. Our first result ensures that, under suitable assumptions, if $A_k$ converges to $A$ as $k \to \infty$, and if $x_k$ is a solution of \eqref{eq:diff-eqn-N1-time-k} for each $k \in \mathbb N$ such that the corresponding sequence of initial conditions $x_{0, k}$ converges to some function $x_0$, then $x_k$ converges, in a suitable sense, to a solution $x$ of the limiting problem \eqref{eq:diff-eqn-N1-time}.

\begin{theorem}
\label{thm:cd-on-p-1}
Let $(A_k)_{k \in \mathbb N}$ be a sequence of matrices in $\mathcal M_d(\mathbb R)$ converging to some $A \in \mathcal M_d(\mathbb R)$, $\tau$ be a delay function satisfying \ref{hypo:tau-infimum}, $h$ be its associated largest delay function, and, for each $k \in \mathbb N$, $x_k \colon \lvert -h(0), +\infty) \to \mathbb R^d$ be a solution of \eqref{eq:diff-eqn-N1-time-k} with initial condition $x_{0, k} \colon \lvert -h(0), 0) \to \mathbb R^d$. Assume that there exists a function $x_0 \colon \lvert -h(0), +\infty) \to \mathbb R^d$ such that $\lim_{k \to +\infty} x_{0, k}(t) = x_0(t)$ for every $t \in \lvert -h(0), 0)$. Then the following assertions hold true.
\begin{enumerate}
\item\label{item:cd-simple} We have $\lim_{k \to +\infty} x_k(t) = x(t)$ for every $t \in \lvert -h(0), +\infty)$, where $x$ is the unique solution of \eqref{eq:diff-eqn-N1-time} with initial condition $x_0$.

\item\label{item:cd-uniform-compact} If $\lim_{k \to +\infty} x_{0, k}(t) = x_0(t)$ uniformly for $t \in \lvert -h(0), 0)$ and $x_0$ is bounded, then, for every $T > 0$, we have $\lim_{k \to +\infty} x_k(t) = x(t)$ uniformly for $t \in \lvert -h(0), T]$, where $x$ is as in \ref{item:cd-simple}.

\item\label{item:cd-uniform-all} If $\lim_{k \to +\infty} x_{0, k}(t) = x_0(t)$ uniformly for $t \in \lvert -h(0), 0)$, $x_0$ is bounded, and $A$ satisfies \ref{hypo:spr-A-less-1}, then $\lim_{k \to +\infty} x_k(t) = x(t)$ uniformly for $t \in \lvert -h(0), +\infty)$, where $x$ is as in \ref{item:cd-simple}.
\end{enumerate}
\end{theorem}

\begin{proof}
By Theorem~\ref{theorem-1}, we have, for all $k \in \mathbb N$ and $t \in \mathbb R_+$,
\[
x_k(t) = A_k^{\mathbf n(t)} x_{0, k}(\sigma_{\mathbf n(t)}(t)), \qquad x(t) = A^{\mathbf n(t)} x_{0}(\sigma_{\mathbf n(t)}(t)),
\]
where $\mathbf n \colon \mathbb R \to \mathbb N$ is the function from Corollary~\ref{coro:integer-iterations}. Assertion \ref{item:cd-simple} follows immediately from the above formulas by letting $k \to +\infty$.

Assume now that the convergence of $x_{0, k}$ to $x_0$ is uniform in $\lvert -h(0), 0)$ and that $x_0$ is bounded and take $T > 0$. Fix an arbitrary norm $\abs{\cdot}$ in $\mathbb R^d$ and take $M > 0$ such that $\abs*{x_0(t)} \leq M$ for every $t \in \lvert -h(0), 0)$. Using Lemma~\ref{lemma:n-loc-bounded}, we also take $N > 0$ such that $\mathbf n(t) \leq N$ for every $t \in \lvert -h(0), T]$. Let $C > 0$ be such that $\abs*{A^n} \leq C$ for every $n \in \llbracket 0, N\rrbracket$. Finally, take $\varepsilon \in (0, 2 M C)$.

Since $A_k \to A$ as $k \to +\infty$, there exists $K \in \mathbb N$ such that $\abs*{A_k^n - A^n} < \frac{\varepsilon}{2 M}$ for every $k \in \mathbb N$ with $k \geq K$ and $n \in \llbracket 0, N\rrbracket$. In particular, we have $\abs*{A_k^n} \leq C + \frac{\varepsilon}{2 M} < 2 C$ for every $k \in \mathbb N$ with $k \geq K$ and $n \in \llbracket 0, N\rrbracket$. Using the uniform convergence of $x_{0, k}$ to $x_k$, we deduce that, up to increasing $K$, we have $\abs*{x_{0, k}(s) - x_0(s)} < \frac{\varepsilon}{4 C}$ for every $s \in \lvert -h(0), 0)$ and $k \in \mathbb N$ with $k \geq K$. Hence, for every $k \in \mathbb N$ and $t \in [0, T]$, if $k \geq K$, we then have
\begin{align}
\abs*{x_k(t) - x(t)} & = \abs*{A_k^{\mathbf n(t)} x_{0, k}(\sigma_{\mathbf n(t)}(t)) - A^{\mathbf n(t)} x_{0}(\sigma_{\mathbf n(t)}(t))} \notag \\
& \leq \abs*{A_k^{\mathbf n(t)}} \abs*{x_{0, k}(\sigma_{\mathbf n(t)}(t)) - x_{0}(\sigma_{\mathbf n(t)}(t))} + \abs*{A_k^{\mathbf n(t)} - A^{\mathbf n(t)}} \abs*{x_{0}(\sigma_{\mathbf n(t)}(t))} \notag \\
& < 2 C \frac{\varepsilon}{4 C} + \frac{\varepsilon}{2 M} M = \varepsilon, \label{eq:estim-uniform-cv}
\end{align}
completing the proof of \ref{item:cd-uniform-compact}.

To prove \ref{item:cd-uniform-all}, fix a norm $\abs*{\cdot}$ in $\mathbb R^d$ such that $\abs*{A} < 1$, which exists thanks to \ref{hypo:spr-A-less-1}. We take $K \in \mathbb N$ large enough so that $\abs{A_k} < 1$ for every $k \in \mathbb N$ with $k \geq K$. In particular, $\abs{A_k^n} \leq \abs{A_k}^n < 1$ and $\abs{A^n} \leq \abs{A}^n < 1$ for every $n \in \mathbb N$ and $k \in \mathbb N$ with $k \geq K$. We let $M > 0$ be as in the proof of \ref{item:cd-uniform-compact} and we take $\varepsilon > 0$. Using Lemma~\ref{lemma:Akn-cv-uniform}, up to increasing $K$, we have $\abs*{A_k^n - A^n} < \frac{\varepsilon}{2 M}$ for every $n \in \mathbb N$ and $k \in \mathbb N$ with $k \geq K$. Using the uniform convergence of $x_{0, k}$ to $x_k$, we deduce that, up to increasing $K$, we have $\abs*{x_{0, k}(s) - x_0(s)} < \frac{\varepsilon}{2}$ for every $s \in \lvert -h(0), 0)$ and $k \in \mathbb N$ with $k \geq K$. Hence, for every $k \in \mathbb N$ and $t \in \mathbb R_+$, if $k \geq K$, the same estimate \eqref{eq:estim-uniform-cv} as above holds with $C$ replaced by $\frac{1}{2}$, yielding the conclusion.
\end{proof}

\begin{example}
Let $a \in \mathbb R$. For $k \in \mathbb N^\ast$, consider the scalar equation
\[x_k(t) = \left(a + \dfrac{1}{k}\right) x_k(t - 1), \qquad t \geq 0,\]
with initial condition
\[x_{0, k}(t) = a + \dfrac{1}{k} \qquad \text{ for } t \in [-1, 0).\]
It is immediate to verify that the corresponding solution is given by
\[x_k(t) = \left(a + \dfrac{1}{k}\right)^{\floor{t} + 2} \qquad \text{ for } t \in [-1, +\infty).\]
Since $a + \frac{1}{k} \to a$ as $k \to +\infty$, it follows from Theorem~\ref{thm:cd-on-p-1} that the function $x\colon [-1, +\infty) \to \mathbb R$ defined by
\[
x(t) = \lim_{k \to \infty} x_k(t) = a^{\floor{t} + 2}
\]
is the unique solution of the equation
\[
x(t) = a x(t - 1), \qquad t \geq 0,
\]
with initial condition $x_0\colon [-1, 0) \to \mathbb R$ given by $x_0(t) = a$ for $t \in [-1, 0)$, a fact that can also be verified by an immediate computation. The convergence of $(x_k(t))_{k \in \mathbb N}$ to $x(t)$ is uniform in $t$ in any interval of the form $[-1, T]$ for $T > 0$, and it is also uniform in $[-1, +\infty)$ when $\abs{a} < 1$.
\end{example}

\begin{remark}
Theorem~\ref{thm:cd-on-p-1} deals with continuity with respect to the matrix $A$ and the initial condition $x_0$ of the difference equation \eqref{eq:diff-eqn-N1-time}, but not with respect to the delay function $\tau$, and the main reason for that is that the counterpart of Theorem~\ref{thm:cd-on-p-1} is not true for continuity with respect to $\tau$. Indeed, for $k \in \mathbb N$, consider the scalar difference equation
\begin{equation}
\label{eq:expl-not-cd-tau}
x_k(t) = x_k(t - \tau_k(t)), \qquad t \geq 0,
\end{equation}
where $\tau_k(t) = t + 1 - \frac{1}{k + 2}$ for $t \in \mathbb R_+$. Note that $\tau_k$ satisfies \ref{hypo:tau-infimum} for every $k \in \mathbb N$ and that $\tau_k(t) \to \tau(t)$ as $k \to +\infty$, where $\tau(t) = t + 1$ for $t \in \mathbb R_+$, this convergence being uniform in $t \in \mathbb R_+$. We let $x_0\colon [-1, 0) \to \mathbb R$ be given by $x_0(-1) = 0$ and $x_0(t) = 1$ for $t \in (-1, 0)$.

For $k \in \mathbb N$, the unique solution $x_k$ of \eqref{eq:expl-not-cd-tau} with initial condition $x_0 \rvert_{[-1 + \frac{1}{k + 2}, 0)}$ satisfies $x_k(t) = x_0(-1 + \frac{1}{k + 2}) = 1$ for every $t \geq 0$. On the other hand, the unique solution $x$ of the limit equation
\[x(t) = x(t - \tau(t)), \qquad t \geq 0,\]
with initial condition $x_0$ satisfies $x(t) = x_0(-1) = 0$ for every $t \geq 0$. Hence, for every $t \in \mathbb R_+$, $x_k(t)$ does not converge to $x(t)$ as $k \to +\infty$.

This example can be seen as reminiscent of the fact that, even for time-delay systems with constant delays, some properties of the system, such as exponential stability of the trivial solution, may fail to be robust with respect to variations of the delay, as illustrated for instance in \cite[Section~9.6]{Hale1993Introduction}, \cite[Example~1.37]{Mazanti2016Stability}, and \cite[Example~1.36]{Michiels2014Stability}. Issues of robustness (or continuous dependence of solutions) with respect to the delays are typically much more involved to analyze and deserve a study on their own, being thus out of the scope of the present paper.
\end{remark}

As easy consequences of Theorem~\ref{thm:cd-on-p-1}, we obtain at once the following two results, which deal with continuous dependence of continuous and regulated solutions of \eqref{eq:diff-eqn-N1-time}.

\begin{corollary}
\label{col:cd-on-p-2}
Let $(A_k)_{k \in \mathbb N}$ be a sequence of matrices in $\mathcal M_d(\mathbb R)$ converging to some $A \in \mathcal M_d(\mathbb R)$, $\tau$ be a delay function satisfying \ref{hypo:tau-infimum}, $h$ be its associated largest delay function, and, for each $k \in \mathbb N$, $x_k \in C(\lvert -h(0), +\infty), \mathbb R^d)$ be a continuous solution of \eqref{eq:diff-eqn-N1-time-k} with initial condition $x_{0, k} \in C(\lvert -h(0), 0], \mathbb R^d)$. Assume that there exists a continuous and bounded function $x_0 \in C(\lvert -h(0), 0], \mathbb R^d)$ such that $\lim_{k \to +\infty} x_{0, k}(t) = x_0(t)$ uniformly for $t \in \lvert -h(0), 0]$. Then, for every $T > 0$, we have $\lim_{k \to +\infty} x_{k}(t) = x(t)$ uniformly for $t \in \lvert -h(0), T]$, where $x \in C(\lvert -h(0), +\infty), \mathbb R^d)$ is the unique continuous solution of \eqref{eq:diff-eqn-N1-time} with initial condition $x_0$. In addition, if $A$ satisfies \ref{hypo:spr-A-less-1}, then $\lim_{k \to +\infty} x_{k}(t) = x(t)$ uniformly for $t \in \lvert -h(0), +\infty)$.
\end{corollary}

\begin{corollary}
\label{col:cd-on-p-3}
Let $(A_k)_{k \in \mathbb N}$ be a sequence of matrices in $\mathcal M_d(\mathbb R)$ converging to some $A \in \mathcal M_d(\mathbb R)$, $\tau$ be a delay function satisfying \ref{hypo:tau-infimum}, $h$ be its associated largest delay function, and, for each $k \in \mathbb N$, $x_k \in G(\lvert -h(0), +\infty), \mathbb R^d)$ be a regulated solution of \eqref{eq:diff-eqn-N1-time-k} with initial condition $x_{0, k} \in G(\lvert -h(0), 0], \mathbb R^d)$. Assume that there exists a bounded regulated function $x_0 \in G(\lvert -h(0), 0], \mathbb R^d)$ such that $\lim_{k \to +\infty} x_{0, k}(t) = x_0(t)$ uniformly for $t \in \lvert -h(0), 0]$. Then, for every $T > 0$, we have $\lim_{k \to +\infty} x_{k}(t) = x(t)$ uniformly for $t \in \lvert -h(0), T]$, where $x \in G(\lvert -h(0), +\infty), \mathbb R^d)$ is the unique regulated solution of \eqref{eq:diff-eqn-N1-time} with initial condition $x_0$. In addition, if $A$ satisfies \ref{hypo:spr-A-less-1}, then $\lim_{k \to +\infty} x_{k}(t) = x(t)$ uniformly for $t \in \lvert -h(0), +\infty)$.
\end{corollary}

\begin{remark}
In Corollary~\ref{col:cd-on-p-2} (respectively, Corollary~\ref{col:cd-on-p-3}), the fact that $x_k$ is a continuous (respectively, regulated) solution of \eqref{eq:diff-eqn-N1-time-k} implies that its initial condition $x_{0, k}$ satisfies the compatibility condition $x_{0, k}(0) = A_k x_{0, k}(-\tau(0))$. In particular, letting $k \to +\infty$, we deduce that $x_0(0) = A x_0(-\tau(0))$, i.e., $x_0$ also satisfies a compatibility condition.
\end{remark}

\begin{remark}
Note that, contrarily to Theorem~\ref{thm:exist-continuous}, we do not assume that $\tau$ satisfies \ref{hypo:tau-cont} in Corollary~\ref{col:cd-on-p-2}. We assume, however, that continuous solutions $x_k$ of \eqref{eq:diff-eqn-N1-time-k} do exist (and they are unique given their initial conditions $x_{0, k}$ thanks to Corollary~\ref{coro:uniqueness}). The existence and uniqueness of the solution $x$ of \eqref{eq:diff-eqn-N1-time} with initial condition $x_0$ then follows from Theorem~\ref{thm:cd-on-p-1}, and its continuity is a consequence of the uniform convergence from Theorem~\ref{thm:cd-on-p-1}\ref{item:cd-uniform-compact}. A similar remark applies to Corollary~\ref{col:cd-on-p-3}, \ref{hypo:regulated}, and regulated solutions.
\end{remark}

We also obtain, as a consequence of Theorem~\ref{thm:cd-on-p-1}, the counterpart of Corollaries~\ref{col:cd-on-p-2} and \ref{col:cd-on-p-3} for $L^\infty$ solutions. Since the proof of such a result is slightly more delicate, we provide its details below.

\begin{corollary}
\label{col:cd-on-p-Linfty}
Let $(A_k)_{k \in \mathbb N}$ be a sequence of matrices in $\mathcal M_d(\mathbb R)$ converging to some $A \in \mathcal M_d(\mathbb R)$, $\tau$ be a delay function satisfying \ref{hypo:tau-infimum}, \ref{hypo:tau-measurable}, and \ref{hypo:null}, $h$ be its associated largest delay function, and, for each $k \in \mathbb N$, $x_k \in L^\infty_{\mathrm{loc}}(\lvert -h(0), +\infty), \mathbb R^d)$ be an $L^\infty$ solution of \eqref{eq:diff-eqn-N1-time-k} with initial condition $x_{0, k} \in L^\infty(\lvert -h(0), 0], \mathbb R^d)$. Assume that there exists $x_0 \in L^\infty(\lvert -h(0), 0], \mathbb R^d)$ such that $\lim_{k \to +\infty} \norm{x_{0, k} - x_0}_{L^\infty(\lvert -h(0), 0], \mathbb R^d)} = 0$. Then, for every $T > 0$, we have $\lim_{k \to +\infty} \norm{x_{k} - x}_{L^\infty(\lvert -h(0), T], \mathbb R^d)} = 0$, where $x \in L^\infty_{\mathrm{loc}}(\lvert -h(0), +\infty), \mathbb R^d)$ is the unique $L^\infty$ solution of \eqref{eq:diff-eqn-N1-time} with initial condition $x_0$. In addition, if $A$ satisfies \ref{hypo:spr-A-less-1}, then $\lim_{k \to +\infty} \norm{x_{k} - x}_{L^\infty(\lvert -h(0), +\infty), \mathbb R^d)} = 0$.
\end{corollary}

\begin{proof}
In this proof, for sake of clarity, we make an explicit difference between functions defined everywhere and equivalence classes of functions by almost everywhere equality. We hence consider $x_k$, $x_{0, k}$, and $x_0$ as equivalence classes of functions. Since $\lim_{k \to +\infty} \norm{x_{0, k} - x_0}_{L^\infty(\lvert -h(0), 0], \mathbb R^d)} = 0$, one may take a representative $\overline x_0\colon \lvert -h(0), 0) \to \mathbb R^d$ of $x_0$ and, for each $k \in \mathbb N$, a representative $\overline x_{0, k}\colon \lvert -h(0), 0) \to \mathbb R^d$ of $x_{0, k}$ such that $\overline x_0$ and $\overline x_{0, k}$ are $(\mathfrak L, \mathfrak B)$-measurable and
\[
\begin{aligned}
\sup_{t \in \lvert -h(0), 0)} \abs*{\overline x_0(t)} & = \norm{x_0}_{L^\infty(\lvert -h(0), 0], \mathbb R^d)}, \\
\sup_{t \in \lvert -h(0), 0)} \abs*{\overline x_{0, k}(t)} & = \norm{x_{0, k}}_{L^\infty(\lvert -h(0), 0], \mathbb R^d)}, \\
\lim_{k \to +\infty} \sup_{t \in \lvert -h(0), 0)} \abs{\overline x_{0, k}(t) - \overline x_0(t)} & = 0.
\end{aligned}
\]

Let $\overline x\colon \lvert -h(0), +\infty) \to \mathbb R^d$ be the unique solution of \eqref{eq:diff-eqn-N1-time} with initial condition $\overline x_{0}$ and, for $k \in \mathbb N$, let $\overline x_k\colon \lvert -h(0), +\infty) \to \mathbb R^d$ be the unique solution of \eqref{eq:diff-eqn-N1-time-k} with initial condition $\overline x_{0, k}$. The existence of these solutions is asserted by Theorem~\ref{thm:exist-general} and, by Theorems~\ref{thm:exist-measurable} and \ref{thm:exist-ae}, we deduce that $\overline x$ and $\overline x_k$ are measurable a.e.\ solutions and that $\overline x_k$ is a representative of $x_k$. We denote by $x$ the equivalence class of $\overline x$ with respect to almost everywhere equality.

By Theorem~\ref{thm:cd-on-p-1}\ref{item:cd-uniform-compact}, for every $T > 0$, we have
\[
\lim_{k \to +\infty} \sup_{t \in \lvert -h(0), T]} \abs*{\overline x_k(t) - \overline x(t)} = 0.
\]
In particular, we deduce that $x \in L^\infty_{\mathrm{loc}}(\lvert -h(0), +\infty), \mathbb R^d)$, showing that $x$ is an $L^\infty$ solution of \eqref{eq:diff-eqn-N1-time} with initial condition $x_0$ (which is unique by Theorem~\ref{thm:exist-L-infty}), and $\lim_{k \to +\infty} \norm{x_{k} - x}_{L^\infty(\lvert -h(0), T], \mathbb R^d)} = 0$. The last part of the statement follows similarly from Theorem~\ref{thm:cd-on-p-1}\ref{item:cd-uniform-all}.
\end{proof}

We finally turn to the counterpart of Corollaries~\ref{col:cd-on-p-2}, \ref{col:cd-on-p-3}, and \ref{col:cd-on-p-Linfty} for $L^p$ solutions for $p \in [1, +\infty)$. Contrarily to the former results, we can no longer rely on Theorem~\ref{thm:cd-on-p-1} to obtain a result on continuous dependence on parameters in $L^p$ norm, since assuming convergence in $L^p$ of the sequence of initial conditions $(x_{0, k})_{k \in \mathbb N}$ to some $x_0$ is not sufficient to ensure the convergence of $(x_{0, k}(t))_{k \in \mathbb N}$ to $x_0(t)$ for almost every $t$.

\begin{theorem}
\label{thm:cd-on-p-Lp}
Let $p \in [1, +\infty)$, $(A_k)_{k \in \mathbb N}$ be a sequence of matrices in $\mathcal M_d(\mathbb R)$ converging to some $A \in \mathcal M_d(\mathbb R)$, $\tau$ be a delay function satisfying \ref{hypo:tau-infimum} and \ref{hypo:radon-nik-bdd}, $h$ be its associated largest delay function, and, for each $k \in \mathbb N$, $x_k \in L^p_{\mathrm{loc}}(\lvert -h(0), +\infty), \mathbb R^d)$ be an $L^p$ solution of \eqref{eq:diff-eqn-N1-time-k} with initial condition $x_{0, k} \in L^p(\lvert -h(0), 0], \mathbb R^d)$. Assume that there exists $x_0 \in L^p(\lvert -h(0), 0], \mathbb R^d)$ such that $\lim_{k \to +\infty} \norm{x_{0, k} - x_0}_{L^p(\lvert -h(0), 0], \mathbb R^d)} = 0$. Then, for every $T > 0$, we have $\lim_{k \to +\infty} \norm{x_{k} - x}_{L^p(\lvert -h(0), T], \mathbb R^d)} = 0$, where $x \in L^p_{\mathrm{loc}}(\lvert -h(0), +\infty), \mathbb R^d)$ is the unique $L^p$ solution of \eqref{eq:diff-eqn-N1-time} with initial condition $x_0$. In addition, if \ref{hypo:radon-nik-and-rho} is satisfied, then $\lim_{k \to +\infty} \norm{x_{k} - x}_{L^p(\lvert -h(0), +\infty), \mathbb R^d)} = 0$.
\end{theorem}

\begin{proof}
Let $x \in L^p_{\mathrm{loc}}(\lvert -h(0), +\infty), \mathbb R^d)$ be the $L^p$ solution of \eqref{eq:diff-eqn-N1-time} with initial condition $x_0$, which exists and is unique thanks to Theorem~\ref{thm:exist-L-p}. Let $(D_n)_{n \in \mathbb N}$ be the nonincreasing sequence of sets from Definition~\ref{def:Dn-sigman}.

Let $T > 0$ be given and $N_T > 0$ be as in the statement of Lemma~\ref{lemma:n-loc-bounded}. Set $\overline \sigma_1 = \sigma_1\rvert_{[0, T]}$ and denote by $\varphi_T$ the Radon--Nikodym derivative of the measure ${\overline\sigma_1}_{\#}\mathcal L$ with respect to $\mathcal L$, which exists and is bounded thanks to \ref{hypo:radon-nik-bdd}. For $(k, n) \in \mathbb N^2$, define $I_k^n$ by
\[
I_k^n = \int_{\lvert -h(0), T] \cap (D_n \setminus D_{n+1})} \abs*{x_k(t) - x(t)}^p \diff t
\]
and remark that, by Lemma~\ref{lemma:n-loc-bounded}, if $n > N_T$, we then have $\lvert -h(0), T] \cap (D_n \setminus D_{n+1}) = \emptyset$, and thus $I_k^n = 0$. In particular, this implies that
\begin{equation}
\label{eq:estim-Lp-sum}
\norm{x_k - x}_{L^p(\lvert -h(0), T], \mathbb R^d)}^p = \sum_{n = 0}^{N_T} I_k^n.
\end{equation}

Note that, for every $k \in \mathbb N$, we have
\[
I_k^0 = \int_{-h(0)}^0 \abs*{x_k(t) - x(t)}^p \diff t = \norm*{x_{0, k} - x_0}_{L^p(\lvert -h(0), 0], \mathbb R^d)}^p.
\]
Moreover, for $k \in \mathbb N$ and $n \in \llbracket 1, N_T\rrbracket$, we have
\[
I_k^n = \int_{[0, T] \cap (D_n \setminus D_{n+1})} \abs*{A_k x_k(\overline\sigma_1(t)) - A x(\overline\sigma_1(t))}^p \diff t,
\]
where we used \eqref{eq:diff-eqn-N1-time}, \eqref{eq:diff-eqn-N1-time-k}, and the fact that $D_n \subset \mathbb R_+$ for $n \in \mathbb N^\ast$. Using Lemma~\ref{lemma:D_n-sigma1}, one can verify that $[0, T] \cap (D_n \setminus D_{n+1}) \subset \overline\sigma_1^{-1}\left(\lvert -h(0), T] \cap (D_{n-1} \setminus D_{n})\right)$, and thus
\begin{align*}
I_k^n & \leq \int_{\overline\sigma_1^{-1}\left(\lvert -h(0), T] \cap (D_{n-1} \setminus D_{n})\right)} \abs*{A_k x_k(\overline\sigma_1(t)) - A x(\overline\sigma_1(t))}^p \diff t \displaybreak[0] \\
& = \int_{\lvert -h(0), T] \cap (D_{n-1} \setminus D_{n})} \abs*{A_k x_k(t) - A x(t)}^p \diff {\overline\sigma_1}_{\#}\mathcal L(t) \displaybreak[0] \\
& = \int_{\lvert -h(0), T] \cap (D_{n-1} \setminus D_{n})} \varphi_T(t) \abs*{A_k x_k(t) - A x(t)}^p \diff t \displaybreak[0] \\
& \leq 2^{p-1} \norm{\varphi_T}_{L^\infty(\lvert -h(0), T], \mathbb R)} \left[\abs*{A_k}^p I_k^{n-1} + \abs{A_k - A}^p \int_{\lvert -h(0), T] \cap (D_{n-1} \setminus D_{n})} \abs*{x(t)}^p \diff t\right].
\end{align*}
An immediate inductive argument on $n$ yields that there exists a constant $C > 1$, depending only on $p$, $N_T$, $\norm{\varphi_T}_{L^\infty(\lvert -h(0), T], \mathbb R)}$, and $\sup_{k \in \mathbb N} \abs{A_k}$, such that, for every $k \in \mathbb N$ and $n \in \llbracket 1, N_T\rrbracket$, we have
\begin{align*}
I_k^n & \leq C I_k^{0} + C \abs{A_k - A}^p \sum_{j = 1}^n \int_{\lvert -h(0), T] \cap (D_{j-1} \setminus D_{j})} \abs*{x(t)}^p \diff t \displaybreak[0] \\
& \leq C I_k^0 + C \abs{A_k - A}^p \norm{x}_{L^p(\lvert -h(0), T], \mathbb R^d)}^p.
\end{align*}
Inserting into \eqref{eq:estim-Lp-sum}, we deduce that
\begin{align*}
\norm{x_k - x}_{L^p(\lvert -h(0), T], \mathbb R^d)}^p & \leq C (N_T + 1) \norm*{x_{0, k} - x_0}_{L^p(\lvert -h(0), 0], \mathbb R^d)}^p \\
& \hphantom{\leq {}} + C (N_T + 1) \abs{A_k - A}^p \norm{x}_{L^p(\lvert -h(0), T], \mathbb R^d)}^p,
\end{align*}
yielding that
\[
\lim_{k \to +\infty} \norm{x_k - x}_{L^p(\lvert -h(0), T], \mathbb R^d)}^p = 0,
\]
as required.

To prove the last part of the statement, assume that \ref{hypo:radon-nik-and-rho} is satisfied, let $\varphi$ be as in that assumption, and choose a norm $\abs{\cdot}$ in $\mathbb R^d$ such that $\norm{\varphi}_{L^\infty(\mathbb R, \mathbb R)} \abs{A}^p < 1$. Since $A_k \to A$ as $k \to +\infty$, there exists $K \in \mathbb N$ such that $\norm{\varphi}_{L^\infty(\mathbb R, \mathbb R)} \sup_{\substack{k \in \mathbb N \\ k \geq K}} \abs{A_k}^p < 1$. In particular, $\norm{\varphi}_{L^\infty(\mathbb R, \mathbb R)} \rho(A_k)^p < 1$ for every $k \in \mathbb N$ such that $k \geq K$. We fix $\delta > 0$ such that
\[
(1 + \delta)^{p-1} \norm{\varphi}_{L^\infty(\mathbb R, \mathbb R)} \sup_{\substack{k \in \mathbb N \\ k \geq K}} \abs{A_k}^p < 1,
\]
and we denote by $c \in (0, 1)$ the left-hand side of the above inequality.

We now refine the proof of the first part of the statement as follows. For $(k, n) \in \mathbb N^2$, define $J_k^n$ by
\[
J_k^n = \int_{\lvert -h(0), +\infty) \cap (D_n \setminus D_{n+1})} \abs*{x_k(t) - x(t)}^p \diff t
\]
and note that, by Lemma~\ref{lemma:intersect-Dn-empty}, we have
\begin{equation}
\label{eq:estim-Lp-sum-infty}
\norm{x_k - x}_{L^p(\lvert -h(0), +\infty), \mathbb R^d)}^p = \sum_{n = 0}^{+\infty} J_k^n.
\end{equation}
In addition, applying Lemma~\ref{lemma:x-in-Lp} to \eqref{eq:diff-eqn-N1-time} and to \eqref{eq:diff-eqn-N1-time-k}, we deduce that $x \in L^p(\lvert -h(0), +\infty), \mathbb R^d)$ and $x_k \in L^p(\lvert -h(0), +\infty), \mathbb R^d)$ for every $k \in \mathbb N$ with $k \geq K$. In particular, for all such $k$, the left-hand side of \eqref{eq:estim-Lp-sum-infty} is finite, and hence so is the sum in the right-hand side of \eqref{eq:estim-Lp-sum-infty}.

As before, for every $k \in \mathbb N$, we have
\begin{equation}
\label{eq:Jk0}
J_k^0 = \norm*{x_{0, k} - x_0}_{L^p(\lvert -h(0), 0], \mathbb R^d)}^p.
\end{equation}
Moreover, proceeding as in the proof of the first part of the statement, for $k \in \mathbb N$ with $k \geq K$ and $n \in \mathbb N^\ast$, we have
\begin{align*}
J_k^n & = \int_{[0, +\infty) \cap (D_n \setminus D_{n+1})} \abs*{A_k x_k(\overline\sigma_1(t)) - A x(\overline\sigma_1(t))}^p \diff t \\
& \leq \int_{\sigma_1^{-1}\left(\lvert -h(0), +\infty) \cap (D_{n-1} \setminus D_{n})\right)} \abs*{A_k x_k(\overline\sigma_1(t)) - A x(\overline\sigma_1(t))}^p \diff t \displaybreak[0] \\
& = \int_{\lvert -h(0), +\infty) \cap (D_{n-1} \setminus D_{n})} \abs*{A_k x_k(t) - A x(t)}^p \diff {\sigma_1}_{\#}\mathcal L(t) \displaybreak[0] \\
& = \int_{\lvert -h(0), +\infty) \cap (D_{n-1} \setminus D_{n})} \varphi(t) \abs*{A_k x_k(t) - A x(t)}^p \diff t \displaybreak[0] \\
& \leq c J_k^{n-1} + \left(1 + \frac{1}{\delta}\right)^{p-1} \norm{\varphi}_{L^\infty(\mathbb R, \mathbb R)} \abs{A_k - A}^p \int_{\lvert -h(0), +\infty) \cap (D_{n-1} \setminus D_{n})} \abs*{x(t)}^p \diff t,
\end{align*}
where we have used the inequality $(a + b)^p \leq (1 + \delta)^{p - 1} a^p + \left(1 + \frac{1}{\delta}\right)^{p - 1} b^p$ for $a$ and $b$ in $\mathbb R_+$, which is a consequence of Jensen's inequality. We now take the sum of the above inequality over $n \in \mathbb N^\ast$ and add $J_k^0$ on both sides, which gives, for $k \in \mathbb N$ with $k \geq K$,
\[
\sum_{n=0}^{+\infty} J_k^n \leq c \sum_{n=0}^{+\infty} J_k^n + \left(1 + \frac{1}{\delta}\right)^{p-1} \norm{\varphi}_{L^\infty(\mathbb R, \mathbb R)} \abs{A_k - A}^p \norm{x}_{L^p(\lvert -h(0), +\infty), \mathbb R^d)}^p + J_k^0.
\]
Using the facts that $\sum_{n=0}^{+\infty} J_k^n < +\infty$ and $c \in (0, 1)$ and combining with \eqref{eq:estim-Lp-sum-infty} and \eqref{eq:Jk0}, we finally obtain that, for every $k \in \mathbb N$ with $k \geq K$, we have
\begin{multline*}
\norm{x_k - x}_{L^p(\lvert -h(0), +\infty), \mathbb R^d)}^p \leq \left(1 + \frac{1}{\delta}\right)^{p-1} \frac{\norm{\varphi}_{L^\infty(\mathbb R, \mathbb R)} \norm{x}_{L^p(\lvert -h(0), +\infty), \mathbb R^d)}^p}{1 - c} \abs{A_k - A}^p \\ + \frac{1}{1 - c} \norm*{x_{0, k} - x_0}_{L^p(\lvert -h(0), 0], \mathbb R^d)}^p.
\end{multline*}
The conclusion follows by letting $k \to +\infty$.
\end{proof}

\section{Asymptotic behavior of solutions}
\label{sec:stability}

Now that we have established results on the existence of solutions of \eqref{eq:diff-eqn-N1-time} in Section~\ref{sec:well-posed}, we turn to the analysis of their asymptotic behavior, with the aim of providing sufficient (and, whenever possible, necessary) conditions for convergence of solutions to $0$.

\subsection{Convergence to the origin}
\label{sec:asymptotic}

Our first result provides sufficient conditions for solutions of \eqref{eq:diff-eqn-N1-time} with bounded initial condition to converge to $0$.

\begin{theorem}
\label{thm:cv-to-0}
Let $A \in \mathcal M_d(\mathbb R)$ satisfy \ref{hypo:spr-A-less-1}, $\tau$ be a delay function satisfying \ref{hypo:tau-infimum} and \ref{hypo:delay-infinity}, $h$ be its associated largest delay function, and $x\colon \lvert -h(0), +\infty) \to \mathbb R^d$ be a solution of \eqref{eq:diff-eqn-N1-time} with bounded initial condition $x_0\colon \lvert -h(0), 0) \to \mathbb R^d$. Then $x(t) \to 0$ as $t \to +\infty$.
\end{theorem}

\begin{proof}
Since $\rho(A) < 1$, there exists a norm $\abs{\cdot}$ in $\mathbb R^d$ such that $\abs{A} < 1$.

We first prove that $\limsup_{t \to +\infty} \abs{x(t)} < +\infty$. Let $M > 0$ be such that $\abs{x_0(t)} \leq M$ for every $t \in \lvert -h(0), 0)$. By Theorem~\ref{theorem-1}, we then have, for every $t \in \mathbb R_+$,
\[
\abs{x(t)} \leq \abs{A^{\mathbf n(t)}} \abs{x_0(\sigma_{\mathbf n(t)}(t))} \leq M,
\]
where $\mathbf n$ and $(\sigma_n)_{n \in \mathbb N}$ are defined as in Section~\ref{sec:explicit-formula}. Hence $\limsup_{t \to +\infty} \abs{x(t)} \leq M$.

For every $t \geq 0$, using \eqref{eq:diff-eqn-N1-time}, we have \(\abs{x(t)} \leq \abs{A} \abs{x(t - \tau(t))}\). In particular,
\[
\limsup_{t \to +\infty}\abs{x(t)} \leq \abs{A} \limsup_{t \to +\infty}\abs{x(t - \tau(t))},
\]
and, thanks to \ref{hypo:delay-infinity}, we deduce that
\[
\limsup_{t \to +\infty}\abs{x(t)} \leq \abs{A} \limsup_{t \to +\infty}\abs{x(t)}.
\]
Since $\abs{A} < 1$ and $\limsup_{t \to +\infty} \abs{x(t)} < +\infty$, we deduce that
\[\limsup_{t \to +\infty}\abs{x(t)} = 0,\]
yielding the conclusion.
\end{proof}

Note that convergence to $0$ in Theorem~\ref{thm:cv-to-0} holds in the sense of the pointwise value $x(t)$ converging to $0$ in the topology of $\mathbb R^d$. If we consider a continuous solution of \eqref{eq:diff-eqn-N1-time}, we also have convergence to $0$ in the norm of $\mathsf C_t^{\mathrm b}$.

\begin{corollary}
\label{Corol-1-C}
Let $A \in \mathcal M_d(\mathbb R)$ satisfy \ref{hypo:spr-A-less-1}, $\tau$ be a delay function satisfying \ref{hypo:tau-infimum} and \ref{hypo:delay-infinity}, $h$ be its associated largest delay function, and $x \in C(\lvert -h(0), +\infty), \mathbb R^d)$ be a continuous solution of \eqref{eq:diff-eqn-N1-time} with bounded initial condition $x_0 \in \mathsf C_0^{\mathrm b}$. Then $\norm{x_t}_{\mathsf C_t^{\mathrm b}} \to 0$ as $t \to +\infty$.
\end{corollary}

\begin{proof}
By Definition~\ref{def:h} and assumption \ref{hypo:delay-infinity}, we have $t - h(t) = \inf_{s \in [t, +\infty)} \sigma_1(s) \to +\infty$ as $t \to +\infty$. Hence, by Theorem~\ref{thm:cv-to-0}, we have
\[
\norm{x_t}_{\mathsf C_t^{\mathrm b}} = \sup_{s \in \lvert -h(t), 0]} \abs{x_t(s)} = \sup_{s \in \lvert t - h(t), t]} \abs{x(s)} \to 0
\]
as $t \to +\infty$.
\end{proof}

The same idea used in the above proof also works for regulated solutions, yielding the following result.

\begin{corollary}
\label{Prop-1-C}
Let $A \in \mathcal M_d(\mathbb R)$ satisfy \ref{hypo:spr-A-less-1}, $\tau$ be a delay function satisfying \ref{hypo:tau-infimum} and \ref{hypo:delay-infinity}, $h$ be its associated largest delay function, and $x \in G(\lvert -h(0), +\infty), \mathbb R^d)$ be a regulated solution of \eqref{eq:diff-eqn-N1-time} with bounded initial condition $x_0 \in \mathsf G_0^{\mathrm b}$. Then $\norm{x_t}_{\mathsf G_t^{\mathrm b}} \to 0$ as $t \to +\infty$.
\end{corollary}

A similar result also holds true for solutions in $L^\infty$, but one must take additional care due to the fact that solutions in $L^\infty$ satisfy \eqref{eq:diff-eqn-N1-time} only for a.e.\ $t \in \mathbb R_+$.

\begin{corollary}
\label{coro:cv-to-0-Linfty}
Let $A \in \mathcal M_d(\mathbb R)$ satisfy \ref{hypo:spr-A-less-1}, $\tau$ be a delay function satisfying \ref{hypo:tau-infimum}, \ref{hypo:null}, and \ref{hypo:delay-infinity}, $h$ be its associated largest delay function, and $x \in L^\infty_{\mathrm{loc}}(\lvert -h(0), +\infty),\allowbreak \mathbb R^d)$ be an $L^\infty$ solution of \eqref{eq:diff-eqn-N1-time} with bounded initial condition $x_0 \in L^\infty(\lvert -h(0), 0),\allowbreak \mathbb R^d)$. Then $\norm{x_t}_{L^\infty(\lvert -h(t), 0), \mathbb R^d)} \to 0$ as $t \to +\infty$.
\end{corollary}

\begin{proof}
Up to modifying $x_0$ on a set of Lebesgue measure zero, there exists $M > 0$ such that $\abs{x_0(t)} \leq M$ for every $t \in \lvert -h(0), 0)$. Denote by $\widetilde x\colon \lvert -h(0), +\infty) \to \mathbb R^d$ the unique solution of \eqref{eq:diff-eqn-N1-time} with initial condition $x_0$, whose existence is guaranteed by Theorem~\ref{thm:exist-general}. By Theorem~\ref{thm:cv-to-0}, we have $\widetilde x(t) \to 0$ as $t \to +\infty$. In addition, as in the proof of Corollary~\ref{Corol-1-C}, we have $t - h(t) \to +\infty$ as $t \to +\infty$, and thus $\sup_{s \in \lvert t - h(t), t)} \abs{\widetilde x(s)} \to 0$ as $t \to +\infty$. The conclusion follows since, by Theorem~\ref{thm:explicit-ae}, we have $\widetilde x = x$ a.e.\ in $\lvert -h(0), +\infty)$.
\end{proof}

To obtain convergence to $0$ of solutions, Theorem~\ref{thm:cv-to-0} requires assumptions \ref{hypo:spr-A-less-1} and \ref{hypo:delay-infinity}. We now show that \ref{hypo:spr-A-less-1} is also a necessary condition to have convergence of solutions to $0$, and the same is also true of \ref{hypo:delay-infinity} under the additional assumption that $A$ is not nilpotent.

\begin{proposition}
\label{prop:necessary-asympt}
Let $A \in \mathcal M_d(\mathbb R)$, $\tau$ be a delay function satisfying \ref{hypo:tau-infimum}, and $h$ be its associated largest delay function. Assume that, for every solution $x\colon \lvert -h(0), +\infty)\allowbreak \to \mathbb R^d$ of \eqref{eq:diff-eqn-N1-time} with bounded initial condition, we have $x(t) \to 0$ as $t \to +\infty$. Then \ref{hypo:spr-A-less-1} holds true. If, in addition, $A$ is not nilpotent, then \ref{hypo:delay-infinity} also holds true.
\end{proposition}

\begin{proof}
Let $\lambda \in \sigma(A)$ and $v \in \mathbb C^d \setminus \{0\}$ be an eigenvector of $A$ associated with $\lambda$. Consider the solution\footnote{Although we have so far only considered solutions taking values in $\mathbb R^d$, it is immediate to verify that the previous results extend with the same proofs to solutions taking values in $\mathbb C^d$.} $x\colon \lvert -h(0), +\infty) \to \mathbb C^d$ of \eqref{eq:diff-eqn-N1-time} with constant initial condition equal to $v$ and note that, due to the linearity of \eqref{eq:diff-eqn-N1-time}, the functions $t \mapsto \Real x(t)$ and $t \mapsto \Imag x(t)$ are solutions of \eqref{eq:diff-eqn-N1-time} taking values in $\mathbb R^d$ and with bounded initial conditions. Hence, by assumption, $\Real x(t) \to 0$ and $\Imag x(t) \to 0$ as $t \to +\infty$, showing that $x(t) \to 0$ as $t \to +\infty$. On the other hand, using Theorem~\ref{theorem-1}, we compute $x(t) = \lambda^{\mathbf{n}(t)} v$ for every $t \in \mathbb R_+$, hence $\lambda^{\mathbf n(t)} \to 0$ as $t \to +\infty$, implying that $\abs{\lambda} < 1$. As a consequence, $A$ satisfies \ref{hypo:spr-A-less-1}.

If, in addition, $A$ is not nilpotent, then there exists $\lambda \in \sigma(A)$ with $\lambda \neq 0$. Defining $v \in \mathbb C^d \setminus \{0\}$ and $x\colon \lvert -h(0), +\infty) \to \mathbb C^d$ from $\lambda$ as before and recalling that $x(t) \to 0$ as $t \to +\infty$ and $x(t) = \lambda^{\mathbf{n}(t)} v$ for every $t \in \mathbb R_+$, we obtain in particular that $x(\sigma_1(t)) = \lambda^{\mathbf{n}(\sigma_1(t))} v$ for every $t \in \mathbb R_+$. Indeed, this is a consequence of the previous fact if $\sigma_1(t) \geq 0$, and it is trivially verified if $\sigma_1(t) < 0$ since $\mathbf n(s) = 0$ for every $s \in \mathbb R_-^\ast$. Thus, for every $t \in \mathbb R_+$, we have $x(t) = A x(\sigma_1(t)) = \lambda^{\mathbf{n}(\sigma_1(t)) + 1} v$, showing that $\lambda^{\mathbf{n}(\sigma_1(t)) + 1} \to 0$ as $t \to +\infty$ since $v \neq 0$. As $\lambda \neq 0$, we deduce that $\mathbf{n}(\sigma_1(t)) \to +\infty$ as $t \to +\infty$, and thus, by Lemma~\ref{lemma:n-loc-bounded}, we conclude that $\sigma_1(t) \to +\infty$ as $t \to +\infty$, as required.
\end{proof}

A natural question is whether Corollary~\ref{coro:cv-to-0-Linfty} remains true if $L^\infty$ is replaced with $L^p$ for some $p \in [1, +\infty)$. Our next example shows that this is not the case: in dimension $d = 1$, for every $p \in [1, +\infty)$ and $A \in (-1, 1) \setminus \{0\}$, there exists a delay function $\tau$ satisfying \ref{hypo:tau-infimum}, \ref{hypo:radon-nik-bdd}, and \ref{hypo:delay-infinity}, such that, for every initial condition $x_0 \in L^p(\lvert -h(0), 0], \mathbb R)$, the unique solution $x \in L^p_{\mathrm{loc}}(\lvert -h(0), +\infty), \mathbb R)$ of \eqref{eq:diff-eqn-N1-time} does not converge to $0$ in $L^p$ norm as $t \to +\infty$.

\begin{example}\label{example-3.1}
Consider \eqref{eq:diff-eqn-N1-time} in dimension $d = 1$ and let $p \in [1, +\infty)$, $A \in (-1, 1) \setminus \{0\}$, $\alpha \in (1 - \abs{A}^{p}, 1)$, and $\tau$ be the delay function given for $t \in \mathbb R_+$ by $\tau(t) = \alpha t + 1$. Note that $\sigma_1(t) = (1 - \alpha) t - 1$ for every $t \in \mathbb R_+$ and that $\tau$ satisfies \ref{hypo:tau-infimum}, \ref{hypo:radon-nik-bdd}, and \ref{hypo:delay-infinity}. In addition, since $\sigma_1$ is nondecreasing, an immediate computation using Definition~\ref{def:h} shows that the largest delay function associated with $\tau$ is $h = \tau$.

Let $x_0 \in L^p([-1, 0), \mathbb R) \setminus \{0\}$ and $x$ be the unique $L^p$ solution of \eqref{eq:diff-eqn-N1-time} with initial condition $x_0$, which exists thanks to Theorem~\ref{thm:exist-L-p}. For every $t \geq \frac{1}{1 - \alpha}$, we compute
\begin{align*}
\norm{x_t}_{L^p([-h(t), 0), \mathbb R)}^p & = \int_{t - h(t)}^t \abs{x(s)}^p \diff s = \abs{A}^p \int_{\sigma_1(t)}^t \abs{x(s - \tau(s))}^p \diff s \\
& = \frac{\abs{A}^p}{1 - \alpha} \int_{\sigma_1(t) - \tau(\sigma_1(t))}^{\sigma_1(t)} \abs{x(s)}^p \diff s = \frac{\abs{A}^p}{1 - \alpha} \norm*{x_{\sigma_1(t)}}_{L^p([-h(\sigma_1(t)), 0), \mathbb R)}^p.
\end{align*}
Let $t_0 = 0$ and $t_n = \frac{1 + t_{n-1}}{1 - \alpha}$ for $n \in \mathbb N^\ast$. Noticing that $\sigma_1(t_n) = t_{n-1}$ and that $t_n = \sum_{k=1}^n \frac{1}{(1 - \alpha)^k}$ for every $n \in \mathbb N^\ast$, we deduce that $t_n \to +\infty$ as $n \to +\infty$ and thus
\[
\norm{x_{t_n}}_{L^p([-h(t_n), 0), \mathbb R)}^p = \left(\frac{\abs{A}^p}{1 - \alpha}\right)^n \norm{x_0}_{L^p([-1, 0), \mathbb R)}^p \to +\infty
\]
as $n \to +\infty$, since $\frac{\abs{A}^p}{1 - \alpha} > 1$ and $\norm{x_0}_{L^p([-1, 0), \mathbb R)} > 0$. Hence the conclusion of Corollary~\ref{coro:cv-to-0-Linfty} does not hold when $L^\infty$ is replaced by $L^p$.

One might wonder whether the previous divergence to $+\infty$ of the norm is related to the fact that $h(t_n) = \tau(t_n) \to +\infty$ as $n \to +\infty$, i.e., the norm $\norm{x_{t_n}}_{L^p([-h(t_n), 0), \mathbb R)}$ is being computed in intervals whose sizes tend to $+\infty$ as $n \to +\infty$. To show that this is not the case, let $\beta \in (0, 1/p)$ be such that $(1 - \alpha)^{\beta p} < \abs{A}^p$, which exists since $1 - \alpha < \abs{A}^p$, and consider the initial condition $x_0$ given by $x_0(s) = \frac{1}{\abs{s}^\beta}$, which belongs to $L^p([-1, 0), \mathbb R)$ since $\beta < 1/p$. For every $t \geq \frac{1}{1 - \alpha}$, $n \in \mathbb N$, and $a \in (0, h(t)]$, we have
\begin{align*}
\norm{x_t}_{L^p([-a, 0), \mathbb R)}^p & = \abs{A}^p \int_{t - a}^t \abs{x(s - \tau(s))}^p \diff s = \frac{\abs{A}^p}{1 - \alpha} \int_{\sigma_1(t-a)}^{\sigma_1(t)} \abs{x(s)}^p \diff s \\
& = \frac{\abs{A}^p}{1 - \alpha} \int_{\sigma_1(t) - a(1 - \alpha)}^{\sigma_1(t)} \abs{x(s)}^p \diff s = \frac{\abs{A}^p}{1 - \alpha} \norm*{x_{\sigma_1(t)}}_{L^p([-a(1 - \alpha), 0), \mathbb R)}^p,
\end{align*}
and thus, by an immediate inductive argument, we deduce that
\[
\norm{x_{t_n}}_{L^p([-1, 0), \mathbb R)}^p = \left(\frac{\abs{A}^p}{1 - \alpha}\right)^n \norm*{x_0}_{L^p([-(1 - \alpha)^n, 0), \mathbb R)}^p,
\]
where $(t_n)_{n \in \mathbb N}$ is the same sequence as before. In addition,
\[
\norm*{x_0}_{L^p([-(1 - \alpha)^n, 0), \mathbb R)}^p = \int_{-(1 - \alpha)^n}^{0} \frac{1}{\abs{s}^{\beta p}} \diff s = \frac{(1 - \alpha)^{n(1-\beta p)}}{1 - \beta p},
\]
and hence
\[
\norm{x_{t_n}}_{L^p([-1, 0), \mathbb R)}^p = \frac{1}{1 - \beta p} \left(\frac{\abs{A}}{(1 - \alpha)^{\beta}}\right)^{p n} \to +\infty
\]
as $n \to +\infty$, since $\frac{\abs{A}}{(1 - \alpha)^\beta} > 1$.
\end{example}

Our next result provides sufficient conditions for convergence to the origin in $L^p$ with $p \in [1, +\infty)$.

\begin{theorem}
\label{Thm-1-L-P}
Let $p \in [1, +\infty)$, $A \in \mathcal M_d(\mathbb R)$, $\tau$ be a delay function satisfying \ref{hypo:tau-infimum}, \ref{hypo:delay-infinity}, and \ref{hypo:radon-nik-and-rho}, and $h$ be its associated largest delay function. Let $x \in L^p_{\mathrm{loc}}(\lvert -h(0),\allowbreak +\infty),\allowbreak \mathbb R^d)$ be an $L^p$ solution of \eqref{eq:diff-eqn-N1-time} with initial condition $x_0 \in L^p(\lvert -h(0),\allowbreak 0),\allowbreak \mathbb R^d)$. Then $x \in L^p(\lvert -h(0), +\infty), \mathbb R^d)$ and $\norm{x_t}_{L^p(\lvert -h(t), 0), \mathbb R^d)} \allowbreak\to 0$ as $t \to +\infty$.
\end{theorem}

\begin{proof}
The fact that $x \in L^p(\lvert -h(0), +\infty), \mathbb R^d)$ was already shown in Lemma~\ref{lemma:x-in-Lp}. Thanks to \ref{hypo:delay-infinity}, it follows from Definition~\ref{def:h} that $\lim_{t \to +\infty} t - h(t) = +\infty$, and thus, using the fact that $x \in L^p(\lvert -h(0), +\infty), \mathbb R^d)$, we deduce that
\[
\norm{x_t}_{L^p(\lvert -h(t), 0), \mathbb R^d)} = \norm{x(\cdot)}_{L^p(\lvert t-h(t), t), \mathbb R^d)} \leq \norm{x(\cdot)}_{L^p(\lvert t-h(t), +\infty), \mathbb R^d)} \xrightarrow[t \to +\infty]{} 0,
\]
as required.
\end{proof}

\begin{remark}
Under the assumptions of Theorem~\ref{Thm-1-L-P}, we have $\norm{\varphi}_{L^\infty(\mathbb R, \mathbb R)} \geq 1$, and in particular \ref{hypo:spr-A-less-1} is necessarily satisfied. Indeed, by \ref{hypo:delay-infinity}, there exists $T_0 \in \mathbb R_+^\ast$ such that $\sigma_1(t) \geq 0$ for every $t \in [T_0, +\infty)$. On the other hand, thanks to the definition of $\sigma_1$, we have $\sigma_1(t) < t$ for every $t \in \mathbb R_+$. In particular, for every $t \in [T_0, +\infty)$, we have $[T_0, t] \subset \sigma_1^{-1}([0, t])$. We also have
\[
\mathcal L(\sigma_1^{-1}([0, t])) = \int_0^t \diff {\sigma_1}_{\#} \mathcal L(s) = \int_0^t \varphi(s) \diff s \leq t \norm{\varphi}_{L^\infty(\mathbb R, \mathbb R)}.
\]
Hence, for every $t \in [T_0, +\infty)$, we have
\[
\norm{\varphi}_{L^\infty(\mathbb R, \mathbb R)} \geq \frac{1}{t} \mathcal L(\sigma_1^{-1}([0, t])) \geq \frac{1}{t} \mathcal L([T_0, t]) = \frac{t - T_0}{t} \xrightarrow[t \to +\infty]{} 1,
\]
showing that $\norm{\varphi}_{L^\infty(\mathbb R, \mathbb R)} \geq 1$. Together with \ref{hypo:radon-nik-and-rho}, one deduces that $\rho(A) < 1$.
\end{remark}

\begin{remark}
Recalling Remark~\ref{remk:radon-nikodym-derivative-tau}, one observes that, if the delay function $\tau$ is differentiable and there exists $\alpha \in (0, 1)$ such that $\tau'(t) \leq \alpha$ for every $t \in \mathbb R_+$, then $\sigma_1$ is a diffeomorphism between $\mathbb R_+$ and its image, and the Radon--Nikodym derivative of ${\sigma_1}_{\#}\mathcal L$ with respect to $\mathcal L$ is bounded by $\frac{1}{1 - \alpha}$. In this case, \ref{hypo:radon-nik-and-rho} is satisfied as soon as $\frac{\rho(A)^p}{1 - \alpha} < 1$.
\end{remark}

\begin{remark}
Example~\ref{example-3.1} illustrates the importance of \ref{hypo:radon-nik-and-rho} in Theorem~\ref{Thm-1-L-P}. Indeed, the delay function $\tau$ from Example~\ref{example-3.1} satisfies \ref{hypo:tau-infimum}, \ref{hypo:tau-measurable}, \ref{hypo:null}, and \ref{hypo:delay-infinity}, and we also have that ${\sigma_1}_{\#} \mathcal L$ is $\sigma$-finite. Recalling Remark~\ref{remk:radon-nikodym-derivative-tau}, the Radon--Nikodym derivative $\varphi$ of ${\sigma_1}_{\#}\mathcal L$ with respect to $\mathcal L$ is the function $\varphi$ given by $\varphi(t) = \frac{1}{1 - \alpha}$ if $t \geq -1$ and $\varphi(t) = 0$ otherwise. In particular, it is bounded, but we have $\norm{\varphi}_{L^\infty(\mathbb R, \mathbb R)} \rho(A)^p = \frac{\abs{A}^p}{1 - \alpha} > 1$, and the conclusion of Theorem~\ref{Thm-1-L-P} does not hold in that example.
\end{remark}

\subsection{Exponential stability}
\label{sec:exponential}

We now turn to the analysis of the exponential stability of \eqref{eq:diff-eqn-N1-time}, i.e., whether its solutions converge exponentially to zero in a given norm. The following technical assumption will be crucial in the analysis carried out here.

\begin{hypothesis}
Let $\tau$ be a delay function.
\begin{hypolist}
\item\label{hypo:lower-bound-n} The function $\tau$ satisfies \ref{hypo:tau-infimum} and there exist $\alpha > 0$ and $\beta \in \mathbb R$ such that $\mathbf n(t) \geq \alpha t + \beta$ for every $t \in \mathbb R_+$, where $\mathbf n\colon \mathbb R \to \mathbb N$ is the function defined from $\tau$ in Corollary~\ref{coro:integer-iterations}.
\end{hypolist}
\end{hypothesis}

Hypothesis~\ref{hypo:lower-bound-n} is not easy to verify in practice, since $\mathbf n$ may be difficult to compute explicitly. However, it is satisfied in the important particular case of a bounded delay function $\tau$, as stated in the following result.

\begin{proposition}
\label{prop:tau-bounded-implies-n-lower-bounded}
Let $\tau$ be a delay function satisfying \ref{hypo:delay-bounded} and assume that $\tau$ is upper bounded by some positive constant $\tau_{\max}$. Then $\tau$ satisfies \ref{hypo:lower-bound-n} with $\alpha = \frac{1}{\tau_{\max}}$ and $\beta = 0$.
\end{proposition}

\begin{proof}
Let $(D_n)_{n \in \mathbb N}$ be the sequence of sets from Definition~\ref{def:Dn-sigman}. We claim that, for every $n \in \mathbb N$, we have $D_n \supset [(n-1) \tau_{\max}, +\infty)$. Indeed, since $D_0 = \mathbb R$ and $D_1 = \mathbb R_+$, the previous property clearly holds for $n \in \{0, 1\}$. Assume now that $n \in \mathbb N$ is such that $D_n \supset [(n-1) \tau_{\max}, +\infty)$ and take $t \in [n \tau_{\max}, +\infty)$. Then $t - \tau(t) \in [(n-1) \tau_{\max}, +\infty) \subset D_n$ and, using \eqref{eq:defi-Dn}, we deduce that $t \in D_{n+1}$, yielding the desired result.

Now, for $t \in \mathbb R_+$, it follows from the definition of $\mathbf n$ that $t \notin D_{\mathbf n(t) + 1}$, implying that $t \notin [\mathbf n(t) \tau_{\max}, +\infty)$. Hence $t < \mathbf n(t) \tau_{\max}$, yielding the conclusion.
\end{proof}

Even though \ref{hypo:delay-bounded} is a sufficient condition for \ref{hypo:lower-bound-n}, it is not necessary, as shown in the following example.

\begin{example}
\label{expl:exponential-unbounded}
Let $\tau\colon \mathbb R_+ \to \mathbb R_+^\ast$ be the delay function defined by
\[
\tau(t) = \begin{dcases*}
k & if there exists $k \in \mathbb N^\ast$ such that $t \in [2^k, 2^k + 1)$, \\
1 & otherwise.
\end{dcases*}
\]
An illustration of the graph of $\tau$ is provided in Figure~\ref{fig:tau-expl}. Note that $\tau$ satisfies \ref{hypo:tau-infimum} and \ref{hypo:delay-infinity}, but not \ref{hypo:delay-bounded}.

\begin{figure}[ht]
\centering
\begin{tikzpicture}[x=0.2cm, y=0.2cm]
\draw[black!15!white, step=0.2cm] (-0.4, -0.4) grid (67.4, 6.4);
\draw[->, thick] (-1, 0) -- (68, 0) node[right] {$t$};
\draw[->, thick] (0, -1) -- (0, 7) node[left] {$\tau$};
\foreach \i in {5, 10, ..., 65} {
  \node[below] at ({\i}, 0) {\scriptsize $\i$};
}
\node[below left] at (0, 0) {\scriptsize $0$};
\node[left] at (0, 5) {\scriptsize $5$};

\draw[blue, thick] (0, 1) -- (2, 1);
\foreach \k in {1, 2, ..., 6} {
	\draw[blue, thick] ({2^\k}, {\k}) -- ({2^\k + 1}, {\k});
	\draw[blue, thick] ({2^\k + 1}, 1) -- ({min(2^(\k + 1), 67.4)}, 1);
}
\end{tikzpicture}
\caption{Illustration of the graph of the delay function $\tau$ from Example~\ref{expl:exponential-unbounded}.}
\label{fig:tau-expl}
\end{figure}
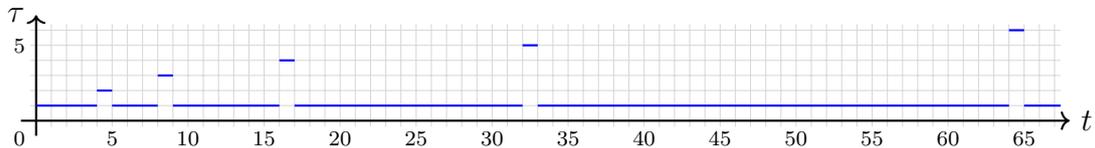

Let $\mathbf n$ be the function defined from $\tau$ as in Corollary~\ref{coro:integer-iterations} and note that, by Lemma~\ref{lemma:properties_n}\ref{item:n_sigma_1}, for every $t \in \mathbb R_+$, we have
\begin{equation}
\label{eq:n-recursive}
\mathbf n(t) = \begin{dcases*}
\mathbf n(t - k) + 1 & if there exists $k \in \mathbb N^\ast$ such that $t \in [2^k, 2^k + 1)$, \\
\mathbf n(t - 1) + 1 & otherwise.
\end{dcases*}
\end{equation}

We claim that
\begin{equation}
\label{eq:n-explicit}
\mathbf n(t) = \begin{dcases*}
0 & if $t < 0$, \\
1 & if $t \in [0, 1)$, \\
\floor{t} - \frac{1}{2}\left(\floor{\log_2 t} - 2\right)\left(\floor{\log_2 t} + 1\right) & if $t \geq 1$.
\end{dcases*}
\end{equation}
It follows from the definition of $\mathbf n$ that $\mathbf n(t) = 0$ for $t < 0$, and one easily deduces that $\mathbf n(t) = 1$ for $t \in [0, 1)$ using \eqref{eq:n-recursive}. To obtain the expression of $\mathbf n(t)$ for $t \geq 1$, we show by strong induction on $\ell$ that, for every $\ell \in \mathbb N^\ast$,
\begin{equation}
\label{eq:n-induction}
\mathbf n(t) = \floor{t} - \frac{1}{2}\left(\floor{\log_2 t} - 2\right)\left(\floor{\log_2 t} + 1\right) \qquad \text{for every } t \in [\ell, \ell + 1).
\end{equation}
Indeed, for $t \in [1, 2)$, we have $\mathbf n(t) = \mathbf n(t - 1) + 1 = 2$, while, on the other hand, $\floor{\log_2 t} = 0$, and thus \eqref{eq:n-induction} is satisfied for $\ell = 1$. Assume now that $\ell_0 \in \mathbb N^\ast$ with $\ell_0 \geq 2$ is such that \eqref{eq:n-induction} is satisfied for every $\ell \in \llbracket 1, \ell_0-1\rrbracket$ and take $t \in [\ell_0, \ell_0 + 1)$. Let $k_0 = \floor{\log_2 \ell_0}$. If $\ell_0 \in \llbracket 2^{k_0} + 1, 2^{k_0 + 1} - 1\rrbracket$, then $t \in [2^{k_0} + 1, 2^{k_0 + 1})$ and thus $\floor{\log_2 t} = \floor{\log_2 (t-1)} = k_0$. Using \eqref{eq:n-recursive} and \eqref{eq:n-induction} with $\ell = \ell_0 - 1$, we have
\begin{align*}
\mathbf n(t) & = \mathbf n(t-1) + 1 = \floor{t - 1} - \frac{1}{2}\left(k_0 - 2\right)\left(k_0 + 1\right) + 1 \\
& = \floor{t} - \frac{1}{2}\left(k_0 - 2\right)\left(k_0 + 1\right),
\end{align*}
showing \eqref{eq:n-induction} for $\ell_0$. Otherwise, we have $\ell_0 = 2^{k_0}$. In this case, $t - k_0 \in [2^{k_0} - k_0, 2^{k_0} - k_0 + 1) \subset [2^{k_0 - 1}, 2^{k_0})$ since $k_0 \geq 1$, and thus $\floor{\log_2 t} = k_0$ and $\floor{\log_2 (t-k_0)} = k_0 - 1$. Hence, using \eqref{eq:n-recursive} and \eqref{eq:n-induction} with $\ell = \ell_0 - k_0 \geq 1$, we have
\begin{align*}
\mathbf n(t) & = \mathbf n(t - k_0) + 1 = \floor{t - k_0} - \frac{1}{2}\left(k_0 - 3\right)k_0 + 1 \\
 & = \floor{t - k_0} - \frac{1}{2}\left(k_0 - 2\right)\left(k_0 + 1\right) + k_0 \\
 & = \floor{t} - \frac{1}{2}\left(k_0 - 2\right)\left(k_0 + 1\right),
\end{align*}
showing \eqref{eq:n-induction} for $\ell_0$ and concluding the proof of the inductive step.

Using \eqref{eq:n-explicit}, one can now prove that
\begin{equation}
\label{eq:n-lower-bound}
\mathbf n(t) \geq \frac{t}{2} - 1 \qquad \text{for every } t \in \mathbb R.
\end{equation}
Indeed, this is immediate for $t \in (-\infty, 1]$. For $t \geq 1$, since $\floor{t} \geq t - 1$, it suffices to prove that $\left(\floor{\log_2 t} - 2\right)\left(\floor{\log_2 t} + 1\right) \leq t$. We have $\left(\floor{\log_2 t} - 2\right)\left(\floor{\log_2 t} + 1\right) \leq \left(\floor{\log_2 t}\right)^2 - 2 \leq (\log_2 t)^2 - 2$ for every $t \geq 1$, and we can prove that $(\log_2 t)^2 - 2 \leq t$ for every $t \geq 1$ thanks to an analysis of the sign changes of the derivative of $t \mapsto t - (\log_2 t)^2 + 2$, yielding \eqref{eq:n-lower-bound}. Hence, $\tau$ satisfies \ref{hypo:lower-bound-n}.
\end{example}

Our first main result of this section is the following theorem, showing that \ref{hypo:tau-infimum}, \ref{hypo:spr-A-less-1}, and \ref{hypo:lower-bound-n} are sufficient conditions for the exponential convergence to $0$ of $\abs{x(t)}$ for solutions with bounded initial condition.

\begin{theorem}
\label{thm:exp-to-0}
Let $A \in \mathcal M_d(\mathbb R)$ satisfy \ref{hypo:spr-A-less-1}, $\tau$ be a delay function satisfying \ref{hypo:tau-infimum} and \ref{hypo:lower-bound-n}, and $h$ be its associated largest delay function. There exist positive constants $C$ and $\gamma$ such that, for every solution $x\colon \lvert -h(0), +\infty) \to \mathbb R^d$ of \eqref{eq:diff-eqn-N1-time} with a bounded initial condition $x_0\colon \lvert -h(0), 0) \to \mathbb R^d$, we have
\[\abs{x(t)} \leq C e^{-\gamma t} \sup_{s \in \lvert -h(0), 0)} \abs{x_0(s)} \qquad \text{ for every } t \in \mathbb R_+.\]
\end{theorem}

\begin{proof}
Let $\alpha$ and $\beta$ be the constants from \ref{hypo:lower-bound-n}, and let $\abs{\cdot}$ be a norm in $\mathbb R^d$ such that $\abs{A} < 1$, which exists thanks to \ref{hypo:spr-A-less-1}. Using Theorem~\ref{theorem-1}, \ref{hypo:lower-bound-n}, and the fact that $\abs{A} < 1$, we have, for every $t \in \mathbb R_+$,
\[
\abs{x(t)} \leq \abs{A}^{\mathbf n(t)} \abs{x_0(\sigma_{\mathbf n(t)} (t))} \leq \abs{A}^{\alpha t + \beta} \sup_{s \in \lvert -h(0), 0)} \abs{x_0(s)},
\]
yielding the conclusion with $C = \abs{A}^\beta > 0$ and $\gamma = -\alpha \ln\abs{A} > 0$. Note that, thanks to the equivalence of all norms in $\mathbb R^d$, the conclusion also holds true if $\abs{\cdot}$ is replaced by any other norm in $\mathbb R^d$, up to modifying the constant $C$.
\end{proof}

\begin{remark}
Hypothesis~\ref{hypo:lower-bound-n} is important in the proof of Theorem~\ref{thm:exp-to-0} in order to bound $\abs{A}^{\mathbf n(t)}$ by an exponential in $t$. However, it is not necessary: if, for instance, we replace \ref{hypo:lower-bound-n} by the more general lower bound $\mathbf n(t) \geq \alpha t + \beta - \ln p(t)$ for some $\alpha > 0$, $\beta \in \mathbb R$, and a polynomial $p$ satisfying $p(t) \geq 1$ for every $t \in \mathbb R_+$, then $\abs{A}^{\mathbf n(t)}$ can be upper bounded by $C_0 p(t)^{\nu} e^{-\gamma_0 t}$, with $C_0 = \abs{A}^\beta > 0$, $\nu = -\ln\abs{A} > 0$, and $\gamma_0 = \alpha \nu > 0$, and $C_0 p(t)^{\nu} e^{-\gamma_0 t}$ can be upper bounded by $C e^{-\gamma t}$ for any choice of $\gamma \in (0, \gamma_0)$, by choosing $C = \max_{t \in \mathbb R_+} C_0 p(t)^\nu e^{-(\gamma_0 - \gamma) t}$.
\end{remark}

As a consequence of Theorem~\ref{thm:exp-to-0}, we obtain the following criterion for the exponential convergence to $0$ of continuous solutions of \eqref{eq:diff-eqn-N1-time}.

\begin{corollary}
\label{coro:exp-cont}
Let $A \in \mathcal M_d(\mathbb R)$ satisfy \ref{hypo:spr-A-less-1}, $\tau$ be a delay function satisfying \ref{hypo:tau-infimum} and \ref{hypo:delay-bounded}, and $h$ be its associated largest delay function. Then there exist positive constants $C$ and $\gamma$ such that, for every continuous solution $x \in C(\lvert -h(0), +\infty), \mathbb R^d)$ of \eqref{eq:diff-eqn-N1-time} with a bounded initial condition $x_0 \in \mathsf C_0^{\mathrm b}$, we have
\[\norm{x_t}_{\mathsf C_t^{\mathrm b}} \leq C e^{-\gamma t} \norm{x_0}_{\mathsf C_0^{\mathrm b}} \qquad \text{ for every } t \in \mathbb R_+.\]
\end{corollary}

\begin{proof}
Since $\tau$ satisfies \ref{hypo:tau-infimum} and \ref{hypo:delay-bounded}, it follows from Proposition~\ref{prop:tau-bounded-implies-n-lower-bounded} that \ref{hypo:lower-bound-n} also holds true. In addition, setting $\tau_{\max} = \sup_{t \in \mathbb R_+} \tau(t)$, we have from Definition~\ref{def:h} that $h(t) \leq \tau_{\max}$ for every $t \in \mathbb R_+$. Let $C_0$ and $\gamma$ denote the constants whose existence is asserted in Theorem~\ref{thm:exp-to-0} and assume, with no loss of generality, that $C_0 \geq 1$. Since $x$ is a continuous solution of \eqref{eq:diff-eqn-N1-time} with a bounded initial condition $x_0$, it follows from Theorem~\ref{thm:exp-to-0} that
\[\abs{x(t)} \leq C_0 e^{-\gamma t} \norm{x_0}_{\mathsf C_0^{\mathrm b}}\]
for every $t \in \mathbb R_+$ and, since $C_0 \geq 1$ and $\gamma > 0$, the above inequality trivially holds true also for $t \in \lvert -h(0), 0)$. Then, for every $t \in \mathbb R_+$, we have
\[
\norm{x_t}_{\mathsf C_t^{\mathrm b}} = \sup_{s \in \lvert t - h(t), t]} \abs{x(s)} \leq C_0 e^{-\gamma (t - h(t))} \norm{x_0}_{\mathsf C_0^{\mathrm b}} \leq C e^{-\gamma t} \norm{x_0}_{\mathsf C_0^{\mathrm b}},
\]
with $C = C_0 e^{\gamma \tau_{\max}}$.
\end{proof}

Assumption~\ref{hypo:delay-bounded} is useful in Corollary~\ref{coro:exp-cont} since it ensures that the function $h$ from Definition~\ref{def:h} is upper bounded by $\tau_{\max}$, and hence, in the above proof, one can bound $C_0 e^{-\gamma (t - h(t))}$ by $C_0 e^{\gamma \tau_{\max}} e^{-\gamma t}$. However, \ref{hypo:delay-bounded} is not necessary for the conclusions of Corollary~\ref{coro:exp-cont} to hold true, as illustrated in the following example.

\begin{example}
Let $\tau$ be the delay function from Example~\ref{expl:exponential-unbounded}, $h$ be its associated largest delay function, and $A \in \mathcal M_d(\mathbb R)$ be a matrix satisfying \ref{hypo:spr-A-less-1}. Note that $A$ and $\tau$ satisfy the assumptions of Theorem~\ref{thm:exp-to-0}, but not those of Corollary~\ref{coro:exp-cont}. We will prove, however, that the conclusion of Corollary~\ref{coro:exp-cont} remains true in this case.

Note that, by the definition of $\tau$, we have, for every $t \in \mathbb R_+$,
\begin{equation}
\label{eq:estim-inf}
\inf_{s \in [t, +\infty)} (s - \tau(s)) = \min\left(\inf_{s \in [t, +\infty) \setminus \bigcup_{k=1}^\infty [2^k, 2^k+1)} (s - 1), \inf_{\substack{(s, k) \in [t, +\infty) \times \mathbb N^\ast \\ \floor{s} = 2^k}} (s - k)\right).
\end{equation}
We clearly have
\begin{equation}
\label{eq:estim-inf-1}
\inf_{s \in [t, +\infty) \setminus \bigcup_{k=1}^\infty [2^k, 2^k+1)} (s - 1) \geq t - 1.
\end{equation}
On the other hand, for every $(s, k) \in [t, +\infty) \times \mathbb N^\ast$ with $\floor{s} = 2^k$, we have $s \geq 2$ and $k = \log_2 \floor{s} < \log_2 (s+1)$, and thus $s - k > s - \log_2 (s+1) \geq 2 - \log_2 3 > 0$ since $z \mapsto z - \log_2 (z+1)$ is increasing in $[1, +\infty)$. We then have
\begin{equation}
\label{eq:estim-inf-2}
\inf_{\substack{(s, k) \in [t, +\infty) \times \mathbb N^\ast \\ \floor{s} = 2^k}} (s - k) \geq t - \log_2 (t+1).
\end{equation}
Indeed, if $t \geq 1$, this follows from the fact that $z \mapsto z - \log_2 (z+1)$ is increasing in $[1, +\infty)$ while, if $t \in [0, 1)$, the inequality is trivial since $t - \log_2 (t+1) \leq 0$ but the left-hand side of \eqref{eq:estim-inf-2} is positive.

Inserting \eqref{eq:estim-inf-1} and \eqref{eq:estim-inf-2} into \eqref{eq:estim-inf} and using Definition~\ref{def:h}, we deduce that
\[
h(t) \leq 1 + \log_2(t+1) \qquad \text{ for every } t \in \mathbb R_+.
\]
Hence, applying Theorem~\ref{thm:exp-to-0} and proceeding as in the proof of Corollary~\ref{coro:exp-cont}, we deduce that, for every continuous solution $x \in C(\lvert -h(0), +\infty), \mathbb R^d)$ of \eqref{eq:diff-eqn-N1-time} with a bounded initial condition $x_0 \in \mathsf C_0^{\mathrm b}$, we have
\[\norm{x_t}_{\mathsf C_t^{\mathrm b}} \leq C_0 e^{-\gamma_0 (t - h(t))} \norm{x_0}_{\mathsf C_0^{\mathrm b}} \qquad \text{ for every } t \in \mathbb R_+,\]
where $C_0$ and $\gamma_0$ are the constants from Theorem~\ref{thm:exp-to-0}. For every $t \in \mathbb R_+$, we have $C_0 e^{-\gamma_0 (t - h(t))} \leq C_0 e^{\gamma_0} (t+1)^{\frac{1}{\ln 2}} e^{-\gamma_0 t} \leq C e^{-\gamma t}$, with an arbitrary $\gamma \in (0, \gamma_0)$ and $C > 0$ given by $C = C_0 e^{\gamma_0} \max_{t \in \mathbb R_+} (t+1)^{\frac{1}{\ln 2}} e^{-(\gamma_0 - \gamma) t}$. Hence, the conclusion of Corollary~\ref{coro:exp-cont} holds true.
\end{example}

The proof of Corollary~\ref{coro:exp-cont} can be straightforwardly adapted to regulated solutions, yielding the following result.

\begin{corollary}
\label{coro:exp-G}
Let $A \in \mathcal M_d(\mathbb R)$ satisfy \ref{hypo:spr-A-less-1}, $\tau$ be a delay function satisfying \ref{hypo:tau-infimum} and \ref{hypo:delay-bounded}, and $h$ be its associated largest delay function. Then there exist positive constants $C$ and $\gamma$ such that, for every regulated solution $x \in G(\lvert -h(0), +\infty), \mathbb R^d)$ of \eqref{eq:diff-eqn-N1-time} with a bounded initial condition $x_0 \in \mathsf G_0^{\mathrm b}$, we have
\[\norm{x_t}_{\mathsf G_t^{\mathrm b}} \leq C e^{-\gamma t} \norm{x_0}_{\mathsf G_0^{\mathrm b}} \qquad \text{ for every } t \in \mathbb R_+.\]
\end{corollary}

We also have the following counterpart of Corollaries~\ref{coro:exp-cont} and \ref{coro:exp-G} for solutions in $L^\infty$.

\begin{corollary}
\label{coro:exp-Linfty}
Let $A \in \mathcal M_d(\mathbb R)$ satisfy \ref{hypo:spr-A-less-1}, $\tau$ be a delay function satisfying \ref{hypo:tau-infimum}, \ref{hypo:null}, and \ref{hypo:delay-bounded}, and $h$ be its associated largest delay function. Then there exist positive constants $C$ and $\gamma$ such that, for every $L^\infty$ solution $x \in L^\infty_{\mathrm{loc}}(\lvert -h(0), +\infty), \mathbb R^d)$ of \eqref{eq:diff-eqn-N1-time} with an initial condition $x_0 \in L^\infty(\lvert -h(0), 0), \mathbb R^d)$, we have
\[\norm{x_t}_{L^\infty(\lvert -h(t), 0), \mathbb R^d)} \leq C e^{-\gamma t} \norm{x_0}_{L^\infty(\lvert -h(0), 0), \mathbb R^d)} \qquad \text{ for every } t \in \mathbb R_+.\]
\end{corollary}

The proof of Corollary~\ref{coro:exp-Linfty} is slightly more subtle than that of Corollaries~\ref{coro:exp-cont} and \ref{coro:exp-G}, since $L^\infty$ solutions only satisfy \eqref{eq:diff-eqn-N1-time} for almost every $t \in \mathbb R_+$. It is omitted here since it can be done through an easy adaptation of the argument already provided in the proof of Corollary~\ref{coro:cv-to-0-Linfty}: one can select a measurable bounded function $\widetilde x_0$ in the equivalence class of $x_0$ with respect to almost everywhere equality and consider the solution $\widetilde x$ of \eqref{eq:diff-eqn-N1-time} with initial condition $\widetilde x_0$, to which one can apply Theorem~\ref{thm:exp-to-0}. One obtains the exponential convergence to $0$ of $\norm{\widetilde x_t}_{L^\infty(\lvert -h(t), 0), \mathbb R^d)}$ by proceeding as in the proof of Corollary~\ref{coro:exp-cont}, and the conclusion of Corollary~\ref{coro:exp-Linfty} follows since $x = \widetilde x$ almost everywhere thanks to Theorem~\ref{thm:explicit-ae}.

Even though \ref{hypo:tau-infimum}, \ref{hypo:spr-A-less-1}, and \ref{hypo:delay-bounded} are sufficient to obtain exponential stability of \eqref{eq:diff-eqn-N1-time} for continuous and regulated solutions, and also for solutions in $L^\infty$ if one also assumes \ref{hypo:null}, they are not sufficient to ensure the exponential stability of \eqref{eq:diff-eqn-N1-time} in $L^p$ for $p \in [1, +\infty)$, as illustrated by the following example.

\begin{example}
Consider \eqref{eq:diff-eqn-N1-time} in dimension $d = 1$ and let $p \in [1, +\infty)$, $A \in (-1, 1) \setminus \{0\}$, and $\alpha \in (1 - \abs{A}^{p}, 1)$. Let $(t_n)_{n \in \mathbb N}$ be the sequence given by $t_n = \frac{n}{1 - \alpha}$ for $n \in \mathbb N$. Let $\tau$ be the delay function defined by $\tau(t) = \alpha (t - t_n) + 1$ if $t \in [t_n, t_{n+1})$ for some $n \in \mathbb N$. It is easy to verify that $\tau$ is well-defined, satisfies \ref{hypo:tau-infimum}, \ref{hypo:radon-nik-bdd}, and \ref{hypo:delay-bounded}, and in addition $\tau(t) \in [1, \frac{1}{1 - \alpha})$ for every $t \in \mathbb R_+$. We denote by $h$ the largest delay function associated with $\tau$.

For $n \in \mathbb N$ and $t \in [t_n, t_{n+1})$, we have $\sigma_1(t) = (1 - \alpha)(t - t_n) + t_n - 1$. Thus, $\sigma_1$ is increasing, since it is affine and increasing in any interval of the form $[t_n, t_{n+1})$, $n \in \mathbb N$, and its jump at $t_{n+1}$ is $\sigma_1(t_{n+1}) - \sigma_1(t_{n+1}^-) = t_{n+1} - 1 - t_n = \frac{1}{1 - \alpha} - 1 > 0$. In particular, it follows from Definition~\ref{def:h} that $h = \tau$. Using the above expression of $\sigma_1$, we also obtain easily that, for every $n \in \mathbb N$ and $t \in \mathbb R_+$, we have $t \in [t_n, t_{n+1})$ if and only if $\sigma_1(t) \in (t_{n-1}, t_n)$, with the convention $t_{-1} = -\infty$.

We claim that, for every $L^p$ solution $x$ of \eqref{eq:diff-eqn-N1-time}, $n \in \mathbb N$, and $k \in \llbracket 0, n\rrbracket$, we have
\begin{equation}
\label{eq:expl-Lp}
\int_{\sigma_1(t_n)}^{t_n} \abs{x(s)}^p \diff s = \left(\frac{\abs{A}^p}{1 - \alpha}\right)^k \int_{t_{n - k} - (1 - \alpha)^k}^{t_{n-k}} \abs{x(s)}^p \diff s.
\end{equation}
Indeed, this is true for $k = 0$ since $\sigma_1(t_n) = t_n - 1$. Assuming that $k \in \llbracket 0, n-1\rrbracket$ is such that \eqref{eq:expl-Lp} holds true, we remark that $t_{n-k} - (1 - \alpha)^k \geq t_{n-k} - \frac{1}{1 - \alpha} = t_{n - k - 1}$, and thus, for every $s \in (t_{n - k} - (1 - \alpha)^k, t_{n-k})$, we have in particular $s \in (t_{n-k-1}, t_{n-k})$, and thus $\sigma_1(s) = (1 - \alpha)(s - t_{n-k-1}) + t_{n-k-1} - 1$. Using \eqref{eq:diff-eqn-N1-time} and performing a change of variables, we obtain that
\begin{align*}
\int_{t_{n - k} - (1 - \alpha)^k}^{t_{n-k}} \abs{x(s)}^p \diff s & = \abs{A}^p \int_{t_{n - k} - (1 - \alpha)^k}^{t_{n-k}} \abs{x((1 - \alpha)(s - t_{n-k-1}) + t_{n-k-1} - 1)}^p \diff s \\
& = \frac{\abs{A}^p}{1 - \alpha} \int_{t_{n - k - 1} - (1 - \alpha)^{k+1}}^{t_{n-k-1}} \abs{x(s)}^p \diff s,
\end{align*}
and thus, inserting into \eqref{eq:expl-Lp}, we deduce that \eqref{eq:expl-Lp} holds with $k$ replaced by $k+1$, yielding the conclusion.

As in Example~\ref{example-3.1}, take $\beta \in (0, 1/p)$ such that $(1 - \alpha)^{\beta p} < \abs{A}^p$ and let $x_0\colon [-1, 0) \to \mathbb R$ be given for $s \in [-1, 0)$ by $x_0(s) = \frac{1}{\abs{s}^\beta}$. Thanks to the choice of $\beta$, we have $x_0 \in L^p([-1, 0), \mathbb R)$. Let $x$ be the unique $L^p$ solution of \eqref{eq:diff-eqn-N1-time} with initial condition $x_0$, which exists thanks to Theorem~\ref{thm:exist-L-p}. Given $n \in \mathbb N$, we apply \eqref{eq:expl-Lp} to $x$ with $k = n$, obtaining that
\begin{align*}
\norm{x_{t_n}}_{L^p([-h(t_n), 0), \mathbb R)}^p & = \int_{\sigma_1(t_n)}^{t_n} \abs{x(s)}^p \diff s = \left(\frac{\abs{A}^p}{1 - \alpha}\right)^n \int_{-(1 - \alpha)^n}^{0} \abs{x_0(s)}^p \diff s \\
& = \frac{1}{1 - \beta p} \left(\frac{\abs{A}^p}{(1 - \alpha)^{\beta p}}\right)^n \xrightarrow[n \to +\infty]{} +\infty.
\end{align*}
This shows that $\norm{x_t}_{L^p([-h(t), 0), \mathbb R)}^p$ cannot converge exponentially to $0$ as $t \to +\infty$.
\end{example}

Our last result in this section provides sufficient conditions for exponential stability in $L^p$.

\begin{theorem}
\label{thm:exp-Lp}
Let $p \in [1, +\infty)$, $A \in \mathcal M_d(\mathbb R)$, $\tau$ be a delay function satisfying \ref{hypo:tau-infimum}, \ref{hypo:radon-nik-and-rho}, and \ref{hypo:delay-bounded}, and $h$ be its associated largest delay function. Then there exist positive constants $C$ and $\gamma$ such that, for every $L^p$ solution $x \in L^p_{\mathrm{loc}}(\lvert -h(0), +\infty),\allowbreak \mathbb R^d)$ with initial condition $x_0 \in L^p(\lvert -h(0), 0), \mathbb R^d)$, we have $x \in L^p(\lvert -h(0), +\infty), \mathbb R^d)$ and
\begin{equation}
\label{eq:estim-exp-Lp}
\norm{x_t}_{L^p(\lvert -h(t), 0), \mathbb R^d)} \leq C e^{-\gamma t} \norm{x_0}_{L^p(\lvert -h(0), 0), \mathbb R^d)} \qquad \text{ for every } t \in \mathbb R_+.
\end{equation}
\end{theorem}

\begin{proof}
Note that the assumptions of Theorem~\ref{Thm-1-L-P} are satisfied and thus, in particular, the conclusion that $x \in L^p(\lvert -h(0), +\infty), \mathbb R^d)$ follows from that theorem. Note also that, thanks to \ref{hypo:delay-bounded}, it follows from Definition~\ref{def:h} that $h(t) \leq \tau_{\max}$ for every $t \in \mathbb R_+$, where $\tau_{\max} = \sup_{t \in \mathbb R_+} \tau(t)$.

Proceeding as in the proof of Theorem~\ref{Thm-1-L-P} and using its notations, we obtain the inequality
\[
\int_{D_\ast \cap (D_n \setminus D_{n+1})} \abs{x(t)}^p \diff t \leq \left(\norm{\varphi}_{L^\infty(\mathbb R, \mathbb R)} \abs{A}^p\right)^n \norm{x_0}_{L^p(\lvert -h(0), 0), \mathbb R^d)}^p
\]
for every $n \in \mathbb N$.

Note that the inequality in \eqref{eq:estim-exp-Lp} is trivially true for $t = 0$ if one chooses $C \geq 1$. Take $t > 0$ and set $I_t = (t - h(t), t)$. Using Proposition~\ref{prop:tau-bounded-implies-n-lower-bounded}, we have $\mathbf n(s) \geq \frac{s}{\tau_{\max}} \geq \frac{t - h(t)}{\tau_{\max}} \geq \frac{t}{\tau_{\max}} - 1$ for every $s \in I_t$ and thus, in particular, $I_t \subset D_{N_t}$, where $N_t = \ceil*{\frac{t}{\tau_{\max}}} - 1 \in \mathbb N$. Hence
\begin{align*}
\norm{x_t}_{L^p(\lvert -h(t), 0), \mathbb R^d)}^p & = \int_{I_t} \abs{x(s)}^p \diff s = \sum_{n = N_t}^{\infty} \int_{I_t \cap (D_n \setminus D_{n+1})} \abs{x(s)}^p \diff s \\
& \leq \sum_{n = N_t}^{\infty} \int_{D_\ast \cap (D_n \setminus D_{n+1})} \abs{x(s)}^p \diff s \\
& \leq \sum_{n = N_t}^{\infty} \left(\norm{\varphi}_{L^\infty(\mathbb R, \mathbb R)} \abs{A}^p\right)^n \norm{x_0}_{L^p(\lvert -h(0), 0), \mathbb R^d)}^p \\
& = \frac{\left(\norm{\varphi}_{L^\infty(\mathbb R, \mathbb R)} \abs{A}^p\right)^{N_t}}{1 - \norm{\varphi}_{L^\infty(\mathbb R, \mathbb R)} \abs{A}^p} \norm{x_0}_{L^p(\lvert -h(0), 0), \mathbb R^d)}^p,
\end{align*}
and, since $\norm{\varphi}_{L^\infty(\mathbb R, \mathbb R)} \abs{A}^p < 1$ and $N_t \geq \frac{t}{\tau_{\max}} - 1$, we obtain the conclusion with $C = \left[\norm{\varphi}_{L^\infty(\mathbb R, \mathbb R)} \abs{A}^p (1 - \norm{\varphi}_{L^\infty(\mathbb R, \mathbb R)} \abs{A}^p)\right]^{-1} \geq 1$ and $\gamma = -\frac{\ln \left(\norm{\varphi}_{L^\infty(\mathbb R, \mathbb R)} \abs{A}^p\right)}{\tau_{\max}} > 0$.
\end{proof}

\section{Applications}
\label{sec:6}

In this section, we describe some applications of the previous results to other contexts, by exploring relations between the difference equation \eqref{eq:diff-eqn-N1-time} and other classes of systems.

\subsection{Difference equation with a single state-de\-pen\-dent delay}\label{sec:6.1}

In this section, we consider the following equation
\begin{equation}\label{equation-SD-1}
x(t) = A x(t - \tau(t, x_t)),
\end{equation}
where $A \in \mathcal M_d(\mathbb R)$ and $\tau(t, x_t)$ is a delay that depends both on the time $t$ and the state $x_t$. Equations with state-dependent delays have a wide range of applications, including in electrodynamics, mechanics, control theory, neural networks, and biology, and their analysis has attracted much research effort in recent years \cite{Hartung2006Functional, Driver1984Neutral, Hernandez2023Explicit, Hartung2021Differentiability, Henriquez2023Existence, Hernandez2020Well}. Even important elementary questions, such as existence and uniqueness of solutions, are highly nontrivial for systems with state-dependent delays, and much work remains to be done for other questions such as stability analysis, with some recent results provided, for instance, in \cite{Xu2021Exponential, BekiarisLiberis2013Compensation, Li2019Algebraic, Li2020Lyapunov, Fiter2013Stability}.

Here, we are interested just in illustrating that some stability results for \eqref{equation-SD-1} can be deduced from those for \eqref{eq:diff-eqn-N1-time} obtained in Section~\ref{sec:stability}. Questions on the well-posedness of \eqref{equation-SD-1}, even though interesting and challenging, and sharper stability results for \eqref{equation-SD-1}, are out of the scope of the present section.

The result we provide here on \eqref{equation-SD-1} is the following one, providing sufficient conditions for the exponential convergence to zero of solutions of \eqref{equation-SD-1}.

\begin{theorem}
\label{thm:state-dep}
Let $A \in \mathcal M_d(\mathbb R)$, $I \subset \mathbb R_+^\ast$ be a compact interval, $\tau_{\min} = \inf I > 0$, $\tau_{\max} = \sup I < +\infty$, and $\tau\colon \mathbb R_+ \times \mathsf L \to I$, where $\mathsf L = C([-\tau_{\max}, 0], \mathbb R^d)$ (respectively, $\mathsf L = G([-\tau_{\max}, 0], \mathbb R^d)$). If $A$ satisfies \ref{hypo:spr-A-less-1}, then there exist positive constants $C$ and $\gamma$ such that, for every continuous (respectively, regulated) solution $x$ of \eqref{equation-SD-1}, denoting by $x_0$ its initial condition, we have
\[
\norm{x_t}_{\mathsf L} \leq C e^{-\gamma t} \norm{x_0}_{\mathsf L} \qquad \text{ for every } t \in \mathbb R_+.
\]
\end{theorem}

\begin{proof}
Let $x$ be a continuous (respectively, regulated) solution of \eqref{equation-SD-1} and consider the function $\tau_\ast\colon \mathbb R_+ \to I$ given by $\tau_\ast(t) = \tau(t, x_t)$. Since $\tau_\ast$ takes values in the compact set $I \subset \mathbb R_+^\ast$, it clearly satisfies assumptions \ref{hypo:tau-infimum} and \ref{hypo:delay-bounded}. Let $h_\ast$ denote the largest delay function associated with $\tau_\ast$. By Theorem~\ref{thm:exp-to-0}, there exist positive constants $C_0$ and $\gamma$ such that
\begin{equation}
\label{eq:estim-exp-SD-proof}
\abs{x(t)} \leq C_0 e^{-\gamma t} \norm{x_0}_{\mathsf L} \qquad \text{ for every } t \in \mathbb R_+.
\end{equation}
In addition, the proof of Theorem~\ref{thm:exp-to-0}, together with the statement of Proposition~\ref{prop:tau-bounded-implies-n-lower-bounded}, shows that the constants $C_0$ and $\gamma$ can be chosen to depend on $\tau_\ast$ only through $\tau_{\max}$, and, in particular, they are independent of the solution $x$. We increase if needed the constant $C_0$ in order to have $C_0 \geq 1$, and hence the inequality in \eqref{eq:estim-exp-SD-proof} trivially holds true for all $t \in [-\tau_{\max}, 0)$ as well. We then have
\[
\norm{x_t}_{\mathsf L} = \sup_{s \in [t - \tau_{\max}, t]} \abs{x(s)} \leq \sup_{s \in [t - \tau_{\max}, t]} C_0 e^{-\gamma s} \norm{x_0}_{\mathsf L} \leq C e^{-\gamma t} \norm{x_0}_{\mathsf L},
\]
with $C = C_0 e^{\gamma \tau_{\max}}$.
\end{proof}

\begin{remark}
Theorem~\ref{thm:state-dep} addresses only the case of continuous or regulated solutions of \eqref{equation-SD-1}, and its strategy relies on fixing first a solution $x$ and then considering the time-dependent delay $\tau_\ast(t) = \tau(t, x_t)$, to which the results of Section~\ref{sec:stability} can be applied. One may follow a similar strategy for $L^p$ solutions of \eqref{equation-SD-1}, but the major difficulty is that the stability results from Section~\ref{sec:stability} concerning $L^p$ solutions all require the time-dependent delay to satisfy \ref{hypo:null}, and it is not straightforward to verify whether $\tau_\ast$ satisfies it or not for a given solution $x$ of \eqref{equation-SD-1}. Even though this could possibly be verified in some simple examples, providing general assumptions on the state-dependent delay $\tau$ ensuring that $\tau_\ast$ satisfies \ref{hypo:null} for every solution $x$ of \eqref{equation-SD-1} is a difficult and interesting question, which is out of the scope of the present paper.
\end{remark}

\subsection{Transport equation}\label{sec:6.2}

Consider the transport equation
\begin{equation}
\label{eq:transport}
\left\{
\begin{aligned}
& \partial_t u(t, x) + \lambda(t, x) \partial_x u(t, x) = 0, & \quad & t > 0, x \in (0, 1), \\
& u(t, 0) = A u(t, 1), & & t \geq 0, \\
& u(0, x) = u_0(x), & & x \in [0, 1],
\end{aligned}
\right.
\end{equation}
where $u\colon \mathbb R_+ \times [0, 1] \to \mathbb R^d$ is the unknown function, $u_0\colon [0, 1] \to \mathbb R^d$ is given, $A \in \mathcal M_d(\mathbb R)$, and $\lambda\colon \mathbb R_+ \times [0, 1] \to \mathbb R_+^\ast$. Let us first recall what is meant by a classical solution of \eqref{eq:transport}.

\begin{definition}
A function $u \in C^0(\mathbb R_+ \times [0, 1], \mathbb R^d) \cap C^1(\mathbb R_+^\ast \times (0, 1), \mathbb R^d)$ is said to be a \emph{classical solution} of \eqref{eq:transport} with \emph{initial condition} $u_0$ if all three equalities of \eqref{eq:transport} are satisfied in their respective domains.
\end{definition}

Note that a necessary condition for the existence of a classical solution is that $u_0$ is continuous with $u_0(0) = A u_0(1)$.

We now want to relate classical solutions of \eqref{eq:transport} with solutions of a certain difference equation of the form \eqref{eq:diff-eqn-N1-time}. For that purpose, we consider for simplicity the following assumption on $\lambda$.

\begin{hypothesis}\mbox{}
\begin{hypolist}
\item\label{hypo:lambda} There exist positive constants $\lambda_{\min}$, $\lambda_{\max}$, and $L$, with $\lambda_{\min} \leq \lambda_{\max}$, such that $\lambda \in C^0(\mathbb R^2, [\lambda_{\min}, \lambda_{\max}])$, $\lambda$ is differentiable with respect to its second variable $x$, and $\partial_x \lambda \in C^0(\mathbb R^2, [-L, L])$.
\end{hypolist}
\end{hypothesis}

Recall that, thanks to the classical theory of ordinary differential equations, under assumption \ref{hypo:lambda}, the flow of the differential equation $x'(t) = \lambda(t, x(t))$ is a well-defined $C^1$ function. More precisely, we have the following result.

\begin{lemma}
\label{lem:flow}
Let $\lambda$ be a function satisfying \ref{hypo:lambda}. Then there exists a unique function $\Phi \in C^1(\mathbb R^3, \mathbb R)$ satisfying the following properties.
\begin{enumerate}
\item\label{item:diff-eqn-flow} $\partial_t \Phi(t, t_0, x_0) = \lambda(t, \Phi(t, t_0, x_0))$ for every $(t, t_0, x_0) \in \mathbb R^3$.
\item $\Phi(t_0, t_0, x_0) = x_0$ for every $(t_0, x_0) \in \mathbb R^2$.
\end{enumerate}
In addition, $\Phi$ also satisfies the following properties.
\begin{enumerate}[resume]
\item\label{item:semigroup} $\Phi(t_2, t_1, \Phi(t_1, t_0, x_0)) = \Phi(t_2, t_0, x_0)$ for every $(t_0, t_1, t_2, x_0) \in \mathbb R^4$.
\item\label{item:flow-reverse} For every $(t_0, t_1, x_0, x_1) \in \mathbb R^4$, we have $\Phi(t_1, t_0, x_0) = x_1$ if and only if $\Phi(t_0, t_1,\allowbreak x_1) = x_0$.
\item\label{item:deriv-x0} $\partial_{x_0} \Phi(t, t_0, x_0) = \exp\left(\int_{t_0}^t \partial_x \lambda(s, \Phi(s, t_0, x_0)) \diff s\right)$ for every $(t, t_0, x_0) \in \mathbb R^3$.
\item\label{item:deriv-t0} $\partial_{t_0} \Phi(t, t_0, x_0) = -\lambda(t_0, x_0) \partial_{x_0} \Phi(t, t_0, x_0)$ for every $(t, t_0, x_0) \in \mathbb R^3$.
\item\label{item:flow-increasing} For every $(t_0, x_0) \in \mathbb R^2$, the function $t \mapsto \Phi(t, t_0, x_0)$ is increasing and, for every $t \in \mathbb R$,
\[
\lambda_{\min} \abs{t - t_0} \leq \abs{\Phi(t, t_0, x_0) - x_0} \leq \lambda_{\max} \abs{t - t_0}.
\]
\end{enumerate}
\end{lemma}

The function $\Phi$ from the statement of Lemma~\ref{lem:flow} is referred to in the sequel as the \emph{flow} of $\lambda$. The proof of Lemma~\ref{lem:flow} can be obtained as an immediate consequence of classical results for ordinary differential equations, such as those from \cite[Chapter~1]{Hale1969Ordinary}, and is thus omitted here.

In order to relate solutions of \eqref{eq:transport} with solutions of a difference equation of the form \eqref{eq:diff-eqn-N1-time}, we need the following technical result.

\begin{lemma}
\label{lem:lambda}
Let $\lambda$ satisfy \ref{hypo:lambda} and $\Phi$ be the flow of $\lambda$. Then the following assertions hold true.
\begin{enumerate}
\item\label{item:T0} There exists a unique $T_0 > 0$ such that $\Phi(-T_0, 0, 1) = 0$.

\item\label{item:X0} The function $X_0 \in C^1(\mathbb R^2, \mathbb R)$ defined for $(t, x) \in \mathbb R^2$ by $X_0(t, x) = \Phi(0, t, x)$ satisfies the following properties.
\begin{enumerate}[label={\textnormal{(\roman*)}}, ref={\theenumi-(\roman*)}, leftmargin=*]
\item\label{item:X0-decr} $\partial_t X_0(t, x) < 0$ for every $(t, x) \in \mathbb R^2$.
\item For every $t \in [-T_0, 0]$, we have $X_0(t, 0) \in [0, 1]$.
\end{enumerate}

\item\label{item:R} There exists a function $R \in C^1(\mathbb R^2, \mathbb R)$ such that $\Phi(R(t, x), t, x) = 0$ for every $(t, x) \in \mathbb R^2$ and, in addition, $R$ satisfies the following properties.
\begin{enumerate}[label={\textnormal{(\roman*)}}, ref={\theenumi-(\roman*)}, leftmargin=*]
\item\label{item:pde-R} For every $(t, x) \in \mathbb R^2$, we have $\partial_t R(t, x) + \lambda(t, x) \partial_x R(t, x) = 0$.
\item\label{item:R-incr-t} For every $x \in \mathbb R$, $t \mapsto R(t, x)$ is increasing.
\item\label{item:R-decr-x} For every $t \in \mathbb R$, $x \mapsto R(t, x)$ is decreasing, with $R(t, 0) = t$. In particular, $R(t, x) < t$ for every $x > 0$.
\item\label{item:R-bounds} For every $(t, x) \in \mathbb R^2$, we have
\[
\frac{\abs{x}}{\lambda_{\max}} \leq \abs{t - R(t, x)} \leq \frac{\abs{x}}{\lambda_{\min}}.
\]
\item\label{item:range-R} For every $(t, x) \in \mathbb R_+ \times [0, 1]$, we have $R(t, x) \in [-T_0, +\infty)$.
\item\label{item:range-R0} For every $x \in [0, 1]$, we have $R(0, x) \in [-T_0, 0]$.
\item\label{item:R-X0} For every $t \in \mathbb R_+$, if $R(t, 1) < 0$, then $X_0(t, 1) \in (0, 1]$.
\item\label{item:X0_1_0} For every $t \in \mathbb R$, we have $X_0(t, 1) = X_0(R(t, 1), 0)$.
\end{enumerate}

\item\label{item:tau} The function $\tau\colon \mathbb R \to \mathbb R$ defined for $t \in \mathbb R$ by $\tau(t) = t - R(t, 1)$ satisfies the following properties.
\begin{enumerate}[label={\textnormal{(\roman*)}}, ref={\theenumi-(\roman*)}, leftmargin=*]
\item\label{item:tau-bounds} $\tau \in C^1\left(\mathbb R, \left[\frac{1}{\lambda_{\max}}, \frac{1}{\lambda_{\min}}\right]\right)$
\item\label{item:tau-prime} There exists a constant $\alpha \in (0, 1)$ such that $\tau'(t) \leq \alpha$ for every $t \in \mathbb R$.
\item\label{item:delay-incr} The function $t \mapsto t - \tau(t)$ is increasing.
\item\label{item:delay-lower} We have $t - \tau(t) \geq -T_0$ for every $t \geq 0$.
\item\label{item:diffeo} For every $t \geq 0$, $R(t, \cdot)$ is a diffeomorphism between $(0, 1)$ and $(t - \tau(t), t)$, with inverse given by $s \mapsto \Phi(t, s, 0)$. In addition, there exist positive constants $\beta_0, \beta_1$ such that, for every $(t, x) \in \mathbb R_+ \times (0, 1)$, we have $\beta_0 \leq \abs{\partial_x R(t, x)} \leq \beta_1$.
\end{enumerate}
\end{enumerate}
\end{lemma}

\begin{proof}
Item~\ref{item:T0} is an immediate consequence of Lemma~\ref{lem:flow}\ref{item:flow-increasing}, the intermediate value theorem, and the facts that $\Phi(0, 0, 1) = 1$ and $\lim_{t \to -\infty}\Phi(t, 0, 1) = -\infty$.

Note that, by Lemma~\ref{lem:flow}\ref{item:deriv-x0} and \ref{item:deriv-t0}, for every $(t, x) \in \mathbb R^2$, we have $\partial_t X_0(t, x) = -\lambda(t, x) \partial_{x_0} \Phi(0, t, x) < 0$. This also shows that $t \mapsto X_0(t, x)$ is decreasing for every $x \in \mathbb R$. Since $X_0(0, 0) = 0$ and $X_0(-T_0, 0) = \Phi(0, -T_0, 0) = 1$ by Lemma~\ref{lem:flow}\ref{item:flow-reverse}, we deduce that $X_0(t, 0) \in [0, 1]$ for every $t \in [-T_0, 0]$, concluding the proof of \ref{item:X0}.

By Lemma~\ref{lem:flow}\ref{item:flow-increasing}, for every $(t, x) \in \mathbb R^2$, the function $s \mapsto \Phi(s, t, x)$ is increasing with $\lim_{s \to -\infty} \Phi(s, t, x) = -\infty$ and $\lim_{s \to +\infty} \Phi(s, t, x) = +\infty$, and thus there exists a unique $R(t, x) \in \mathbb R$ such that $\Phi(R(t, x), t, x) = 0$. The fact that $(t, x) \mapsto R(t, x)$ defines a $C^1$ function follows as an immediate consequence of the implicit function theorem and the fact that $\partial_t \Phi(t, t_0, x_0) > 0$ for every $(t, t_0, x_0) \in \mathbb R^3$. In addition, the implicit function theorem, combined with Lemma~\ref{lem:flow}\ref{item:diff-eqn-flow}, \ref{item:deriv-x0}, and \ref{item:deriv-t0} also yields that, for every $(t, x) \in \mathbb R^2$,
\begin{align}
\partial_t R(t, x) & = \frac{\lambda(t, x)}{\lambda(R(t, x), 0)} \exp\left(\int_{t}^{R(t, x)} \partial_x \lambda(s, \Phi(s, t, x)) \diff s\right), \label{eq:diff-t-R}\\
\partial_x R(t, x) & = -\frac{1}{\lambda(R(t, x), 0)} \exp\left(\int_{t}^{R(t, x)} \partial_x \lambda(s, \Phi(s, t, x)) \diff s\right), \label{eq:diff-x-R}
\end{align}
which implies \ref{item:pde-R}, \ref{item:R-incr-t}, and that $x \mapsto R(t, x)$ is decreasing. Since $\Phi(R(t, 0),\allowbreak t, 0) = 0$, we deduce that $R(t, 0) = t$, and thus \ref{item:R-decr-x} also holds true, while item \ref{item:R-bounds} follows easily from Lemma~\ref{lem:flow}\ref{item:flow-increasing}.

Combining \ref{item:R-incr-t} and \ref{item:R-decr-x}, we deduce that $\min_{(t, x) \in \mathbb R_+ \times [0, 1]} R(t, x) = R(0, 1) \allowbreak = - T_0$, yielding \ref{item:range-R}. Moreover, \ref{item:range-R0} follows immediately from the fact that $R(0, 0) = 0$ and $R(0, 1) = -T_0$.

To prove \ref{item:R-X0}, take $t \in \mathbb R_+$ such that $R(t, 1) < 0$. Since $t \geq 0$ and $s \mapsto X_0(s, 1)$ is decreasing by \ref{item:X0-decr}, we have that $X_0(t, 1) \leq X_0(0, 1) = 1$. On the other hand, using Lemma~\ref{lem:flow}\ref{item:semigroup}, we have $X_0(t, 1) = \Phi(0, R(t, 1), \Phi(R(t, 1), t, 1)) = \Phi(0, R(t, 1), 0)$ and, by Lemma~\ref{lem:flow}\ref{item:deriv-t0}, we have that $s \mapsto \Phi(0, s, 0)$ is decreasing. Thus, since $R(t, 1) < 0$, we have $X_0(t, 1) = \Phi(0, R(t, 1), 0) > \Phi(0, 0, 0) = 0$, yielding the conclusion of \ref{item:R-X0}.

To conclude the proof of \ref{item:R}, we use Lemma~\ref{lem:flow}\ref{item:semigroup} to deduce that, for every $t \in \mathbb R$, we have $X_0(t, 1) = \Phi(0, t, 1) = \Phi(0, R(t, 1), \Phi(R(t, 1), t, 1)) = \Phi(0, R(t, 1), 0) = X_0(R(t, 1), 0)$, yielding \ref{item:X0_1_0}.

Concerning the function $\tau$, note that \ref{item:tau-bounds} is an immediate consequence of \ref{item:R-decr-x} and \ref{item:R-bounds}. In addition, combining \ref{item:R-bounds} and \eqref{eq:diff-t-R}, we have
\[
\tau'(t) = 1 - \partial_t R(t, 1) \leq 1 - \frac{\lambda_{\min}}{\lambda_{\max}} e^{-\frac{L}{\lambda_{\min}}} \in (0, 1),
\]
yielding the conclusion of \ref{item:tau-prime}. Item~\ref{item:delay-incr} is an immediate consequence of \ref{item:R-incr-t}, while \ref{item:delay-lower} is a particular case of \ref{item:range-R}. Finally, \ref{item:diffeo} follows from \eqref{eq:diff-x-R} and the facts that $R(t, 0) = t$ (from \ref{item:R-decr-x}) and $R(t, 1) = t - \tau(t)$, with $\beta_0 = \frac{1}{\lambda_{\max}} e^{-\frac{L}{\lambda_{\min}}}$ and $\beta_1 = \frac{1}{\lambda_{\min}} e^{\frac{L}{\lambda_{\min}}}$.
\end{proof}

Using Lemma~\ref{lem:lambda}, we deduce the following equivalence between classical solutions of \eqref{eq:transport} and solutions of a difference equation of the form \eqref{eq:diff-eqn-N1-time}.

\begin{proposition}
\label{prop:transport}
Let $A \in \mathcal M_d(\mathbb R)$, $\lambda$ satisfy \ref{hypo:lambda}, and $T_0$, $X_0$, $R$, and $\tau$ be as in the statement of Lemma~\ref{lem:lambda}. If $u \in C^0(\mathbb R_+ \times [0, 1], \mathbb R^d) \cap C^1(\mathbb R_+^\ast \times (0, 1), \mathbb R^d)$ is a classical solution of \eqref{eq:transport} with initial condition $u_0$, then the function $v \in C^0([-T_0, +\infty), \mathbb R^d)$ defined for $t \in [-T_0, +\infty)$ by
\begin{equation}
\label{eq:def-v}
v(t) = \begin{dcases*}
u(t, 0), & if $t \geq 0$, \\
u_0(X_0(t, 0)), & if $t \in [-T_0, 0)$
\end{dcases*}
\end{equation}
is a continuous solution of
\begin{equation}
\label{eq:diff-transport}
\left\{
\begin{aligned}
v(t) & = A v(t - \tau(t)), & \qquad & t \geq 0, \\
v(t) & = v_0(t), & & t \in [-T_0, 0),
\end{aligned}
\right.
\end{equation}
with $v_0(t) = u_0(X_0(t, 0))$ for $t \in [-T_0, 0)$, and in addition $v \in C^1((-T_0, +\infty), \mathbb R^d)$. Conversely, if $v \in C^0([-T_0, +\infty), \mathbb R^d) \cap C^1((-T_0, +\infty), \mathbb R^d)$ is a continuous solution of \eqref{eq:diff-transport}, then the function $u\colon \mathbb R_+ \times [0, 1] \to \mathbb R^d$ defined by $u(t, x) = v(R(t, x))$ is a classical solution of \eqref{eq:transport} with initial condition $u_0(x) = v_0(R(0, x))$.

In addition, for every $t \geq 0$, we have
\begin{equation}
\label{eq:transp-norm-infty}
\norm{u(t, \cdot)}_{L^\infty([0, 1], \mathbb R^d)} = \norm{v_t}_{L^\infty([-\tau(t), 0], \mathbb R^d)}
\end{equation}
and, for every $t \geq 0$ and $p \in [1, +\infty)$, we have
\begin{equation}
\label{eq:transp-norm-p}
\beta_1^{-1/p} \norm{v_t}_{L^p([-\tau(t), 0], \mathbb R^d)} \leq \norm{u(t, \cdot)}_{L^p([0, 1], \mathbb R^d)} \leq \beta_0^{-1/p} \norm{v_t}_{L^p([-\tau(t), 0], \mathbb R^d)},
\end{equation}
where $\beta_0$ and $\beta_1$ are the constants from Lemma~\ref{lem:lambda}\ref{item:diffeo}.
\end{proposition}

\begin{proof}
Let $u \in C^0(\mathbb R_+ \times [0, 1], \mathbb R^d) \cap C^1(\mathbb R_+^\ast \times (0, 1), \mathbb R^d)$ be a classical solution of \eqref{eq:transport} with initial condition $u_0$ and $v$ be defined from $u$ and $u_0$ as in \eqref{eq:def-v}. Clearly, $v$ is continuous, since $v(0) = u(0, 0) = u_0(0)$ and $v(0^-) = u_0(X_0(0, 0)) = u_0(0)$. By definition, we have $v(t) = v_0(t)$ for every $t \in [-T_0, 0)$. A straightforward computation using Lemma~\ref{lem:flow}\ref{item:diff-eqn-flow} shows that, for every $(t, t_0, x_0) \in \mathbb R^3$ such that $(t, \Phi(t, t_0, x_0)) \in \mathbb R_+^\ast \times (0, 1)$, we have
\begin{equation}
\label{eq:characteristics}
\frac{\diff}{\diff t} u(t, \Phi(t, t_0, x_0)) = 0.
\end{equation}
In particular, for every $t \geq 0$, if $R(t, 1) \geq 0$, then $(s, \Phi(s, R(t, 1), 0)) \in \mathbb R_+^\ast \times (0, 1)$ for every $s \in (R(t, 1), t)$, and thus $u(t, 1) = u(t, \Phi(t, R(t, 1), 0)) = u(R(t, 1), 0) = u(t - \tau(t), 0)$. Hence, $v(t) = u(t, 0) = A u(t, 1) = A v(t - \tau(t))$, establishing the first case of \eqref{eq:diff-transport} for every $t \geq 0$ such that $R(t, 1) \geq 0$.

Now, let $t \geq 0$ be such that $R(t, 1) < 0$, which implies, by Lemma~\ref{lem:lambda}\ref{item:R-X0}, that $X_0(t, 1) \in (0, 1]$. In addition, for every $s \in (0, t)$, we have $(s, \Phi(s, 0, X_0(t, 1))) \in \mathbb R_+^\ast \times (0, 1)$ since $\Phi(t, 0, X_0(t, 1)) = 1$. Hence, using also Lemma~\ref{lem:lambda}\ref{item:X0_1_0}, we have $u(t, 1) = u(t, \Phi(t, 0, X_0(t, 1))) = u(0, X_0(t, 1)) = u_0(X_0(t, 1)) = u_0(X_0(R(t, 1), 0)) = v_0(t - \tau(t))$, showing that $u(t, 0) = A u(t, 1) = A v(t - \tau(t))$, and concluding thus the proof that $v$ satisfies \eqref{eq:diff-transport}.

To show that $v$ belongs to $C^1((-T_0, +\infty), \mathbb R^d)$, let $\xi\colon\mathbb R_+ \to (0, 1]$ be the function given for $t \in \mathbb R_+^\ast$ by $\xi(t) = \frac{1}{t + 1}$. We claim that, for every $t \in (-T_0, +\infty)$, there exists a unique $S(t) \in \mathbb R_+^\ast$ such that $\Phi(S(t), t, 0) = \xi(S(t))$. Indeed, on the one hand, we have $\Phi(0, -T_0, 0) = 1$ and, since by Lemma~\ref{lem:flow}\ref{item:deriv-t0} the function $s \mapsto \Phi(0, s, 0)$ is decreasing, we deduce that $\Phi(0, t, 0) < 1$ for every $t \in (-T_0, +\infty)$. Thus, $\Phi(0, t, 0) - \xi(0) < 0$ for every $t \in (-T_0, +\infty)$. On the other hand, $\lim_{s \to +\infty} \Phi(s, t, 0) - \xi(s) = +\infty$, and thus the existence of $S(t)$ follows from the intermediate value theorem. In addition, for every $s \in \mathbb R_+^\ast$, we have
\begin{equation}
\label{eq:deriv-St}
\frac{\diff}{\diff s} \left[\Phi(s, t, 0) - \xi(s)\right] = \lambda(s, \Phi(s, t, 0)) + \frac{1}{(t+1)^2} > 0,
\end{equation}
thus $s \mapsto \Phi(s, t, 0) - \xi(s)$ is increasing in $\mathbb R_+^\ast$, yielding the uniqueness of $S(t)$. We also obtain from \eqref{eq:deriv-St} and the implicit function theorem that $S \in C^1((-T_0, +\infty), \mathbb R_+^\ast)$.

For $t \geq 0$, using \eqref{eq:characteristics}, we have that $v(t) = u(t, 0) = u(S(t), \Phi(S(t), t, 0)) = u(S(t), \xi(S(t)))$, since $(s, \Phi(s, t, 0)) \in \mathbb R_+^\ast \times (0, 1)$ for every $s \in (t, S(t))$. For $t \in (-T_0, 0)$, still using \eqref{eq:characteristics}, we have that
\begin{align*}
v(t) & = u(0, X_0(t, 0)) = u(S(t), \Phi(S(t), 0, X_0(t, 0))) \\
& = u(S(t), \Phi(S(t), t, 0)) = u(S(t), \xi(S(t)))
\end{align*}
since $(s, \Phi(s, 0, X_0(t, 0))) \in \mathbb R_+^\ast \times (0, 1)$ for every $s \in (0, S(t))$. Hence, we have $v(t) = u(S(t), \xi(S(t)))$ for every $t \in (-T_0, +\infty)$ and, since $S \in C^1((-T_0, +\infty), \mathbb R_+^\ast)$, $\xi \in C^1(\mathbb R_+^\ast, (0, 1))$, and $u \in C^1(\mathbb R_+^\ast \times (0, 1), \mathbb R^d)$, we deduce that $v \in C^1((-T_0, +\infty), \mathbb R^d)$.

Conversely, let $v \in C^0([-T_0, +\infty), \mathbb R^d) \cap C^1((-T_0, +\infty), \mathbb R^d)$ be a continuous solution of \eqref{eq:diff-transport} with initial condition $v_0$ and define $u\colon \mathbb R_+ \times [0, 1] \to \mathbb R^d$ by $u(t, x) = v(R(t, x))$ and $u_0\colon [0, 1] \to \mathbb R^d$ by $u_0(x) = v_0(R(0, x))$. Since $R \in C^1(\mathbb R^2, \mathbb R)$, we have $u \in C^0(\mathbb R_+ \times [0, 1], \mathbb R^d) \cap C^1(\mathbb R_+^\ast \times (0, 1), \mathbb R^d)$ and, using Lemma~\ref{lem:lambda}\ref{item:pde-R}, it is immediate to verify that $\partial_t u(t, x) + \lambda(t, x) \partial_x u(t, x) = 0$ for every $(t, x) \in \mathbb R_+^\ast \times (0, 1)$. For $x \in [0, 1]$, we have $u(0, x) = v(R(0, x)) = u_0(x)$ and, for $t \geq 0$, we have $u(t, 0) = v(R(t, 0)) = v(t) = A v(R(t, 1)) = u(t, 1)$. Hence, $u$ is a classical solution of \eqref{eq:transport} with initial condition $u_0$.

Let us finally show \eqref{eq:transp-norm-infty} and \eqref{eq:transp-norm-p}. Let $p \in [1, +\infty)$, $t \geq 0$, and $x \in (0, 1)$. We have $u(t, x) = v(R(t, x))$, and thus, using Lemma~\ref{lem:lambda}\ref{item:diffeo}, we deduce that
\[
\norm{u(t, \cdot)}_{L^\infty([0, 1], \mathbb R^d)} = \sup_{x \in (0, 1)} \abs{v(R(t, x))} = \sup_{s \in (t - \tau(t), t)} \abs{v(s)} = \norm{v_t}_{L^\infty([t - \tau(t), t], \mathbb R^d)},
\]
yielding \eqref{eq:transp-norm-infty}, and
\[
\norm{u(t, \cdot)}_{L^p([0, 1], \mathbb R^d)}^p = \int_0^1 \abs{v(R(t, x))}^p \diff x = \int_{t - \tau(t)}^{t} \frac{\abs{v(s)}^p}{\abs{\partial_x R(t, \raisebox{\depth}{\scalebox{1}[-1]{R}}(t, s))}} \diff s,
\]
where $\raisebox{\depth}{\scalebox{1}[-1]{R}}(t, \cdot)$ is the inverse of $R(t, \cdot)$. Hence \eqref{eq:transp-norm-p} follows from Lemma~\ref{lem:lambda}\ref{item:diffeo}.
\end{proof}

Using Proposition~\ref{prop:transport} and the results from Section~\ref{sec:exponential}, we obtain at once the following result.

\begin{corollary}
\label{coro:appli-transport}
Let $A \in \mathcal M_d(\mathbb R)$ satisfy \ref{hypo:spr-A-less-1} and $\lambda$ satisfy \ref{hypo:lambda}. Then there exist positive constants $C$ and $\gamma$ such that, for every classical solution $u \in C^0(\mathbb R_+ \times [0, 1], \mathbb R^d) \cap C^1(\mathbb R_+^\ast \times (0, 1), \mathbb R^d)$ of \eqref{eq:transport} with initial condition $u_0$, we have
\begin{equation}
\label{eq:transport-Linfty}
\norm{u(t, \cdot)}_{L^\infty([0, 1], \mathbb R^d)} \leq C e^{-\gamma t} \norm{u_0}_{L^\infty([0, 1], \mathbb R^d)} \qquad \text{ for every } t \in \mathbb R_+.
\end{equation}

Moreover, let $\alpha$ be the constant from Lemma~\ref{lem:lambda}\ref{item:tau-prime}, $p \in [1, +\infty)$, and assume that $\rho(A)^p < 1-\alpha$. Then there exist positive constants $C$ and $\gamma$ such that, for every classical solution $u \in C^0(\mathbb R_+ \times [0, 1], \mathbb R^d) \cap C^1(\mathbb R_+^\ast \times (0, 1), \mathbb R^d)$ of \eqref{eq:transport} with initial condition $u_0$, we have
\begin{equation}
\label{eq:transport-Lp}
\norm{u(t, \cdot)}_{L^p([0, 1], \mathbb R^d)} \leq C e^{-\gamma t} \norm{u_0}_{L^p([0, 1], \mathbb R^d)} \qquad \text{ for every } t \in \mathbb R_+.
\end{equation}
\end{corollary}

\begin{proof}
Let $\tau$ be defined from $\lambda$ as in Lemma~\ref{lem:lambda}\ref{item:tau} and note that, by item \ref{item:tau-bounds} of that lemma, $\tau$ satisfies \ref{hypo:tau-infimum} and \ref{hypo:delay-bounded}. In addition, Lemma~\ref{lem:lambda}\ref{item:delay-incr} implies, from Definition~\ref{def:h}, that the largest delay function $h$ associated with $\tau$ is $h = \tau$. Moreover, combining Lemma~\ref{lem:lambda}\ref{item:tau-prime} and Remark~\ref{remk:radon-nikodym-derivative-tau}, we deduce that $\tau$ satisfies \ref{hypo:radon-nik-bdd}. Then \eqref{eq:transport-Linfty} is a consequence of Corollary~\ref{coro:exp-Linfty}.

We also note that, combining once again Lemma~\ref{lem:lambda}\ref{item:tau-prime} and Remark~\ref{remk:radon-nikodym-derivative-tau}, the fact that $\rho(A)^p < 1-\alpha$ implies that \ref{hypo:radon-nik-and-rho} is satisfied, and thus \eqref{eq:transport-Lp} is a consequence of Theorem~\ref{thm:exp-Lp}.
\end{proof}

\bibliographystyle{abbrv}
\bibliography{Bib}

\begin{thebibliography}{10}

\bibitem{Avellar1980Zeros}
C.~E. Avellar and J.~K. Hale.
\newblock On the zeros of exponential polynomials.
\newblock {\em J. Math. Anal. Appl.}, 73(2):434--452, 1980.

\bibitem{Bastin2016Stability}
G.~Bastin and J.-M. Coron.
\newblock {\em Stability and boundary stabilization of 1-{D} hyperbolic
  systems}, volume~88 of {\em Progress in Nonlinear Differential Equations and
  their Applications}.
\newblock Birkh\"{a}user/Springer, [Cham], 2016.
\newblock Subseries in Control.

\bibitem{BekiarisLiberis2013Compensation}
N.~Bekiaris-Liberis and M.~Krstic.
\newblock Compensation of state-dependent input delay for nonlinear systems.
\newblock {\em IEEE Trans. Automat. Control}, 58(2):275--289, 2013.

\bibitem{Bonnet2020L2}
C.~Bonnet and J.~R. Partington.
\newblock {$L_2$} and {BIBO} stability of systems with variable delays.
\newblock {\em Systems Control Lett.}, 139:104671, 7, 2020.

\bibitem{Brayton1967Nonlinear}
R.~K. Brayton.
\newblock Nonlinear oscillations in a distributed network.
\newblock {\em Quart. Appl. Math.}, 24(4):289--301, 1967.

\bibitem{Chitour2023Hautus}
Y.~Chitour, S.~Fueyo, G.~Mazanti, and M.~Sigalotti.
\newblock {H}autus--{Y}amamoto criteria for approximate and exact
  controllability of linear difference delay equations.
\newblock {\em Discrete Contin. Dyn. Syst.}, 43(9):3306--3337, 2023.

\bibitem{Chitour2016Stability}
Y.~Chitour, G.~Mazanti, and M.~Sigalotti.
\newblock Stability of non-autonomous difference equations with applications to
  transport and wave propagation on networks.
\newblock {\em Netw. Heterog. Media}, 11(4):563--601, 2016.

\bibitem{Chitour2020Approximate}
Y.~Chitour, G.~Mazanti, and M.~Sigalotti.
\newblock Approximate and exact controllability of linear difference equations.
\newblock {\em J. \'{E}c. polytech. Math.}, 7:93--142, 2020.

\bibitem{Cooke1968Differential}
K.~L. Cooke and D.~W. Krumme.
\newblock Differential-difference equations and nonlinear initial-boundary
  value problems for linear hyperbolic partial differential equations.
\newblock {\em J. Math. Anal. Appl.}, 24:372--387, 1968.

\bibitem{Coron2015Dissipative}
J.-M. Coron and H.-M. Nguyen.
\newblock Dissipative boundary conditions for nonlinear 1-{D} hyperbolic
  systems: sharp conditions through an approach via time-delay systems.
\newblock {\em SIAM J. Math. Anal.}, 47(3):2220--2240, 2015.

\bibitem{Cruz1970Stability}
M.~A. Cruz and J.~K. Hale.
\newblock Stability of functional differential equations of neutral type.
\newblock {\em J. Differential Equations}, 7:334--355, 1970.

\bibitem{Avellar1990Difference}
C.~E. de~Avellar and S.~A.~S. Marconato.
\newblock Difference equations with delays depending on time.
\newblock {\em Bol. Soc. Brasil. Mat. (N. S.)}, 21(1):51--58, 1990.

\bibitem{Driver1984Neutral}
R.~D. Driver.
\newblock A neutral system with state-dependent delay.
\newblock {\em J. Differential Equations}, 54(1):73--86, 1984.

\bibitem{Fiter2013Stability}
C.~Fiter and E.~Fridman.
\newblock Stability of piecewise affine systems with state-dependent delay, and
  application to congestion control.
\newblock In {\em 52nd IEEE Conference on Decision and Control}, pages
  1572--1577, 2013.

\bibitem{Fridman2008Robust}
E.~Fridman.
\newblock On robust stability of linear neutral systems with time-varying
  delays.
\newblock {\em IMA J. Math. Control Inform.}, 25(4):393--407, 2008.

\bibitem{Fridman2007Stability}
E.~Fridman and M.~Gil'.
\newblock Stability of linear systems with time-varying delays: a direct
  frequency domain approach.
\newblock {\em J. Comput. Appl. Math.}, 200(1):61--66, 2007.

\bibitem{Hale1969Ordinary}
J.~K. Hale.
\newblock {\em Ordinary differential equations}.
\newblock Pure and Applied Mathematics, Vol. XXI. Wiley-Interscience [John
  Wiley \& Sons], New York-London-Sydney, 1969.

\bibitem{Hale1993Introduction}
J.~K. Hale and S.~M. Verduyn~Lunel.
\newblock {\em Introduction to functional-differential equations}, volume~99 of
  {\em Applied Mathematical Sciences}.
\newblock Springer-Verlag, New York, 1993.

\bibitem{Hale2002Strong}
J.~K. Hale and S.~M. Verduyn~Lunel.
\newblock Strong stabilization of neutral functional differential equations.
\newblock {\em IMA J. Math. Control Inform.}, 19(1-2):5--23, 2002.
\newblock Special issue on analysis and design of delay and propagation
  systems.

\bibitem{Hale2003Stability}
J.~K. Hale and S.~M. Verduyn~Lunel.
\newblock Stability and control of feedback systems with time delays.
\newblock {\em Internat. J. Systems Sci.}, 34(8-9):497--504, 2003.
\newblock Time delay systems: theory and control.

\bibitem{Hartung2021Differentiability}
F.~Hartung.
\newblock Differentiability of solutions with respect to parameters in a class
  of neutral differential equations with state-dependent delays.
\newblock {\em Electron. J. Qual. Theory Differ. Equ.}, pages Paper No. 56, 41,
  2021.

\bibitem{Hartung2006Functional}
F.~Hartung, T.~Krisztin, H.-O. Walther, and J.~Wu.
\newblock Functional differential equations with state-dependent delays: theory
  and applications.
\newblock In {\em Handbook of differential equations: ordinary differential
  equations. {V}ol. {III}}, Handb. Differ. Equ., pages 435--545.
  Elsevier/North-Holland, Amsterdam, 2006.

\bibitem{Henriquez2023Existence}
H.~R. Henr\'{\i}quez, J.~G. Mesquita, and H.~C. dos Reis.
\newblock Existence results for abstract functional differential equations with
  infinite state-dependent delay and applications.
\newblock {\em Math. Ann.}, 2023.

\bibitem{Henry1974Linear}
D.~Henry.
\newblock Linear autonomous neutral functional differential equations.
\newblock {\em J. Differential Equations}, 15:106--128, 1974.

\bibitem{Hernandez2023Explicit}
E.~Hern\'{a}ndez.
\newblock On explicit abstract neutral differential equations with
  state-dependent delay.
\newblock {\em Proc. Amer. Math. Soc.}, 151(3):1119--1133, 2023.

\bibitem{Hernandez2020Well}
E.~Hern\'{a}ndez, D.~Fernandes, and J.~Wu.
\newblock Well-posedness of abstract integro-differential equations with
  state-dependent delay.
\newblock {\em Proc. Amer. Math. Soc.}, 148(4):1595--1609, 2020.

\bibitem{Horn2013Matrix}
R.~A. Horn and C.~R. Johnson.
\newblock {\em Matrix analysis}.
\newblock Cambridge University Press, Cambridge, second edition, 2013.

\bibitem{Li2019Algebraic}
H.~Li, Y.~Zheng, and F.~E. Alsaadi.
\newblock Algebraic formulation and topological structure of {B}oolean networks
  with state-dependent delay.
\newblock {\em J. Comput. Appl. Math.}, 350:87--97, 2019.

\bibitem{Li2020Lyapunov}
X.~Li and X.~Yang.
\newblock Lyapunov stability analysis for nonlinear systems with
  state-dependent state delay.
\newblock {\em Automatica J. IFAC}, 112:108674, 6, 2020.

\bibitem{Louisell2001Delay}
J.~Louisell.
\newblock Delay differential systems with time-varying delay: new directions
  for stability theory.
\newblock {\em Kybernetika (Prague)}, 37(3):239--251, 2001.
\newblock Special issue on advances in analysis and control of time-delay
  systems.

\bibitem{Mazanti2014Stabilization}
G.~Mazanti.
\newblock Stabilization of persistently excited linear systems by delayed
  feedback laws.
\newblock {\em Systems Control Lett.}, 68:57--67, 2014.

\bibitem{Mazanti2016Stability}
G.~Mazanti.
\newblock {\em Stabilit\'{e} et stabilisation de syst\`{e}mes lin\'{e}aires
  \`{a} commutation en dimensions finie et infinie}.
\newblock PhD thesis, \'{E}cole Polytechnique, Universit\'{e} Paris--Saclay,
  Palaiseau, France, September 2016.

\bibitem{Mazanti2017Relative}
G.~Mazanti.
\newblock Relative controllability of linear difference equations.
\newblock {\em SIAM J. Control Optim.}, 55(5):3132--3153, 2017.

\bibitem{Melvin1974Stability}
W.~R. Melvin.
\newblock Stability properties of functional difference equations.
\newblock {\em J. Math. Anal. Appl.}, 48:749--763, 1974.

\bibitem{Michiels2014Stability}
W.~Michiels and S.-I. Niculescu.
\newblock {\em Stability, control, and computation for time-delay systems. An
  eigenvalue-based approach}, volume~27 of {\em Advances in Design and
  Control}.
\newblock Society for Industrial and Applied Mathematics (SIAM), Philadelphia,
  PA, second edition, 2014.

\bibitem{Michiels2005Stabilization}
W.~Michiels, V.~Van~Assche, and S.-I. Niculescu.
\newblock Stabilization of time-delay systems with a controlled time-varying
  delay and applications.
\newblock {\em IEEE Trans. Automat. Control}, 50(4):493--504, 2005.

\bibitem{Michiels2009Strong}
W.~Michiels, T.~Vyhl{\'{\i}}dal, P.~Z{\'{\i}}tek, H.~Nijmeijer, and D.~Henrion.
\newblock Strong stability of neutral equations with an arbitrary delay
  dependency structure.
\newblock {\em SIAM J. Control Optim.}, 48(2):763--786, 2009.

\bibitem{Shustin2007Delay}
E.~Shustin and E.~Fridman.
\newblock On delay-derivative-dependent stability of systems with fast-varying
  delays.
\newblock {\em Automatica J. IFAC}, 43(9):1649--1655, 2007.

\bibitem{Silkowski1976Star}
R.~A. Silkowski.
\newblock {\em Star-shaped regions of stability in hereditary systems}.
\newblock PhD thesis, Brown University, 1976.

\bibitem{Slemrod1971Nonexistence}
M.~Slemrod.
\newblock Nonexistence of oscillations in a nonlinear distributed network.
\newblock {\em J. Math. Anal. Appl.}, 36:22--40, 1971.

\bibitem{Verriest2012State}
E.~I. Verriest.
\newblock State space for time varying delay.
\newblock In {\em Time delay systems: methods, applications and new trends},
  volume 423 of {\em Lect. Notes Control Inf. Sci.}, pages 135--146. Springer,
  Berlin, 2012.

\bibitem{Xu2021Exponential}
Z.~Xu, X.~Li, and V.~Stojanovic.
\newblock Exponential stability of nonlinear state-dependent delayed impulsive
  systems with applications.
\newblock {\em Nonlinear Anal. Hybrid Syst.}, 42:Paper No. 101088, 12, 2021.

\end{thebibliography}

\end{document}